%% file: main.tex
\documentclass[10pt,reqno]{amsart}
\usepackage[foot]{amsaddr}
\usepackage{setup}
\usepackage{microtype}
\usepackage[normalem]{ulem}

\usepackage{float}
\usepackage{wrapfig}
\usepackage{comment}
\usetikzlibrary{patterns}
\usepackage[numbers,sort&compress]{natbib}
\input{cmds}

\renewcommand{\mathbb}{\mathbf}

\newcommand{\hgt}{\mathrm{ht}}

\renewcommand{\emptyset}{\es}
\newcommand{\cT}{\mathcal{T}}
\newcommand{\eps}{\varepsilon}

\title{Leaf Stripping on Uniform Attachment Trees}
\author{Louigi Addario-Berry$^*$}
\address{$^*$Department of Mathematics and Statistics, McGill University
\texttt{louigi.addario@mcgill.ca}}

\author{Anna Brandenberger$^\circ$}
\address{$^\circ$Department of Mathematics, MIT \texttt{abrande@mit.edu}}

\author{Simon Briend$^\dagger$} 
\address{$^\dagger$Department of Mathematics, Unidistance Suisse \texttt{simon.briend@unidistance.ch}
}

\author{Nicolas Broutin$^\ddagger$} 
\address{$^\ddagger$Laboratoire de Probabilit\'es, Statistique et Mod\'elisation Sorbonne Universit\'e \texttt{nicolas.broutin@lpsm.paris}
}

\author{G\'abor Lugosi$^\mathsection$} 
\address{$^\mathsection$Department of Economics and Business,
Pompeu  Fabra University, Barcelona, Spain; 
ICREA, Pg. Lluís Companys 23, 08010 Barcelona, Spain; Barcelona School of Economics
 \texttt{gabor.lugosi@gmail.com}}

\begin{document}

\begin{abstract}
In this note we analyze the performance of a simple root-finding algorithm in uniform attachment trees. The \emph{leaf-stripping} algorithm recursively removes all leaves of the tree for a carefully chosen number of rounds. We show that, with probability $1-\e$, the set of remaining vertices 
contains the root and has 
 a size only depending on $\e$ but not on the size of the tree. 
\end{abstract}

\maketitle 

\section{Introduction} \label{sec:intro}

The problem of localizing the root vertex in various random tree models goes back to Haigh \cite{haigh1970recovery} and recently has been intensively studied; see \cite{SZ11,bubeck2017finding,lugosi2019finding,khim2016confidence,reddad2019discovery,jog2016analysis,brandenberger2022root,crane2021inference,crane2023root}.

In this paper we are interested in root-finding algorithms for random recursive trees (also known as uniform attachment trees). A random recursive tree $T_n$ of size $n$ is defined as follows. 
$T_1$ consists of a single vertex with label 1. Then, we inductively obtain $T_{i+1}$ from $T_i$ by connecting a new vertex labelled $i+1$ to an independently and uniformly chosen vertex $v \in [n]$. An equivalent way of viewing this tree is as uniformly chosen \textit{increasing tree} of size $n$: these are rooted trees with vertex set $[n]$ whose labels increase along any path from the root to a leaf (see, e.g., \cite{BeFlSa1992} or \cite[Example II.18]{FlSe2009}).

Write $\cT_n$ for the set of all unrooted trees with vertex set $[n]$. A {\em root-finding algorithm} for trees of size $n$ is a function $A:\cT_n \to 2^{[n]}$ which is invariant to relabeling: if $T,T' \in \cT_n$ and $T'$ can be obtained from $T$ by a permutation of vertex labels, then $A(T)=A(T')$. Its {\em error probability} is $\P\{1 \not \in A(T_n)\}$, where we abuse notation and write $A(T_n)$ to mean $A(T)$ where $T$ is the unrooted tree corresponding to $T_n$.
The {\em size} of $A$ is $\max(|A(T)|:T \in \cT_n)$. 
The set 
$A(T)$ is often called a \emph{confidence set} for the root vertex.

The paper \cite{bubeck2017finding} showed that there exist root finding algorithms with low error probability whose size is independent of $n$; more precisely, there exists $c > 0$ such that for all $\eps \in (0,1/e)$, for all $n \ge 1$, there exists a root-finding algorithm $A$ for trees of size $n$ with error at most $\eps$ and size at most $\exp(c\log(1/\eps)/\log\log(1/\eps))$. This result is proved via a rather delicate analysis of the {\em maximum likelihood} root-finding algorithm, which first lists nodes in decreasing order of their likelihood of being the root (given the shape of the tree but not its label information), then returns the $k$ most likely nodes, for a suitable value of $k$.
In \cite{bubeck2017finding} a simpler method was also analyzed that ranks vertices according to their \emph{Jordan centrality}, defined by the size of the largest component of the forest obtained by removing the vertex from the tree. 
This method is guaranteed to output a confidence set whose size is at most $(11/\e)\log(1/\e)$ and contains the root vertex with probability $1-\e$.

The purpose of this paper is to show that a different algorithm which we term {\em leaf-stripping}, which is very simple to describe and easy to implement, also yields root finding algorithms with small error probability and with size independent of $n$. The algorithm proceeds as follows. For $n \in \N$ define
\begin{equation}\label{eq:m_n-definition}
    m_n= \Big\lceil e\log n - \frac{3}{2}\log\log (n+1) \Big\rceil\, ;
\end{equation}
then for $k \in \N$, define the leaf-stripping algorithm $R_k(T_n)$ as follows. 
\begin{alg}
Let $T^{(0)}=T_n$. 
For $1 \le i \le m_n-k$, let $T^{(i)}$ be the tree obtained from $T^{(i-1)}$ by removing all its leaves. Then let $R_k(T_n)$ be the vertex set of $T^{(m_n-k)}$.
\end{alg}
We have defined the algorithm only for $T_n$ since that is the only tree we apply it to. Note that the algorithm succeeds if and only if there are at least two edge-disjoint paths of length at least $m_n-k$ starting from vertex $1$ in $T_n$. Equivalently, $1 \in R_k(T_n)$ if and only if $1$ has at least two children $u,v$ in $T_n$ such that the subtrees rooted at $u$ and $v$ have height at least $m_n-k-1$.
This observation allows us to explain the role of the constant $m_n$; it is known from \cite{addario2013poisson} that the expected height of $T_n$ of size $n$ is $m_n+O(1)$, and that the height is exponentially concentrated around $m_n$. Therefore, stripping leaves $m_n$ times will erase most or all of the tree. We instead strip $m_n - k$ times for a suitably chosen $k=k_\e$, to ensure that on one hand, the resulting set is not too large, and on the other hand, the root is not likely to be removed. Our main results state that leaf-stripping yields a family of root-finding algorithms for which the size and the error probability are polynomially related. 

\begin{thm}\label{thm:main}
    There exist $c, \gamma > 0$ and $c' > 0$ such that for all $\e \in (0,1)$, setting $k = \lceil c \log(2/\e)\rceil$, for all $n$ sufficiently large, $R_{k}(T_n)$ satisfies
    \begin{equation*}
        \P \bcurly { 1 \in R_k(T_n) \text{ and } |R_k(T_n)| \leq \e^{-\gamma} } \geq 1 - \e \,.
    \end{equation*}
\end{thm}

Theorem \ref{thm:main} shows that the performance of the leaf-stripping algorithm does not deteriorate as the size $n$ of the tree grows and the size of the confidence set $R_k(T_n)$ may be bounded by a function of the probability of error only. Note however, that the dependence on the probability of error is inferior to that of the maximum likelihood mentioned above. 
In particular, while ranking vertices by their likelihood of being the root produces a confidence set that is smaller than any power of $1/\e$, 
Theorem \ref{thm:main} only implies a bound that is polynomial in $1/\e$.
Theorem~\ref{thm:main2} below implies that this polynomial dependence is necessary (though we do not claim optimality of the obtained exponents). 
\begin{thm}\label{thm:main2}
    There exists $\gamma' >0$ such that for all $\e > 0$ sufficiently small, for all $n$ sufficiently large, for all $k \in \N$ for which $\P\{ 1 \in R_k(T_n)\} \ge 1-\e$, we also have $\P\{ |R_k(T_n)| \geq \e^{-\gamma'}\} \ge 1-\e$.
\end{thm}

\noindent {\bf Remarks.}
\begin{itemize}
\item In Theorem~\ref{thm:main} the size $|R_k(T_n)|$ is not deterministically bounded, but we may easily modify the leaf-stripping algorithm to obtain an algorithm for which both the size and error probability are bounded. Specifically, consider the algorithm $R_k'$ defined by setting $R_k'(T_n)=R_k(T_n)$ if $R_k(T_n) \le \eps^{-\gamma}$ and setting $R_k'(T_n)=\emptyset$ otherwise. Then $R_k'$ has size at most $\eps^{-\gamma}$ and error probability less than $\eps$.
\item 
Sreedharan, Magner, Grama, and Szpankowski \cite{sreedharan2019inferring}
address the problem of (partially) recovering the labeling
of a Barabási-Albert preferential attachment graph. The first step of the procedure analyzed in \cite{sreedharan2019inferring} 
is a ``peeling'' algorithm whose aim is to recover the root. When the graph is a tree, peeling is equivalent to the leaf-stripping method analyzed here. In \cite{sreedharan2019inferring} (Lemma 6.1 in the supplementary material) it is claimed that peeling deterministically finds the root vertex, which is not the case. In this note we offer an analysis of the peeling algorithm for uniform attachment trees. The analogous problem for preferential attachment trees -- and more generally for preferential attachment graphs -- remains a challenge for future research. 
\item Navlakha and Kingsford \cite{navlakha} proposed a general framework for recovering the past 
 of growing networks from a present-day snapshot. 
 As a general principle, they propose \emph{greedy likelihood} for root reconstruction. Greedy likelihood, at each step, deletes the vertex which is most likely to have been added most recently, breaking ties uniformly at random. In a uniform attachment tree, greedy likelihood deletes a uniformly random leaf at each step. 
 This method is closely related to but not equivalent to leaf stripping. We conjecture that 
 greedy likelihood has similar performance to that of leaf stripping, though we do not see an easy way to extend the analysis of this paper to an analysis of the greedy likelihood algorithm.
\end{itemize}

\subsection{Notation} Given a rooted tree $t$ and a node $v \in t$, we denote its distance from the root by $\hgt(v)$. We further denote the height of $t$ by $\hgt(t) \ceq \max_{v \in t} \hgt(v)$; the subtree of $t$ rooted at $v$ by $t^{v,\downarrow}$
(defined as the subtree containing all vertices $u\in t$ for which the path between $u$ and the root contains $v$); and the parent of $v$ by $\mathrm{par}(v)$. 
Some more general notation includes: for any set $A$, let $A^0 = \emptyset$. The set of positive integers is denoted by $\N$, and we write $[n] = \{1, \dots, n\}$ for any $n \in \N$. We omit ceilings and floors for legibility whenever possible.

\section{Preliminaries}
\subsection{Embedding in the Ulam--Harris tree}
We introduce another representation for increasing trees that relies on the {\em Ulam--Harris tree} $\U$ \cite{Neveu1986}. 
This is the infinite ordered rooted tree with node set $\bigcup_{j=0}^\infty \N^j$ (root $\es$ and nodes on level $j$ of the form $n_1 \cdots n_j$) and edges from $n_1\cdots n_j$ to $n_1\cdots n_j n$ for all $j \geq 0$ and $n \in \N$; see Figure~\ref{fig:ulam-alone}.

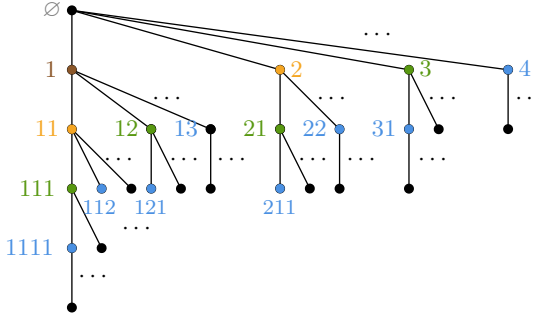
\begin{SCfigure}[.6][hbtp]
    \centering
    \input{ulam-tree-alone}
    \caption{The Ulam--Harris tree $\U$, with zones 1 through 4 in different colours.}
    \label{fig:ulam-alone}
\end{SCfigure}

An ordered rooted tree can be canonically embedded in $\U$. Similarly, an increasing tree can be embedded in $\U$ by ordering vertices according to their labels: given an increasing tree $t$ on $[n]$, we define a map $\varphi = \varphi_t : [n] \to \U$ by 
\begin{subequations}\label{eq:phi-conditions}
\begin{align}
    &\varphi(1) = \es, \label{eq:phi-root} \\
    &\varphi(v) = n_1 \cdots n_jn_{j+1} \implies \varphi^{-1}(n_1\cdots n_{j}) = \mathrm{par}(v) < v, \label{eq:phi-parent}\\ 
    &\varphi(v) = n_1 \cdots n_j n,\ n > 1 \implies \varphi^{-1}(n_1 \cdots n_j (n-1)) < v.\label{eq:phi-sibling}
\end{align} 
\end{subequations}
Condition \eqref{eq:phi-root} ensures that the root of $t$ is mapped to the root of $\U$; \eqref{eq:phi-parent} ensures that parent-child relations in $t$ are preserved in $\U$, which in particular implies that $\hgt(\varphi(v)) = \hgt(v)$ for all $v \in t$; and \eqref{eq:phi-sibling} ensures that for any node in $\U$, its left sibling is a vertex in $t$ with a smaller label.
Then $\{\varphi(i)\}_{i \in [n]}$ is a size-$n$ subtree of $\U$, and clearly $t$ can be recovered from the map $\varphi$; we sometimes refer to $t$ as the increasing tree encoded by $\varphi$. See Figure~\ref{fig:phi-embedding} for an example increasing tree and its embedding into $\U$.

We define the \textit{zone} of a node $u = n_1\cdots n_j \in \U$, $u \neq \es$, to be $z(u) = n_1 + \dots + n_j$. A node being in zone $z+1$ corresponds to it being either the first child of a node of zone $z$ or the next sibling of a node in zone $z$. In particular, the number of nodes of $\U$ in zone $z$ is $|\{u \in \U: z(u) = z\}| = 2^{z-1}$. See Figure~\ref{fig:ulam-alone} for an illustration of zones 1 -- 4 of the Ulam--Harris tree.

We now embed the random recursive tree $T_n$ on $[n]$ into $\U$ via $\varphi_{T_n}$. This embedding helps us analyze the sizes of subtrees. Indeed, note that due to the P\'olya urn structure of the subtree sizes of $T_n$, we have that for vertices $v \in [n]$ in zone $z(\varphi_{T_n}(v)) = z$, 
\begin{equation}\label{eq:subtree-size}
    |T_n^{v,\downarrow}| \overset{d}{=} \lfloor \cdots \lfloor \lfloor n U_1 \rfloor U_2 \rfloor \cdots U_z \rfloor,
\end{equation}
where $U_1, \dots, U_z$ are i.i.d.\ Uniform$[0,1]$ random variables. Indeed, for zone 1, the unique vertex with $z(\varphi_{T_n}(v)) = 1$ is $v = 2$. Since the two sets of vertices that respectively connect to vertex 1 and vertex 2 have sizes distributed according to a P\'olya urn with two colours, we have $|T_n^{2,\downarrow}| \stackrel{d}{=} \lfloor n U \rfloor$ (see for instance \cite[p.~177]{JoKo1977}). The general result follows straightforwardly by induction using the fact that $1-U \stackrel{d}{=} U$. 

In particular, \eqref{eq:subtree-size} implies that the subtree sizes of two nodes in the same zone have the same distribution. 

\begin{figure}[hbtp]
    \centering
    \input{ulam-embedding}
    \caption{An instance of an increasing tree on $\{1, \dots, 14\}$ and its embedding via $\varphi$ into the Ulam--Harris tree.}
    \label{fig:phi-embedding}
\end{figure}
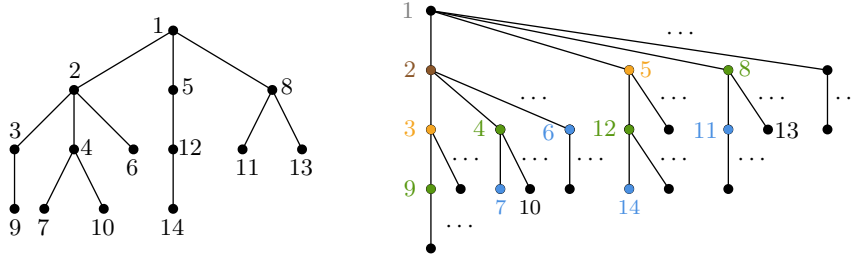

\subsection{Tail bounds on the height of random recursive trees}
In order to prove our main result, we use exponential tail bounds on $\hgt(T_n)$, see e.g.,~\cite{addario2013poisson} or \cite[Theorem 6.32]{drmota2009random}, which we state as follows: there exist $\alpha > 0$ and $0 < \alpha' < 1/(2e)$ such that for all $n \in \N$ sufficiently large and for all $k \in \N$, 
\begin{equation}\label{eq:height-tails}
    \P\{ | \hgt(T_n) - m_n | \geq k \} \leq \alpha e^{-\alpha' k}.
\end{equation}

\section{Proof of Theorem 1}

Recall from \eqref{eq:m_n-definition} that $m_n = \lceil e \log n - \frac{3}{2}\log\log n\rceil$. Recall that $R_k(T_n)$ is the vertex set of the tree obtained from $T_n$ after performing $m_n-k$ rounds of leaf stripping. For simplicity, we write $R_k \equiv R_k(T_n)$. 

\subsection{Detection of the root}

First, we need to make sure that leaf stripping will ``detect" vertex $1$, that is, we show that $\P\{ 1\in R_k \} \geq 1 - \e$.
\begin{lem} \label{lem:detection}
There exist $c,c' > 0$ such that such that for all $\e > 0$, for all $n \in \N$, if $k \geq c\log(c'/\e)$, then $\P\{1 \not\in R_k\} \leq \e$. 
\end{lem}
\begin{proof}  
Recall that $T_n^{2,\downarrow}$ is the subtree of vertex $2$, and let $T^{2,\uparrow}_{n}$ be the tree $T_n$ with $T^{2,\downarrow}_{n}$ removed, i.e., the subtree of vertex $1$ excluding vertex $2$ and its descendants. It is straightforward to see that, conditionally given their sizes, both $T^{2,\uparrow}_n$ and $T^{2,\downarrow}_n$ are distributed as random recursive trees and are (conditionally) independent. Then, we have the following lower and upper bounds:
\begin{equation}\label{eq:lb}
    \P\{1 \not\in R_k\} \geq \P\{\hgt(T^{2,\uparrow}_{n}) < m_n - k\} 
\end{equation} 
and 
\begin{equation}\label{eq:Ek-upper-bd}
\begin{aligned}
    \P\{1 \not\in R_k\} &\leq \P\{ \hgt(T^{2,\uparrow}_{n}) < m_n - k \text{ or } \hgt(T^{2,\downarrow}_{n}) < m_n - k \} \\
    &\leq 2 \P\{\hgt(T^{2,\uparrow}_{n}) < m_n - k\},
\end{aligned}
\end{equation}
where the second inequality comes from a union bound, and 
the fact that $T^{2,\uparrow}_{n}$ and $T^{2,\downarrow}_{n}$ are identically distributed when seen as rooted unlabelled trees.
Therefore
\begin{equation}
\frac{\P\{1\not\in R_k\}}{\P\{\hgt(T^{2,\uparrow}_{n}) < m_n - k\}}\in [1,2]. 
\end{equation}
From this we see that controlling \eqref{eq:Ek-upper-bd}, that is, controlling $\P\{\hgt(T^{2,\uparrow}_{n}) < m_n - k\}$, is sufficient to prove the lemma. 
From \eqref{eq:subtree-size} we have that $|T^{2,\uparrow}_{n}| = n - |T^{2,\downarrow}_{n}| \stackrel{d}{=} \lfloor n U \rfloor  \stackrel{d}{=} \mathrm{Unif}(1,\dots, n-1)$ and  since $T^{2,\uparrow}_{n}$
is distributed as a random recursive tree conditioned on its size,
\begin{align*}
    \P\{\hgt(T^{2,\uparrow}_{n}) < m_n - k\} &= \sum_{j=1}^{n-1}\frac{1}{n-1} \P\{\hgt(T_j) < m_n - k \} \\ 
    &\leq 1/e^{k/3} + \frac{1}{n-1}\sum_{j=n/e^{k/3}}^n \P\{ \hgt(T_j) \leq m_n - k \}.
\end{align*}
For $j \geq n/e^{k/3}$ we have $|m_n - m_j - e \log(n/j)| \le 1 +o(1)$; so, using the upper tail bound from \eqref{eq:height-tails}, we obtain that
\begin{equation}
    \P\{ \hgt(T_j) \leq m_n - k \} \leq \P\{ \hgt(T_j) \leq m_j + 2 - k(1-e/3) \} 
    \leq \alpha e^{-\alpha' k}
\end{equation}  
for some $\alpha, \alpha' > 0$. 
Thus, 
\begin{equation*}
    \P\{\hgt(T^{2,\uparrow}_{n}) < m_n - k\} 
    \leq e^{-k/3} + \alpha e^{-\alpha' k}
\end{equation*}
and therefore
\[ 
\P\{1 \not\in R_k\} \leq c' e^{-\alpha'k}
\]
for some $0 < c' \leq 4\alpha$. The result follows by taking $k = c \log(c'/\e)$, where $c = 1/\alpha'$.
\end{proof}

\subsection{Size of the confidence set}

\begin{figure}[hbtp]
    \centering
    \input{tree-after-bijection}
    \caption{Image of the tree on $\{1,\dots,14\}$ from Figure~\ref{fig:phi-embedding} under the tree flipping involution $\ell^{-1} \circ f_2 \circ \ell$, i.e., flipping up to zone $z = 2$.}
    \label{fig:tree-after-bijection}
\end{figure}
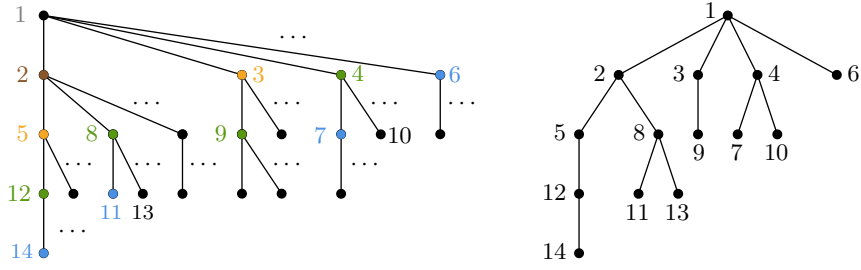

We must now prove that $|R_k|$ is small, with high probability. 

\begin{lem}\label{lem:size}
    For any $c,c'$ as in Lemma~\ref{lem:detection}, there exists $\gamma > 0$ such that for all $\e > 0$, for all $n \in \N$, if $k = c \log(c'/\e)$, we have $\P\{ |R_k| > \e^{-\gamma} \} \leq \e$. 
\end{lem}
\begin{proof}
We use the embedding of $T_n$ into $\U$ by $\varphi = \varphi_{T_n}$; recall from \eqref{eq:subtree-size} that under this embedding, if for $v \in T_n$, $z(\varphi_{T_n}(v)) = z$, then $|T_n^{v,\downarrow}| \stackrel{d}{=} \lfloor \cdots \lfloor nU_1 \rfloor \cdots U_z \rfloor$.

Let $S=\{v \in [n]: z(\varphi_{T_n}(v)) = 4k,\  \hgt(T_n^{v_, \downarrow})\geq m_n-k\}$ 
be the set of vertices of $T_n$ in zone $4k$ with subtrees of height $\geq m_n - k$. Note that if $S$ is empty, then $R_k$ only includes vertices in zones $1, \dots, 4k-1$, so 
\begin{equation}\label{eq:Rk-bound}
    |R_k| = \sum_{z=1}^{4k-1} |\{v \in [n]: z(\varphi_{T_n}(v)) = z\}| 
    \leq 1 + \sum_{z=1}^{4k-1} 2^{z-1} = 2^{4k-1} ,
\end{equation} 
and $k = c \log (c'/\eps)$ suffices for $|R_k| = \e^{-\gamma}$, where $\gamma \geq 4c \log 2 > 8e \log 2$. 

On the other hand, if $S$ is non-empty, then since the height of $T_n$ is at least $\hgt(v)+m_n-k$ for all $v\in S$, we have
\begin{equation}\label{eq:rephrase-pf-condition}
    \P\left\{ \hgt(T_n)\geq m_n+k \mid S\neq \es \right\}\geq \P\left\{ \exists v \in S: \ \hgt(\varphi_{T_n}(v)) \geq 2k \mid S\neq \es \right\}.     
\end{equation}
The heart of the proof then lies in the following claim:
\begin{equation}\label{eq:bijection}
    \P\left\{ \exists v \in S: \ \hgt(\varphi_{T_n}(v)) \geq 2k \mid S\neq \es \right\}\geq \frac{1}{2}.
\end{equation}
We complete the proof that $|R_k|$ is likely to be small before proving \eqref{eq:bijection}. Using \eqref{eq:bijection}, we have the following:
\begin{align*}
    \P\{ |R_k| \geq 2^{4k} \} 
    &\leq \P\{S \neq \es\} \\
    &\leq 2 \P \{ \hgt(T_n) \geq m_n + k \mid S\neq \es\} \cdot \P\{ S \neq \es\} \\ 
    &\leq 2 \P\{ \hgt(T_n) \geq m_n + k\} \\
    &\leq c' e^{- k/c},
\end{align*}
where the first inequality is due to \eqref{eq:Rk-bound}, the second is by multiplying by $1 \leq 2\P\{ \hgt(T_n) \geq m_n + k \mid S\neq \es\}$ using \eqref{eq:rephrase-pf-condition} and \eqref{eq:bijection}, and the fourth is due to the lower tail bound from \eqref{eq:height-tails}, with $c$ and $c'$ as in Lemma~\ref{lem:detection}. 
Therefore, assuming \eqref{eq:bijection} we have that $\P\{ |R_k| > \e^{-\gamma} \} \leq c'e^{-k/c} = \e$, which proves the lemma.
\end{proof}

In order to prove \eqref{eq:bijection}, we define an involution on increasing trees which, intuitively, flips vertices in zone $4k$ at large heights $\hgt(\varphi_{T_n}(v))$ to lower heights, while preserving their subtree structure. See Figure~\ref{fig:tree-after-bijection} for an illustration of the involution applied to the tree from Figure~\ref{fig:phi-embedding} (where for clarity we illustrate it flipping vertices in zone 2 rather than in a zone $4k$ for $k \in \N$); note how the heights of vertices 3 and 5 are swapped, while their subtrees are preserved. We formalize the involution below.

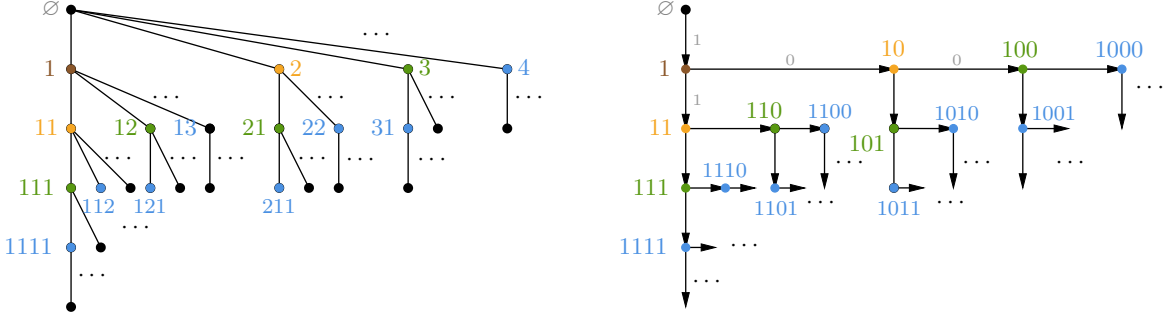
\begin{figure}[hbtp]
    \centering
    \input{ulam-child-sibling}
    \caption{The Ulam--Harris tree (left) and its corresponding child-sibling binary tree (right) given by the bijection $\ell$ in \eqref{eq:child-sibling}, both with zones 1 through 4 illustrated in different colours.}
    \label{fig:ulam-child-sibling}
\end{figure}

We first define the child-sibling bijection $\ell : \U \mapsto \bigcup_{k = 0}^\infty \{0,1\}^k$ by $\ell(\es) \ceq \es$ and, for $n_1, \dots, n_j \in \N^j$,
\begin{equation}\label{eq:child-sibling}
    \ell( n_1\cdots n_j ) = 10^{n_1-1} 10^{n_2-1} \cdots 1 0^{n_j-1}.
\end{equation}
This corresponds to the classic bijection between an ordered rooted tree and its child-sibling binary tree representation: $\ell$ relabels each node according to its identifying path in the binary tree. For instance, $\ell(1) = 1$ and $\ell(12) = 110$; see Figure~\ref{fig:ulam-child-sibling}. Note that under the image of $\ell$, all non-root nodes in $\U$ have first coordinate in the child-sibling bijection equal to 1. Also, note that the zone $z(v)$ of any $v \in \U \setminus \{ \es \}$ is the length of $\ell(v)$. 

For each $j \geq 2$, we now define the tree flipping involution which ``exchanges the roles of child and sibling for all nodes in the first $j$ zones in the child-sibling encoding, and preserves the subtrees of nodes in zone $j$''. Formally, define $f_j: \bigcup_{k = 0}^\infty \{0,1\}^k \to \bigcup_{k = 0}^\infty \{0,1\}^k$ by $f_j(\es) = \es$, $f_j(1) = 1$, and 
\begin{equation}
    f_j( i_1 i_2\cdots i_{j-1} i_j i_{j+1} \cdots i_k ) = i_1 \overline{i_2 \cdots i_{j-1}i_j} i_{j+1} \cdots i_k,
\end{equation}
where $\overline{i} = 1-i$ for $i \in \{0,1\}$.
See Figure~\ref{fig:flipping-bijection} for an illustration for $j = 3$.

\begin{figure}[hbtp]
    \centering
    \input{ulam-bijection}
    \caption{Illustration of $f_3$: the second and third coordinates are flipped for every node. Notice that for each node in zone 3, its subtree remains identical under the map.}
    \label{fig:flipping-bijection}
\end{figure}
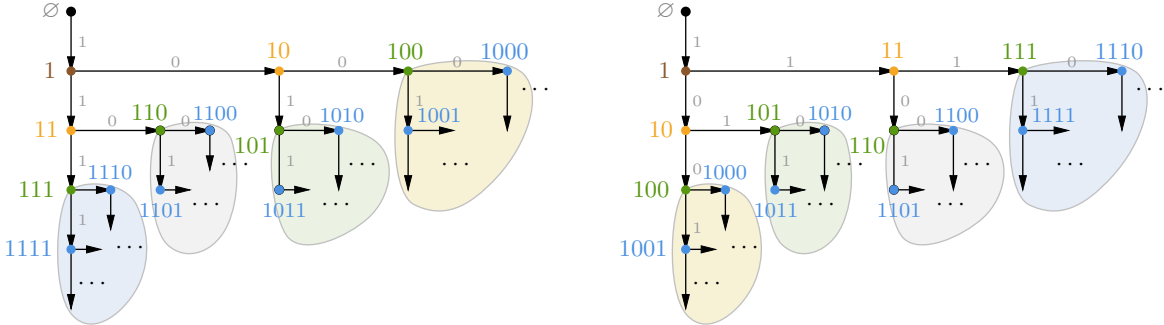

Then, the final involution $b : \U \to \U$ that we consider is 
\begin{equation}
    b \ceq \ell^{-1} \circ f_{4k} \circ \ell,
\end{equation} 
and, given an increasing tree $t$ on $[n]$, we define a function $\overline{\varphi} = \overline{\varphi}_t: [n] \to \U$ by 
\begin{equation}
    \overline{\varphi}_t(v) = b(\varphi_t(v)).
\end{equation}
The following lemma ensures that we indeed obtain a involution between increasing trees. See Figure~\ref{fig:tree-after-bijection} for an illustration of the involution applied to the tree from Figure~\ref{fig:phi-embedding}, with $f_2$ rather than $f_{4k}$ for clarity.
\begin{lem}\label{lem:increasing}
    For $t$ an increasing tree on $[n]$, the map $\overline{\varphi}_t$ encodes an increasing tree on $[n]$.
\end{lem}
\begin{proof}
    We check the conditions \eqref{eq:phi-root}, \eqref{eq:phi-parent} and \eqref{eq:phi-sibling} for $\overline \varphi = \overline{\varphi}_t$. 
    Since $b(\es) = \es$, \eqref{eq:phi-root} is satisfied, and since $b(u) \neq \es$ for $u \neq \es$, $1 < b(\varphi(v))$ for all $1 \neq v \in [n]$. Then, for \eqref{eq:phi-parent} and \eqref{eq:phi-sibling}, it suffices to show that, for any $w \in [n-1]$ with $\overline{\varphi}(w) = n_1 \cdots n_j$ for some $j \geq 1$, for any $v,v'\in [n]$, 
    \begin{enumerate}
        \item if $\overline{\varphi}(v) = n_1 \cdots n_j 1$ then $w < v$;
        \item if $\overline{\varphi}(v') = n_1 \cdots n_{j-1}(n_j+1)$ then $w < v'$.
    \end{enumerate}
    Note that by definition, $\ell(\overline{\varphi}(u)) = f_{4k} (\ell(\varphi(u)))$ for any $u \in [n]$.  
    First, suppose $n_1 + \dots + n_j + 1 \leq 4k$. If $v,v'\in[n]$ satisfy \eqref{eq:phi-parent} and \eqref{eq:phi-sibling}, then since 
    \[
    \ell(\overline{\varphi}(w)) = 1 0^{n_1-1}\cdots 10^{n_j-1} ~,~ \ell(\overline{\varphi}(v)) = 1 0^{n_1-1}\cdots 10^{n_j-1} 1 ~\text{ and }~ \ell(\overline{\varphi}(v')) = 1 0^{n_1-1}\cdots 10^{n_j-1}0 ,
    \]
    we have 
    \[
    \begin{aligned}
    \ell({\varphi}(w)) &= 1 1^{n_1-1} 0 1^{n_2-1}\cdots 01^{n_j-1} ,~ \\ 
    \ell({\varphi}(v)) &= 1 1^{n_1-1}0 1^{n_2-1}\cdots 01^{n_j-1} 0 = \ell(\varphi(w)) 0 \text{, and }~ \\
    \ell({\varphi}(v')) &= 1 1^{n_1-1}0 1^{n_2-1}\cdots 01^{n_j-1}1 = \ell(\varphi(w)) 1.
    \end{aligned}
    \]
    That is, $\varphi(v)$ and $\varphi(v')$ are respectively the right sibling and the first child of $\varphi(w)$ in $\U$. Therefore, by \eqref{eq:phi-sibling} and \eqref{eq:phi-parent} for $\varphi$, we have $w < v$ and $w < v'$, as required. 

    Now, if $n_1 + \dots + n_j + 1 > 4k$, by expanding analogously, we have $\ell(\varphi(v)) = \ell(\varphi(w) 1$ and $\ell(\varphi(v')) = \ell(\varphi(w)) 0$. Then, $\varphi(v)$ and $\varphi(v')$ are respectively the first child and the right sibling of $\varphi(w)$ in $\U$, so $w < v,v'$.
\end{proof}

We can now apply this function to the random recursive tree $T_n$. We denote by $\overline{T}_n$ the increasing tree encoded by $\overline{\varphi} \ceq \overline{\varphi}_{T_n}$. 
Let $\overline{S} = \{ v \in [n]: z(\overline{\varphi}(v)) = 4k,\ \hgt(\overline{T_n}^{v,\downarrow}) \geq m_n - k\}$ be the set of vertices in zone $4k$ that have tall subtrees after applying $\overline{\varphi}$. The involution $\overline{\varphi}$ has the following properties:

\begin{lem}\label{lem:properties}
    Let $T_n$ be a random recursive tree and write $\varphi = \varphi_{T_n}$. Then, the random tree $\overline{T_n}$ encoded by the map $\overline{\varphi} \ceq \overline{\varphi}_{T_n}$, and the set $\overline{S}$ satisfy the following properties: 
\begin{enumerate}
    \item For all $v \in [n]$, $z(\overline{\varphi}(v)) = z({\varphi}(v))$, i.e., the zone of all vertices stays the same,
    \item For all $v \in [n]$ with $z(\varphi(v)) = 4k$, $\overline{T_n}^{v,\downarrow} = T_n^{v,\downarrow}$, 
    and $S = \overline{S}$.
    \item For all $v \in [n]$ with $z(\varphi(v)) = z \in [4k]$, $\hgt(\varphi(v)) + \hgt(\overline{\varphi}(v)) = z + 1$. In particular, if $z(\varphi(v)) = 4k$ and $\hgt(\varphi(v)) \leq 2k+1$ then $\hgt(\overline{\varphi}(v)) \geq 2k$.
\end{enumerate}
\end{lem}

\begin{proof}
For (i), recall that for $v = n_1\cdots n_j \in \U$, $z(v) = n_1 + \dots + n_j$. Also, $\ell(v)$ is a string of length $z(v)$ for all $v\in \U$. The statement then holds since $f_{4k}$ does not change the length of the string $\ell(\varphi(v))$. 

For (ii), we argue similarly to the proof of Lemma~\ref{lem:increasing}. Let $v \in [n]$ such that $z(\varphi(v)) = 4k$.
For any $w \in T_n^{v,\downarrow}$, we can write $\varphi(w)=\varphi(v)n_1\cdots n_k$ for some $k \geq 1$. Then, $\ell(\varphi(w)) = \ell(\varphi(v)) 10^{n_1-1} \cdots 1 0^{n_j-1}$. 
Since $\ell(\varphi(v))$ has length $4k$, writing $\ell(\overline \varphi(v)) = f_{4k}(\ell(\varphi(v)))$, we have \[f_{4k}(\ell(\varphi(w))) = f_{4k}(\ell(\varphi(v))) 10^{n_1-1} \cdots 1 0^{n_j-1} = \ell(\overline \varphi(v)) 10^{n_1-1} \cdots 1 0^{n_j-1}.\] 
Therefore $\overline{\varphi}(w) = \overline{\varphi}(v) n_1 \cdots n_k$, i.e., $w$ is mapped to a node in the same relative position in $\U$ with respect to $v$. As this holds for all $w \in T_n^{v,\downarrow}$, we have $\overline{T_n}^{v,\downarrow} = T_n^{v,\downarrow}$. Combined with (i), this also implies that $S = \overline{S}$.

For (iii), notice that $\hgt(\varphi(v))$ is the number of 1's in $\ell(\varphi(v))$, which has total length $z$. 
Then, applying $f_{4k}$ and noting that the first coordinate stays fixed, we see that $\ell(\overline{\varphi}(v))$ has $1 + z-\hgt(\varphi(v)$ 1's. 
\end{proof}

Armed with $\overline{\varphi}$ and the properties established in Lemma~\ref{lem:properties}, we complete the proof. 

\begin{proof}[Proof of \eqref{eq:bijection}]
Since $b$ is an involution on the set of increasing trees on $[n]$, and due to property (ii) from Lemma~\ref{lem:properties}, we have that 
\vspace{-.4\bs}
\begin{equation}\label{eq:same-dist}
    (\varphi , S ) \stackrel{d}{=} (\overline{\varphi}, \overline{S}) = (\overline{\varphi}, {S});
\end{equation}
in particular, $\overline{\varphi}$ also encodes a random recursive tree. Furthermore, due to (iii), if $S \neq \es$, then either $S$ contains some $v \in [n]$ with $\hgt(\varphi(v)) \geq 2k$, or $\overline{S}$ contains some $v \in [n]$ with $\hgt(\overline \varphi(v)) \geq 2k$. That is, letting $E_S = \{ \exists v \in S : \hgt(\varphi_{T_n}(v)) \geq 2k \}$ and $\overline{E_S} = \{ \exists v \in \overline{S} : \hgt(\overline{\varphi}(v)) \geq 2k \}$, we have 
$( E_S \cup \overline{E_S} ) \cap \{ S \neq \es\} = \{ S \neq \es\}$. Therefore, we have
\begin{equation}\label{eq:ES-and-bar-ES}
    \P \{ E_S \cup \overline{E_S} , S \neq \es \} = \P\{ S \neq \es \}.
\end{equation}
Also, \eqref{eq:same-dist} implies that 
\begin{equation*}
    \P\{ E_S, S \neq \es \} = \P \{ \overline{E_S}, \overline{S} \neq \es\} =  \P \{ \overline{E_S}, S \neq \es \},
\end{equation*}
so \eqref{eq:ES-and-bar-ES} can be rewritten as 
\begin{equation*}
    2 \P\{ E_S, S \neq \es\} \geq \P\{ S \neq \es\},
\end{equation*}
proving \eqref{eq:bijection}.
\end{proof}
\section{Proof of Theorem 2}
\begin{proof} 
First, let $f(k): \N \to \R^+$ be any non decreasing function and consider $n' = \lfloor n/f(k) \rfloor$. For any $k \in \N$, for $n$ sufficiently large, we have $n'\geq n/(2f(k))$, so $\P\{| T_n^{2,\uparrow}|\leq n'\} \geq 1/(2f(k))$. Also, for fixed $k$ we have $|m_n - m_{n'} - e \log f(k)|\leq 1 + o(1)$ as $n \to \infty$, so from \eqref{eq:height-tails}, for $n$ large enough,
\begin{equation}
\label{eq:height-n'-ub}
    \P\{\hgt(T_{n'}) \geq m_n - k \} \leq \P\{\hgt(T_{n'})\geq m_{n'}+e\log f(k) - 2 - k\} \leq \alpha e^{-\alpha'(e\log f(k) - k)}
\end{equation}
and 
\begin{equation}
\label{eq:height-n'-lb}
    \P\{\hgt(T_{n'}) \geq m_n - k \} \geq \P\{\hgt(T_{n'})\geq m_{n'}+e\log f(k) + 2 - k\}.
\end{equation}
Recalling from \eqref{eq:lb} that $\P\{1 \not\in R_k\} \geq \P\{\hgt(T^{2,\uparrow}_{n}) < m_n - k\}$, it follows that for all $k \in \N$, for all $n$ sufficiently large we have 
\begin{align}
    \P\{1 \not\in R_k\}
    &\geq \P \{ \hgt(T_{n}^{2,\uparrow}) < m_n - k \mid |T_n^{2,\uparrow}| \leq n' \} \P\{ |T_n^{2,\uparrow}| \leq n'\} \nonumber \\
    &\geq \P \{ \hgt(T_{n'}) < m_n - k\} ({2f(k)})^{-1} \nonumber \\ 
    &\geq ({1-\alpha e^{-\alpha'(e\log f(k) - k)}})({2f(k)})^{-1} .\label{eq:P1inRk-lb}
\end{align}
Setting $f(k) \equiv 2$, \eqref{eq:P1inRk-lb} yields that for all $k > 0$,  
\[ \liminf_{n \to \infty} \P\{1 \not \in R_k(T_n)\} > 0.\]
Let 
\[
\eps_0 = \min\Big(\liminf_{n \to \infty}\P\{1 \not\in R_k(T_n)\}:1 \le k \le \log(2\alpha)/(\alpha'(e-1))\Big)\, ;
\]
we hereafter assume that $\eps \in (0,\eps_0)$. Since the theorem only makes an assertion about values $k$ for which $\P\{1 \not\in R_k(T_n)\} < \eps$, we can thus also assume that $k \ge \log(2\alpha)/(\alpha'(e-1)))$ from now on. Taking $f(k)\equiv e^{k}$, this then implies that the lower bound in \eqref{eq:P1inRk-lb} is at least $e^{-k}/4$.
Therefore, we must take $k \geq \log(1/(4\e))$ to ensure $\P\{ 1 \not\in R_k\} < \e$. 

We now prove the polynomial lower bound on $|R_k(T_n)|$ for any $k \geq \log(1/(4\e))$. 

Fix $\de > 0$. Set $K = \lceil(1+\de)^k/4 \rceil$ and $n' = \lfloor n/(1+\de)^k \rfloor \sim n/(4K)$, i.e., $f(k) = (1+\de)^k$. Let $T_K$ be the subtree of $T_n$ consisting of vertices $[K]$ and consider the $K$ subtrees $T^{K,1}_n, \dots, T^{K,K}_n$, where $T_n^{K,i}$ is the tree containing vertex $i$ in the forest obtained by removing edges between vertices in $[K]$. The vector of subtree sizes $(|T^{K,1}_n|, \dots, |T^{K,K}_n|)$ is distributed as a standard P\'olya urn with $n$ balls of $K$ colours.
Therefore, 
    \begin{equation*}
        \P\bcurly{ \# \bcurly{ i \in [K] : |T^{K,i}_n| > n' } < \frac{K}{4}  } 
        \le \binom{K}{K/4} (1/4)^{3K/4} 
        \le 2^{-K/4} 
        \le 2^{-(1+\de)^k/16} 
        \le \e/2,
    \end{equation*}
where the final inequality holds provided $\e > 0$ is sufficiently small, using that $k \geq \log(1/(4\e))$. 
Next, using that $\log f(k) = k \log(1+\de)$, we have from \eqref{eq:height-n'-lb} that
\[
\P\{ \hgt(T_{n'}) \ge m_n - k \} \ge \P \{ \hgt(T_{n'}) \ge m_{n'} + ek\log(1+\de) + 2 - k\} \geq 1/2
\]
for $\de > 0$ such that $\log(1+\de)<1/e$, and for $\e$ sufficiently small, since $\mathrm{med}(\hgt(T_{n'})) = m_{n'} + O(1)$. 

To combine these bounds, let $S = \{ i \in [K]: |T_n^{K,i}|>n'\}$. On the event $\{ |S| \geq K/4 \}$, since the events $(\{ \hgt(T_n^{K,i}) \ge m_n-k\})_{i \in [K]}$ are independent conditioned on the subtree sizes, we have that $\# \bcurly{ i \in [K] : \hgt(T_n^{K,i}) \ge m_n-k } \succeq_{\text st} \Bin(K/4,1/2)$. We have
\begin{align*}
\P &\left\{ \# \{ i \in [K]: \hgt(T_n^{k,i}) \ge m_n - k\} < \frac{K}{8} \right\} \\
&\le 
\P\left\{ \# \{ i \in [K]: \hgt(T_n^{k,i}) \ge m_n - k\} < \frac{K}{8} \cap \left(|S|\ge \frac{K}{4} \right)\right\} + \P\bcurly{|S| < \frac{K}{4}} \\ 
&\le \P\bcurly{ \Bin\left(\frac{K}{4},\frac{1}{2}\right)<\frac{K}{8} } + \e/2\\
&\le \exp(-K/32) + \e/2 \le \exp(-(1+\de)^k / 128) + \e/2 \le \e   
\end{align*}
for sufficiently small $\e > 0$, where we again used that $k \geq \log(1/(4\e))$. On the other hand, if $\# \{ i \in [K]: \hgt(T_n^{k,i}) \ge m_n - k\} \geq K/8$, there are at least $K/8$ vertices in $[K]$ with large subtree height, which are all then included in $R_k$, so $|R_k| \geq K/8 = c'' \e^{-\log(1+\de)}$ for some $c'' > 0$. This proves the theorem.
\end{proof}

\section*{Acknowledgements}
This research was initiated during the Seventeenth Annual Workshop on Probability and Combinatorics at McGill University's Bellairs Institute in Holetown, Barbados. We thank Bellairs Institute for its hospitality. 
LAB was supported by an NSERC Discovery Grant and by the CRC program.
AB was supported by NSF GRFP DGE-2141064 and Simons Investigator Award 622132 to Elchanan Mossel. 
GL acknowledges support from the Spanish MINECO grant PID2022-138268NB-100 and
  Ayudas Fundación BBVA a Proyectos de Investigaci\'on Cient\'ifica 2021.

\bibliographystyle{abbrv}
\bibliography{biblio}

\appendix

\end{document}

%% file: cmds.tex
\newcommand{\N}{\mathbb{N}}
\renewcommand{\P}{\mathbb{P}}

\newcommand{\R}{\mathbb{R}}

\usepackage{bbm}

\newcommand{\U}{\mathcal{U}}

\newcommand{\ceq}{\coloneqq} 

\newcommand{\e}{\varepsilon}
\newcommand{\de}{\delta}
\newcommand{\es}{\varnothing}

\newcommand{\bcurly}[1]{\left\{#1\right\}}

\newcommand{\bs}{\baselineskip}

\newcommand{\Bin}{\mathrm{Bin}}



%% file: ulam-tree-alone.tex
\begin{tikzpicture}[x=0.75pt,y=0.75pt,yscale=-1,xscale=1]

\draw    (227,87) -- (227,117) ;
\draw [shift={(227,117)}, rotate = 90] [color={rgb, 255:red, 0; green, 0; blue, 0 }  ][fill={rgb, 255:red, 0; green, 0; blue, 0 }  ][line width=0.75]      (0, 0) circle [x radius= 2.01, y radius= 2.01]   ;
\draw    (227,57) -- (242,87) ;
\draw [shift={(242,87)}, rotate = 63.43] [color={rgb, 255:red, 0; green, 0; blue, 0 }  ][fill={rgb, 255:red, 0; green, 0; blue, 0 }  ][line width=0.75]      (0, 0) circle [x radius= 2.01, y radius= 2.01]   ;
\draw    (277,57) -- (277,86.94) ;
\draw [shift={(277,86.94)}, rotate = 90] [color={rgb, 255:red, 0; green, 0; blue, 0 }  ][fill={rgb, 255:red, 0; green, 0; blue, 0 }  ][line width=0.75]      (0, 0) circle [x radius= 2.01, y radius= 2.01]   ;
\draw    (192,87) -- (192,117) ;
\draw [shift={(192,117)}, rotate = 90] [color={rgb, 255:red, 0; green, 0; blue, 0 }  ][fill={rgb, 255:red, 0; green, 0; blue, 0 }  ][line width=0.75]      (0, 0) circle [x radius= 2.01, y radius= 2.01]   ;
\draw    (162,57) -- (162,86.94) ;
\draw [shift={(162,86.94)}, rotate = 90] [color={rgb, 255:red, 0; green, 0; blue, 0 }  ][fill={rgb, 255:red, 0; green, 0; blue, 0 }  ][line width=0.75]      (0, 0) circle [x radius= 2.01, y radius= 2.01]   ;
\draw    (57,27) -- (162,57) ;
\draw [shift={(162,57)}, rotate = 15.95] [color={rgb, 255:red, 0; green, 0; blue, 0 }  ][fill={rgb, 255:red, 0; green, 0; blue, 0 }  ][line width=0.75]      (0, 0) circle [x radius= 2.01, y radius= 2.01]   ;
\draw    (57,27) -- (57,57) ;
\draw [shift={(57,57)}, rotate = 90] [color={rgb, 255:red, 0; green, 0; blue, 0 }  ][fill={rgb, 255:red, 0; green, 0; blue, 0 }  ][line width=0.75]      (0, 0) circle [x radius= 2.01, y radius= 2.01]   ;
\draw [shift={(57,27)}, rotate = 90] [color={rgb, 255:red, 0; green, 0; blue, 0 }  ][fill={rgb, 255:red, 0; green, 0; blue, 0 }  ][line width=0.75]      (0, 0) circle [x radius= 2.01, y radius= 2.01]   ;
\draw    (57,57) -- (97,87) ;
\draw [shift={(97,87)}, rotate = 36.87] [color={rgb, 255:red, 0; green, 0; blue, 0 }  ][fill={rgb, 255:red, 0; green, 0; blue, 0 }  ][line width=0.75]      (0, 0) circle [x radius= 2.01, y radius= 2.01]   ;
\draw    (57,57) -- (127,87) ;
\draw [shift={(127,87)}, rotate = 23.2] [color={rgb, 255:red, 0; green, 0; blue, 0 }  ][fill={rgb, 255:red, 0; green, 0; blue, 0 }  ][line width=0.75]      (0, 0) circle [x radius= 2.01, y radius= 2.01]   ;
\draw    (57,57) -- (57,87) ;
\draw [shift={(57,87)}, rotate = 90] [color={rgb, 255:red, 0; green, 0; blue, 0 }  ][fill={rgb, 255:red, 0; green, 0; blue, 0 }  ][line width=0.75]      (0, 0) circle [x radius= 2.01, y radius= 2.01]   ;
\draw    (57,87) -- (87,117) ;
\draw [shift={(87,117)}, rotate = 45] [color={rgb, 255:red, 0; green, 0; blue, 0 }  ][fill={rgb, 255:red, 0; green, 0; blue, 0 }  ][line width=0.75]      (0, 0) circle [x radius= 2.01, y radius= 2.01]   ;
\draw    (57,87) -- (72,117) ;
\draw [shift={(72,117)}, rotate = 63.43] [color={rgb, 255:red, 0; green, 0; blue, 0 }  ][fill={rgb, 255:red, 0; green, 0; blue, 0 }  ][line width=0.75]      (0, 0) circle [x radius= 2.01, y radius= 2.01]   ;
\draw    (57,87) -- (57,117) ;
\draw [shift={(57,117)}, rotate = 90] [color={rgb, 255:red, 0; green, 0; blue, 0 }  ][fill={rgb, 255:red, 0; green, 0; blue, 0 }  ][line width=0.75]      (0, 0) circle [x radius= 2.01, y radius= 2.01]   ;
\draw    (162,57) -- (192,87) ;
\draw [shift={(192,87)}, rotate = 45] [color={rgb, 255:red, 0; green, 0; blue, 0 }  ][fill={rgb, 255:red, 0; green, 0; blue, 0 }  ][line width=0.75]      (0, 0) circle [x radius= 2.01, y radius= 2.01]   ;
\draw    (97,87) -- (112,117) ;
\draw [shift={(112,117)}, rotate = 63.43] [color={rgb, 255:red, 0; green, 0; blue, 0 }  ][fill={rgb, 255:red, 0; green, 0; blue, 0 }  ][line width=0.75]      (0, 0) circle [x radius= 2.01, y radius= 2.01]   ;
\draw    (97,87) -- (97,117) ;
\draw [shift={(97,117)}, rotate = 90] [color={rgb, 255:red, 0; green, 0; blue, 0 }  ][fill={rgb, 255:red, 0; green, 0; blue, 0 }  ][line width=0.75]      (0, 0) circle [x radius= 2.01, y radius= 2.01]   ;
\draw [shift={(97,87)}, rotate = 90] [color={rgb, 255:red, 0; green, 0; blue, 0 }  ][fill={rgb, 255:red, 0; green, 0; blue, 0 }  ][line width=0.75]      (0, 0) circle [x radius= 2.01, y radius= 2.01]   ;
\draw    (127,87) -- (127,117) ;
\draw [shift={(127,117)}, rotate = 90] [color={rgb, 255:red, 0; green, 0; blue, 0 }  ][fill={rgb, 255:red, 0; green, 0; blue, 0 }  ][line width=0.75]      (0, 0) circle [x radius= 2.01, y radius= 2.01]   ;
\draw [shift={(127,87)}, rotate = 90] [color={rgb, 255:red, 0; green, 0; blue, 0 }  ][fill={rgb, 255:red, 0; green, 0; blue, 0 }  ][line width=0.75]      (0, 0) circle [x radius= 2.01, y radius= 2.01]   ;
\draw    (57,117) -- (72,147) ;
\draw [shift={(72,147)}, rotate = 63.43] [color={rgb, 255:red, 0; green, 0; blue, 0 }  ][fill={rgb, 255:red, 0; green, 0; blue, 0 }  ][line width=0.75]      (0, 0) circle [x radius= 2.01, y radius= 2.01]   ;
\draw    (57,117) -- (57,147) ;
\draw [shift={(57,147)}, rotate = 90] [color={rgb, 255:red, 0; green, 0; blue, 0 }  ][fill={rgb, 255:red, 0; green, 0; blue, 0 }  ][line width=0.75]      (0, 0) circle [x radius= 2.01, y radius= 2.01]   ;
\draw    (57,27) -- (227,57) ;
\draw [shift={(227,57)}, rotate = 10.01] [color={rgb, 255:red, 0; green, 0; blue, 0 }  ][fill={rgb, 255:red, 0; green, 0; blue, 0 }  ][line width=0.75]      (0, 0) circle [x radius= 2.01, y radius= 2.01]   ;
\draw    (162,86.94) -- (162,116.94) ;
\draw [shift={(162,116.94)}, rotate = 90] [color={rgb, 255:red, 0; green, 0; blue, 0 }  ][fill={rgb, 255:red, 0; green, 0; blue, 0 }  ][line width=0.75]      (0, 0) circle [x radius= 2.01, y radius= 2.01]   ;
\draw [shift={(162,86.94)}, rotate = 90] [color={rgb, 255:red, 0; green, 0; blue, 0 }  ][fill={rgb, 255:red, 0; green, 0; blue, 0 }  ][line width=0.75]      (0, 0) circle [x radius= 2.01, y radius= 2.01]   ;
\draw    (162,86.94) -- (177,116.94) ;
\draw [shift={(177,116.94)}, rotate = 63.43] [color={rgb, 255:red, 0; green, 0; blue, 0 }  ][fill={rgb, 255:red, 0; green, 0; blue, 0 }  ][line width=0.75]      (0, 0) circle [x radius= 2.01, y radius= 2.01]   ;
\draw    (227,57) -- (227,86.94) ;
\draw [shift={(227,86.94)}, rotate = 90] [color={rgb, 255:red, 0; green, 0; blue, 0 }  ][fill={rgb, 255:red, 0; green, 0; blue, 0 }  ][line width=0.75]      (0, 0) circle [x radius= 2.01, y radius= 2.01]   ;
\draw    (57,147) -- (57,177) ;
\draw [shift={(57,177)}, rotate = 90] [color={rgb, 255:red, 0; green, 0; blue, 0 }  ][fill={rgb, 255:red, 0; green, 0; blue, 0 }  ][line width=0.75]      (0, 0) circle [x radius= 2.01, y radius= 2.01]   ;
\draw  [color={rgb, 255:red, 139; green, 87; blue, 42 }  ,draw opacity=1 ][fill={rgb, 255:red, 139; green, 87; blue, 42 }  ,fill opacity=1 ] (55,57) .. controls (55,55.9) and (55.9,55) .. (57,55) .. controls (58.1,55) and (59,55.9) .. (59,57) .. controls (59,58.1) and (58.1,59) .. (57,59) .. controls (55.9,59) and (55,58.1) .. (55,57) -- cycle ;
\draw  [color={rgb, 255:red, 245; green, 166; blue, 35 }  ,draw opacity=1 ][fill={rgb, 255:red, 245; green, 166; blue, 35 }  ,fill opacity=1 ] (55,87) .. controls (55,85.9) and (55.9,85) .. (57,85) .. controls (58.1,85) and (59,85.9) .. (59,87) .. controls (59,88.1) and (58.1,89) .. (57,89) .. controls (55.9,89) and (55,88.1) .. (55,87) -- cycle ;
\draw  [color={rgb, 255:red, 245; green, 166; blue, 35 }  ,draw opacity=1 ][fill={rgb, 255:red, 245; green, 166; blue, 35 }  ,fill opacity=1 ] (160,57) .. controls (160,55.9) and (160.9,55) .. (162,55) .. controls (163.1,55) and (164,55.9) .. (164,57) .. controls (164,58.1) and (163.1,59) .. (162,59) .. controls (160.9,59) and (160,58.1) .. (160,57) -- cycle ;
\draw  [color={rgb, 255:red, 88; green, 152; blue, 17 }  ,draw opacity=1 ][fill={rgb, 255:red, 88; green, 152; blue, 17 }  ,fill opacity=1 ] (225,57) .. controls (225,55.9) and (225.9,55) .. (227,55) .. controls (228.1,55) and (229,55.9) .. (229,57) .. controls (229,58.1) and (228.1,59) .. (227,59) .. controls (225.9,59) and (225,58.1) .. (225,57) -- cycle ;
\draw  [color={rgb, 255:red, 88; green, 152; blue, 17 }  ,draw opacity=1 ][fill={rgb, 255:red, 88; green, 152; blue, 17 }  ,fill opacity=1 ] (160,86.94) .. controls (160,85.84) and (160.9,84.94) .. (162,84.94) .. controls (163.1,84.94) and (164,85.84) .. (164,86.94) .. controls (164,88.04) and (163.1,88.94) .. (162,88.94) .. controls (160.9,88.94) and (160,88.04) .. (160,86.94) -- cycle ;
\draw  [color={rgb, 255:red, 88; green, 152; blue, 17 }  ,draw opacity=1 ][fill={rgb, 255:red, 88; green, 152; blue, 17 }  ,fill opacity=1 ] (95,87) .. controls (95,85.9) and (95.9,85) .. (97,85) .. controls (98.1,85) and (99,85.9) .. (99,87) .. controls (99,88.1) and (98.1,89) .. (97,89) .. controls (95.9,89) and (95,88.1) .. (95,87) -- cycle ;
\draw  [color={rgb, 255:red, 88; green, 152; blue, 17 }  ,draw opacity=1 ][fill={rgb, 255:red, 88; green, 152; blue, 17 }  ,fill opacity=1 ] (55,117) .. controls (55,115.9) and (55.9,115) .. (57,115) .. controls (58.1,115) and (59,115.9) .. (59,117) .. controls (59,118.1) and (58.1,119) .. (57,119) .. controls (55.9,119) and (55,118.1) .. (55,117) -- cycle ;
\draw  [color={rgb, 255:red, 74; green, 144; blue, 226 }  ,draw opacity=1 ][fill={rgb, 255:red, 74; green, 144; blue, 226 }  ,fill opacity=1 ] (225,86.94) .. controls (225,85.84) and (225.9,84.94) .. (227,84.94) .. controls (228.1,84.94) and (229,85.84) .. (229,86.94) .. controls (229,88.04) and (228.1,88.94) .. (227,88.94) .. controls (225.9,88.94) and (225,88.04) .. (225,86.94) -- cycle ;
\draw  [color={rgb, 255:red, 74; green, 144; blue, 226 }  ,draw opacity=1 ][fill={rgb, 255:red, 74; green, 144; blue, 226 }  ,fill opacity=1 ] (190,87) .. controls (190,85.9) and (190.9,85) .. (192,85) .. controls (193.1,85) and (194,85.9) .. (194,87) .. controls (194,88.1) and (193.1,89) .. (192,89) .. controls (190.9,89) and (190,88.1) .. (190,87) -- cycle ;
\draw  [color={rgb, 255:red, 74; green, 144; blue, 226 }  ,draw opacity=1 ][fill={rgb, 255:red, 74; green, 144; blue, 226 }  ,fill opacity=1 ] (160,116.94) .. controls (160,115.84) and (160.9,114.94) .. (162,114.94) .. controls (163.1,114.94) and (164,115.84) .. (164,116.94) .. controls (164,118.04) and (163.1,118.94) .. (162,118.94) .. controls (160.9,118.94) and (160,118.04) .. (160,116.94) -- cycle ;
\draw  [color={rgb, 255:red, 74; green, 144; blue, 226 }  ,draw opacity=1 ][fill={rgb, 255:red, 74; green, 144; blue, 226 }  ,fill opacity=1 ] (95,117) .. controls (95,115.9) and (95.9,115) .. (97,115) .. controls (98.1,115) and (99,115.9) .. (99,117) .. controls (99,118.1) and (98.1,119) .. (97,119) .. controls (95.9,119) and (95,118.1) .. (95,117) -- cycle ;
\draw  [color={rgb, 255:red, 74; green, 144; blue, 226 }  ,draw opacity=1 ][fill={rgb, 255:red, 74; green, 144; blue, 226 }  ,fill opacity=1 ] (70,117) .. controls (70,115.9) and (70.9,115) .. (72,115) .. controls (73.1,115) and (74,115.9) .. (74,117) .. controls (74,118.1) and (73.1,119) .. (72,119) .. controls (70.9,119) and (70,118.1) .. (70,117) -- cycle ;
\draw  [color={rgb, 255:red, 74; green, 144; blue, 226 }  ,draw opacity=1 ][fill={rgb, 255:red, 74; green, 144; blue, 226 }  ,fill opacity=1 ] (55,147) .. controls (55,145.9) and (55.9,145) .. (57,145) .. controls (58.1,145) and (59,145.9) .. (59,147) .. controls (59,148.1) and (58.1,149) .. (57,149) .. controls (55.9,149) and (55,148.1) .. (55,147) -- cycle ;
\draw    (57,27) -- (277,57) ;
\draw [shift={(277,57)}, rotate = 7.77] [color={rgb, 255:red, 0; green, 0; blue, 0 }  ][fill={rgb, 255:red, 0; green, 0; blue, 0 }  ][line width=0.75]      (0, 0) circle [x radius= 2.01, y radius= 2.01]   ;
\draw  [color={rgb, 255:red, 74; green, 144; blue, 226 }  ,draw opacity=1 ][fill={rgb, 255:red, 74; green, 144; blue, 226 }  ,fill opacity=1 ] (275,57) .. controls (275,55.9) and (275.9,55) .. (277,55) .. controls (278.1,55) and (279,55.9) .. (279,57) .. controls (279,58.1) and (278.1,59) .. (277,59) .. controls (275.9,59) and (275,58.1) .. (275,57) -- cycle ;

\draw (96.33,67.4) node [anchor=north west][inner sep=0.75pt]  [font=\small]  {$\cdots $};
\draw (72,98.4) node [anchor=north west][inner sep=0.75pt]  [font=\small]  {$\cdots $};
\draw (105,98.4) node [anchor=north west][inner sep=0.75pt]  [font=\small]  {$\cdots $};
\draw (129,98.4) node [anchor=north west][inner sep=0.75pt]  [font=\small]  {$\cdots $};
\draw (81,132.9) node [anchor=north west][inner sep=0.75pt]  [font=\small]  {$\cdots $};
\draw (41,21.4) node [anchor=north west][inner sep=0.75pt]  [font=\small,color={rgb, 255:red, 128; green, 128; blue, 128 }  ,opacity=1 ]  {$\emptyset $};
\draw (42,51.4) node [anchor=north west][inner sep=0.75pt]  [font=\small,color={rgb, 255:red, 139; green, 87; blue, 42 }  ,opacity=1 ]  {$1$};
\draw (37,81.4) node [anchor=north west][inner sep=0.75pt]  [font=\small,color={rgb, 255:red, 245; green, 166; blue, 35 }  ,opacity=1 ]  {$11$};
\draw (29,111.4) node [anchor=north west][inner sep=0.75pt]  [font=\small,color={rgb, 255:red, 88; green, 152; blue, 17 }  ,opacity=1 ]  {$111$};
\draw (166,51.4) node [anchor=north west][inner sep=0.75pt]  [font=\small,color={rgb, 255:red, 245; green, 166; blue, 35 }  ,opacity=1 ]  {$2$};
\draw (77,81.4) node [anchor=north west][inner sep=0.75pt]  [font=\small,color={rgb, 255:red, 88; green, 152; blue, 17 }  ,opacity=1 ]  {$12$};
\draw (107,81.4) node [anchor=north west][inner sep=0.75pt]  [font=\small,color={rgb, 255:red, 74; green, 144; blue, 226 }  ,opacity=1 ]  {$13$};
\draw (142,81.4) node [anchor=north west][inner sep=0.75pt]  [font=\small,color={rgb, 255:red, 88; green, 152; blue, 17 }  ,opacity=1 ]  {$21$};
\draw (22,141.4) node [anchor=north west][inner sep=0.75pt]  [font=\small,color={rgb, 255:red, 74; green, 144; blue, 226 }  ,opacity=1 ]  {$1111$};
\draw (172,81.4) node [anchor=north west][inner sep=0.75pt]  [font=\small,color={rgb, 255:red, 74; green, 144; blue, 226 }  ,opacity=1 ]  {$22$};
\draw (179.33,67.4) node [anchor=north west][inner sep=0.75pt]  [font=\small]  {$\cdots $};
\draw (231,51) node [anchor=north west][inner sep=0.75pt]  [font=\small,color={rgb, 255:red, 88; green, 152; blue, 17 }  ,opacity=1 ] [align=left] {$\displaystyle 3$};
\draw (61,121.4) node [anchor=north west][inner sep=0.75pt]  [font=\footnotesize,color={rgb, 255:red, 74; green, 144; blue, 226 }  ,opacity=1 ]  {$112$};
\draw (87,121.4) node [anchor=north west][inner sep=0.75pt]  [font=\footnotesize,color={rgb, 255:red, 74; green, 144; blue, 226 }  ,opacity=1 ]  {$121$};
\draw (152,121.4) node [anchor=north west][inner sep=0.75pt]  [font=\footnotesize,color={rgb, 255:red, 74; green, 144; blue, 226 }  ,opacity=1 ]  {$211$};
\draw (207,81.4) node [anchor=north west][inner sep=0.75pt]  [font=\small,color={rgb, 255:red, 74; green, 144; blue, 226 }  ,opacity=1 ]  {$31$};
\draw (235.33,67.4) node [anchor=north west][inner sep=0.75pt]  [font=\small]  {$\cdots $};
\draw (170,98.4) node [anchor=north west][inner sep=0.75pt]  [font=\small]  {$\cdots $};
\draw (59,157.4) node [anchor=north west][inner sep=0.75pt]  [font=\small]  {$\cdots $};
\draw (279.33,67.4) node [anchor=north west][inner sep=0.75pt]  [font=\small]  {$\cdots $};
\draw (281,51) node [anchor=north west][inner sep=0.75pt]  [font=\small,color={rgb, 255:red, 74; green, 144; blue, 226 }  ,opacity=1 ] [align=left] {$\displaystyle 4$};
\draw (194.33,98.4) node [anchor=north west][inner sep=0.75pt]  [font=\small]  {$\cdots $};
\draw (230.33,98.4) node [anchor=north west][inner sep=0.75pt]  [font=\small]  {$\cdots $};
\draw (202.33,35.4) node [anchor=north west][inner sep=0.75pt]  [font=\small]  {$\cdots $};

\end{tikzpicture}

%% file: ulam-embedding.tex
\begin{tikzpicture}[x=0.75pt,y=0.75pt,yscale=-1,xscale=1]

\draw    (420,100) -- (420,130) ;
\draw [shift={(420,130)}, rotate = 90] [color={rgb, 255:red, 0; green, 0; blue, 0 }  ][fill={rgb, 255:red, 0; green, 0; blue, 0 }  ][line width=0.75]      (0, 0) circle [x radius= 2.01, y radius= 2.01]   ;
\draw    (420,70) -- (440,100) ;
\draw [shift={(440,100)}, rotate = 56.31] [color={rgb, 255:red, 0; green, 0; blue, 0 }  ][fill={rgb, 255:red, 0; green, 0; blue, 0 }  ][line width=0.75]      (0, 0) circle [x radius= 2.01, y radius= 2.01]   ;
\draw    (470,70) -- (470,99.94) ;
\draw [shift={(470,99.94)}, rotate = 90] [color={rgb, 255:red, 0; green, 0; blue, 0 }  ][fill={rgb, 255:red, 0; green, 0; blue, 0 }  ][line width=0.75]      (0, 0) circle [x radius= 2.01, y radius= 2.01]   ;
\draw    (370,70) -- (370,99.94) ;
\draw [shift={(370,99.94)}, rotate = 90] [color={rgb, 255:red, 0; green, 0; blue, 0 }  ][fill={rgb, 255:red, 0; green, 0; blue, 0 }  ][line width=0.75]      (0, 0) circle [x radius= 2.01, y radius= 2.01]   ;
\draw    (270,40) -- (370,70) ;
\draw [shift={(370,70)}, rotate = 16.7] [color={rgb, 255:red, 0; green, 0; blue, 0 }  ][fill={rgb, 255:red, 0; green, 0; blue, 0 }  ][line width=0.75]      (0, 0) circle [x radius= 2.01, y radius= 2.01]   ;
\draw    (270,40) -- (270,70) ;
\draw [shift={(270,70)}, rotate = 90] [color={rgb, 255:red, 0; green, 0; blue, 0 }  ][fill={rgb, 255:red, 0; green, 0; blue, 0 }  ][line width=0.75]      (0, 0) circle [x radius= 2.01, y radius= 2.01]   ;
\draw [shift={(270,40)}, rotate = 90] [color={rgb, 255:red, 0; green, 0; blue, 0 }  ][fill={rgb, 255:red, 0; green, 0; blue, 0 }  ][line width=0.75]      (0, 0) circle [x radius= 2.01, y radius= 2.01]   ;
\draw    (270,70) -- (305,100) ;
\draw [shift={(305,100)}, rotate = 40.6] [color={rgb, 255:red, 0; green, 0; blue, 0 }  ][fill={rgb, 255:red, 0; green, 0; blue, 0 }  ][line width=0.75]      (0, 0) circle [x radius= 2.01, y radius= 2.01]   ;
\draw    (270,70) -- (340,100) ;
\draw [shift={(340,100)}, rotate = 23.2] [color={rgb, 255:red, 0; green, 0; blue, 0 }  ][fill={rgb, 255:red, 0; green, 0; blue, 0 }  ][line width=0.75]      (0, 0) circle [x radius= 2.01, y radius= 2.01]   ;
\draw    (270,70) -- (270,100) ;
\draw [shift={(270,100)}, rotate = 90] [color={rgb, 255:red, 0; green, 0; blue, 0 }  ][fill={rgb, 255:red, 0; green, 0; blue, 0 }  ][line width=0.75]      (0, 0) circle [x radius= 2.01, y radius= 2.01]   ;
\draw    (270,100) -- (285,130) ;
\draw [shift={(285,130)}, rotate = 63.43] [color={rgb, 255:red, 0; green, 0; blue, 0 }  ][fill={rgb, 255:red, 0; green, 0; blue, 0 }  ][line width=0.75]      (0, 0) circle [x radius= 2.01, y radius= 2.01]   ;
\draw    (270,100) -- (270,130) ;
\draw [shift={(270,130)}, rotate = 90] [color={rgb, 255:red, 0; green, 0; blue, 0 }  ][fill={rgb, 255:red, 0; green, 0; blue, 0 }  ][line width=0.75]      (0, 0) circle [x radius= 2.01, y radius= 2.01]   ;
\draw    (370,70) -- (390,100) ;
\draw [shift={(390,100)}, rotate = 56.31] [color={rgb, 255:red, 0; green, 0; blue, 0 }  ][fill={rgb, 255:red, 0; green, 0; blue, 0 }  ][line width=0.75]      (0, 0) circle [x radius= 2.01, y radius= 2.01]   ;
\draw    (305,100) -- (320,130) ;
\draw [shift={(320,130)}, rotate = 63.43] [color={rgb, 255:red, 0; green, 0; blue, 0 }  ][fill={rgb, 255:red, 0; green, 0; blue, 0 }  ][line width=0.75]      (0, 0) circle [x radius= 2.01, y radius= 2.01]   ;
\draw    (305,100) -- (305,130) ;
\draw [shift={(305,130)}, rotate = 90] [color={rgb, 255:red, 0; green, 0; blue, 0 }  ][fill={rgb, 255:red, 0; green, 0; blue, 0 }  ][line width=0.75]      (0, 0) circle [x radius= 2.01, y radius= 2.01]   ;
\draw [shift={(305,100)}, rotate = 90] [color={rgb, 255:red, 0; green, 0; blue, 0 }  ][fill={rgb, 255:red, 0; green, 0; blue, 0 }  ][line width=0.75]      (0, 0) circle [x radius= 2.01, y radius= 2.01]   ;
\draw    (340,100) -- (340,130) ;
\draw [shift={(340,130)}, rotate = 90] [color={rgb, 255:red, 0; green, 0; blue, 0 }  ][fill={rgb, 255:red, 0; green, 0; blue, 0 }  ][line width=0.75]      (0, 0) circle [x radius= 2.01, y radius= 2.01]   ;
\draw [shift={(340,100)}, rotate = 90] [color={rgb, 255:red, 0; green, 0; blue, 0 }  ][fill={rgb, 255:red, 0; green, 0; blue, 0 }  ][line width=0.75]      (0, 0) circle [x radius= 2.01, y radius= 2.01]   ;
\draw    (270,130) -- (270,160) ;
\draw [shift={(270,160)}, rotate = 90] [color={rgb, 255:red, 0; green, 0; blue, 0 }  ][fill={rgb, 255:red, 0; green, 0; blue, 0 }  ][line width=0.75]      (0, 0) circle [x radius= 2.01, y radius= 2.01]   ;
\draw    (270,40) -- (420,70) ;
\draw [shift={(420,70)}, rotate = 11.31] [color={rgb, 255:red, 0; green, 0; blue, 0 }  ][fill={rgb, 255:red, 0; green, 0; blue, 0 }  ][line width=0.75]      (0, 0) circle [x radius= 2.01, y radius= 2.01]   ;
\draw    (370,99.94) -- (370,129.94) ;
\draw [shift={(370,129.94)}, rotate = 90] [color={rgb, 255:red, 0; green, 0; blue, 0 }  ][fill={rgb, 255:red, 0; green, 0; blue, 0 }  ][line width=0.75]      (0, 0) circle [x radius= 2.01, y radius= 2.01]   ;
\draw [shift={(370,99.94)}, rotate = 90] [color={rgb, 255:red, 0; green, 0; blue, 0 }  ][fill={rgb, 255:red, 0; green, 0; blue, 0 }  ][line width=0.75]      (0, 0) circle [x radius= 2.01, y radius= 2.01]   ;
\draw    (370,99.94) -- (390,129.94) ;
\draw [shift={(390,129.94)}, rotate = 56.31] [color={rgb, 255:red, 0; green, 0; blue, 0 }  ][fill={rgb, 255:red, 0; green, 0; blue, 0 }  ][line width=0.75]      (0, 0) circle [x radius= 2.01, y radius= 2.01]   ;
\draw    (420,70) -- (420,99.94) ;
\draw [shift={(420,99.94)}, rotate = 90] [color={rgb, 255:red, 0; green, 0; blue, 0 }  ][fill={rgb, 255:red, 0; green, 0; blue, 0 }  ][line width=0.75]      (0, 0) circle [x radius= 2.01, y radius= 2.01]   ;
\draw  [color={rgb, 255:red, 139; green, 87; blue, 42 }  ,draw opacity=1 ][fill={rgb, 255:red, 139; green, 87; blue, 42 }  ,fill opacity=1 ] (268,70) .. controls (268,68.9) and (268.9,68) .. (270,68) .. controls (271.1,68) and (272,68.9) .. (272,70) .. controls (272,71.1) and (271.1,72) .. (270,72) .. controls (268.9,72) and (268,71.1) .. (268,70) -- cycle ;
\draw  [color={rgb, 255:red, 245; green, 166; blue, 35 }  ,draw opacity=1 ][fill={rgb, 255:red, 245; green, 166; blue, 35 }  ,fill opacity=1 ] (268,100) .. controls (268,98.9) and (268.9,98) .. (270,98) .. controls (271.1,98) and (272,98.9) .. (272,100) .. controls (272,101.1) and (271.1,102) .. (270,102) .. controls (268.9,102) and (268,101.1) .. (268,100) -- cycle ;
\draw  [color={rgb, 255:red, 245; green, 166; blue, 35 }  ,draw opacity=1 ][fill={rgb, 255:red, 245; green, 166; blue, 35 }  ,fill opacity=1 ] (368,70) .. controls (368,68.9) and (368.9,68) .. (370,68) .. controls (371.1,68) and (372,68.9) .. (372,70) .. controls (372,71.1) and (371.1,72) .. (370,72) .. controls (368.9,72) and (368,71.1) .. (368,70) -- cycle ;
\draw  [color={rgb, 255:red, 88; green, 152; blue, 17 }  ,draw opacity=1 ][fill={rgb, 255:red, 88; green, 152; blue, 17 }  ,fill opacity=1 ] (418,70) .. controls (418,68.9) and (418.9,68) .. (420,68) .. controls (421.1,68) and (422,68.9) .. (422,70) .. controls (422,71.1) and (421.1,72) .. (420,72) .. controls (418.9,72) and (418,71.1) .. (418,70) -- cycle ;
\draw  [color={rgb, 255:red, 88; green, 152; blue, 17 }  ,draw opacity=1 ][fill={rgb, 255:red, 88; green, 152; blue, 17 }  ,fill opacity=1 ] (303,100) .. controls (303,98.9) and (303.9,98) .. (305,98) .. controls (306.1,98) and (307,98.9) .. (307,100) .. controls (307,101.1) and (306.1,102) .. (305,102) .. controls (303.9,102) and (303,101.1) .. (303,100) -- cycle ;
\draw  [color={rgb, 255:red, 88; green, 152; blue, 17 }  ,draw opacity=1 ][fill={rgb, 255:red, 88; green, 152; blue, 17 }  ,fill opacity=1 ] (268,130) .. controls (268,128.9) and (268.9,128) .. (270,128) .. controls (271.1,128) and (272,128.9) .. (272,130) .. controls (272,131.1) and (271.1,132) .. (270,132) .. controls (268.9,132) and (268,131.1) .. (268,130) -- cycle ;
\draw  [color={rgb, 255:red, 74; green, 144; blue, 226 }  ,draw opacity=1 ][fill={rgb, 255:red, 74; green, 144; blue, 226 }  ,fill opacity=1 ] (303,130) .. controls (303,128.9) and (303.9,128) .. (305,128) .. controls (306.1,128) and (307,128.9) .. (307,130) .. controls (307,131.1) and (306.1,132) .. (305,132) .. controls (303.9,132) and (303,131.1) .. (303,130) -- cycle ;
\draw    (270,40) -- (470,70) ;
\draw [shift={(470,70)}, rotate = 8.53] [color={rgb, 255:red, 0; green, 0; blue, 0 }  ][fill={rgb, 255:red, 0; green, 0; blue, 0 }  ][line width=0.75]      (0, 0) circle [x radius= 2.01, y radius= 2.01]   ;
\draw    (140,50) -- (90,80) ;
\draw [shift={(90,80)}, rotate = 149.04] [color={rgb, 255:red, 0; green, 0; blue, 0 }  ][fill={rgb, 255:red, 0; green, 0; blue, 0 }  ][line width=0.75]      (0, 0) circle [x radius= 2.01, y radius= 2.01]   ;
\draw [shift={(140,50)}, rotate = 149.04] [color={rgb, 255:red, 0; green, 0; blue, 0 }  ][fill={rgb, 255:red, 0; green, 0; blue, 0 }  ][line width=0.75]      (0, 0) circle [x radius= 2.01, y radius= 2.01]   ;
\draw    (90,80) -- (60,110) ;
\draw [shift={(60,110)}, rotate = 135] [color={rgb, 255:red, 0; green, 0; blue, 0 }  ][fill={rgb, 255:red, 0; green, 0; blue, 0 }  ][line width=0.75]      (0, 0) circle [x radius= 2.01, y radius= 2.01]   ;
\draw    (140,50) -- (190,80) ;
\draw [shift={(190,80)}, rotate = 30.96] [color={rgb, 255:red, 0; green, 0; blue, 0 }  ][fill={rgb, 255:red, 0; green, 0; blue, 0 }  ][line width=0.75]      (0, 0) circle [x radius= 2.01, y radius= 2.01]   ;
\draw    (60,110) -- (60,140) ;
\draw [shift={(60,140)}, rotate = 90] [color={rgb, 255:red, 0; green, 0; blue, 0 }  ][fill={rgb, 255:red, 0; green, 0; blue, 0 }  ][line width=0.75]      (0, 0) circle [x radius= 2.01, y radius= 2.01]   ;
\draw    (90,80) -- (90,110) ;
\draw [shift={(90,110)}, rotate = 90] [color={rgb, 255:red, 0; green, 0; blue, 0 }  ][fill={rgb, 255:red, 0; green, 0; blue, 0 }  ][line width=0.75]      (0, 0) circle [x radius= 2.01, y radius= 2.01]   ;
\draw    (140,50) -- (140,80) ;
\draw [shift={(140,80)}, rotate = 90] [color={rgb, 255:red, 0; green, 0; blue, 0 }  ][fill={rgb, 255:red, 0; green, 0; blue, 0 }  ][line width=0.75]      (0, 0) circle [x radius= 2.01, y radius= 2.01]   ;
\draw    (90,110) -- (75,140) ;
\draw [shift={(75,140)}, rotate = 116.57] [color={rgb, 255:red, 0; green, 0; blue, 0 }  ][fill={rgb, 255:red, 0; green, 0; blue, 0 }  ][line width=0.75]      (0, 0) circle [x radius= 2.01, y radius= 2.01]   ;
\draw    (90,80) -- (120,110) ;
\draw [shift={(120,110)}, rotate = 45] [color={rgb, 255:red, 0; green, 0; blue, 0 }  ][fill={rgb, 255:red, 0; green, 0; blue, 0 }  ][line width=0.75]      (0, 0) circle [x radius= 2.01, y radius= 2.01]   ;
\draw    (90,110) -- (105,140) ;
\draw [shift={(105,140)}, rotate = 63.43] [color={rgb, 255:red, 0; green, 0; blue, 0 }  ][fill={rgb, 255:red, 0; green, 0; blue, 0 }  ][line width=0.75]      (0, 0) circle [x radius= 2.01, y radius= 2.01]   ;
\draw  [color={rgb, 255:red, 74; green, 144; blue, 226 }  ,draw opacity=1 ][fill={rgb, 255:red, 74; green, 144; blue, 226 }  ,fill opacity=1 ] (338,100) .. controls (338,98.9) and (338.9,98) .. (340,98) .. controls (341.1,98) and (342,98.9) .. (342,100) .. controls (342,101.1) and (341.1,102) .. (340,102) .. controls (338.9,102) and (338,101.1) .. (338,100) -- cycle ;
\draw    (140,80) -- (140,110) ;
\draw [shift={(140,110)}, rotate = 90] [color={rgb, 255:red, 0; green, 0; blue, 0 }  ][fill={rgb, 255:red, 0; green, 0; blue, 0 }  ][line width=0.75]      (0, 0) circle [x radius= 2.01, y radius= 2.01]   ;
\draw    (190,80) -- (175,110) ;
\draw [shift={(175,110)}, rotate = 116.57] [color={rgb, 255:red, 0; green, 0; blue, 0 }  ][fill={rgb, 255:red, 0; green, 0; blue, 0 }  ][line width=0.75]      (0, 0) circle [x radius= 2.01, y radius= 2.01]   ;
\draw    (190,80) -- (205,110) ;
\draw [shift={(205,110)}, rotate = 63.43] [color={rgb, 255:red, 0; green, 0; blue, 0 }  ][fill={rgb, 255:red, 0; green, 0; blue, 0 }  ][line width=0.75]      (0, 0) circle [x radius= 2.01, y radius= 2.01]   ;
\draw  [color={rgb, 255:red, 74; green, 144; blue, 226 }  ,draw opacity=1 ][fill={rgb, 255:red, 74; green, 144; blue, 226 }  ,fill opacity=1 ] (418,100) .. controls (418,98.9) and (418.9,98) .. (420,98) .. controls (421.1,98) and (422,98.9) .. (422,100) .. controls (422,101.1) and (421.1,102) .. (420,102) .. controls (418.9,102) and (418,101.1) .. (418,100) -- cycle ;
\draw  [color={rgb, 255:red, 88; green, 152; blue, 17 }  ,draw opacity=1 ][fill={rgb, 255:red, 88; green, 152; blue, 17 }  ,fill opacity=1 ] (368,100) .. controls (368,98.9) and (368.9,98) .. (370,98) .. controls (371.1,98) and (372,98.9) .. (372,100) .. controls (372,101.1) and (371.1,102) .. (370,102) .. controls (368.9,102) and (368,101.1) .. (368,100) -- cycle ;
\draw    (140,110) -- (140,140) ;
\draw [shift={(140,140)}, rotate = 90] [color={rgb, 255:red, 0; green, 0; blue, 0 }  ][fill={rgb, 255:red, 0; green, 0; blue, 0 }  ][line width=0.75]      (0, 0) circle [x radius= 2.01, y radius= 2.01]   ;
\draw  [color={rgb, 255:red, 74; green, 144; blue, 226 }  ,draw opacity=1 ][fill={rgb, 255:red, 74; green, 144; blue, 226 }  ,fill opacity=1 ] (368,130) .. controls (368,128.9) and (368.9,128) .. (370,128) .. controls (371.1,128) and (372,128.9) .. (372,130) .. controls (372,131.1) and (371.1,132) .. (370,132) .. controls (368.9,132) and (368,131.1) .. (368,130) -- cycle ;

\draw (313.33,80.4) node [anchor=north west][inner sep=0.75pt]  [font=\small]  {$\cdots $};
\draw (279,111.4) node [anchor=north west][inner sep=0.75pt]  [font=\small]  {$\cdots $};
\draw (314,111.4) node [anchor=north west][inner sep=0.75pt]  [font=\small]  {$\cdots $};
\draw (342,111.4) node [anchor=north west][inner sep=0.75pt]  [font=\small]  {$\cdots $};
\draw (276,144.9) node [anchor=north west][inner sep=0.75pt]  [font=\small]  {$\cdots $};
\draw (254,34.4) node [anchor=north west][inner sep=0.75pt]  [font=\small,color={rgb, 255:red, 128; green, 128; blue, 128 }  ,opacity=1 ]  {$1$};
\draw (255,64.4) node [anchor=north west][inner sep=0.75pt]  [font=\small,color={rgb, 255:red, 139; green, 87; blue, 42 }  ,opacity=1 ]  {$2$};
\draw (255,94.4) node [anchor=north west][inner sep=0.75pt]  [font=\small,color={rgb, 255:red, 245; green, 166; blue, 35 }  ,opacity=1 ]  {$3$};
\draw (374,64.4) node [anchor=north west][inner sep=0.75pt]  [font=\small,color={rgb, 255:red, 245; green, 166; blue, 35 }  ,opacity=1 ]  {$5$};
\draw (290,94.4) node [anchor=north west][inner sep=0.75pt]  [font=\small,color={rgb, 255:red, 88; green, 152; blue, 17 }  ,opacity=1 ]  {$4$};
\draw (325,96.4) node [anchor=north west][inner sep=0.75pt]  [font=\small,color={rgb, 255:red, 74; green, 144; blue, 226 }  ,opacity=1 ]  {$6$};
\draw (384,80.4) node [anchor=north west][inner sep=0.75pt]  [font=\small]  {$\cdots $};
\draw (424,64) node [anchor=north west][inner sep=0.75pt]  [font=\small,color={rgb, 255:red, 88; green, 152; blue, 17 }  ,opacity=1 ] [align=left] {$\displaystyle 8$};
\draw (301,134.4) node [anchor=north west][inner sep=0.75pt]  [font=\footnotesize,color={rgb, 255:red, 74; green, 144; blue, 226 }  ,opacity=1 ]  {$7$};
\draw (433.33,80.4) node [anchor=north west][inner sep=0.75pt]  [font=\small]  {$\cdots $};
\draw (384,111.4) node [anchor=north west][inner sep=0.75pt]  [font=\small]  {$\cdots $};
\draw (472.33,80.4) node [anchor=north west][inner sep=0.75pt]  [font=\small]  {$\cdots $};
\draw (423.33,111.4) node [anchor=north west][inner sep=0.75pt]  [font=\small]  {$\cdots $};
\draw (387.33,47.4) node [anchor=north west][inner sep=0.75pt]  [font=\small]  {$\cdots $};
\draw (128,42) node [anchor=north west][inner sep=0.75pt]  [font=\small] [align=left] {$\displaystyle 1$};
\draw (86,65) node [anchor=north west][inner sep=0.75pt]  [font=\small] [align=left] {$\displaystyle 2$};
\draw (115,114) node [anchor=north west][inner sep=0.75pt]  [font=\small] [align=left] {$\displaystyle 6$};
\draw (92,104) node [anchor=north west][inner sep=0.75pt]  [font=\small] [align=left] {$\displaystyle 4$};
\draw (143,73) node [anchor=north west][inner sep=0.75pt]  [font=\small] [align=left] {$\displaystyle 5$};
\draw (56,95) node [anchor=north west][inner sep=0.75pt]  [font=\small] [align=left] {$\displaystyle 3$};
\draw (70,144) node [anchor=north west][inner sep=0.75pt]  [font=\small] [align=left] {$\displaystyle 7$};
\draw (193,73) node [anchor=north west][inner sep=0.75pt]  [font=\small] [align=left] {$\displaystyle 8$};
\draw (56,144) node [anchor=north west][inner sep=0.75pt]  [font=\small] [align=left] {$\displaystyle 9$};
\draw (97,144) node [anchor=north west][inner sep=0.75pt]  [font=\small] [align=left] {$\displaystyle 10$};
\draw (313,134.4) node [anchor=north west][inner sep=0.75pt]  [font=\footnotesize,color={rgb, 255:red, 0; green, 0; blue, 0 }  ,opacity=1 ]  {$10$};
\draw (255,124.4) node [anchor=north west][inner sep=0.75pt]  [font=\small,color={rgb, 255:red, 88; green, 152; blue, 17 }  ,opacity=1 ]  {$9$};
\draw (350,94.4) node [anchor=north west][inner sep=0.75pt]  [font=\small,color={rgb, 255:red, 88; green, 152; blue, 17 }  ,opacity=1 ]  {$12$};
\draw (141,104) node [anchor=north west][inner sep=0.75pt]  [font=\small] [align=left] {$\displaystyle 12$};
\draw (170,114) node [anchor=north west][inner sep=0.75pt]  [font=\small] [align=left] {$\displaystyle 11$};
\draw (197,114) node [anchor=north west][inner sep=0.75pt]  [font=\small] [align=left] {$\displaystyle 13$};
\draw (401,95.4) node [anchor=north west][inner sep=0.75pt]  [font=\small,color={rgb, 255:red, 74; green, 144; blue, 226 }  ,opacity=1 ]  {$11$};
\draw (442,95.4) node [anchor=north west][inner sep=0.75pt]  [font=\small,color={rgb, 255:red, 0; green, 0; blue, 0 }  ,opacity=1 ]  {$13$};
\draw (132,144) node [anchor=north west][inner sep=0.75pt]  [font=\small] [align=left] {$\displaystyle 14$};
\draw (363,134.4) node [anchor=north west][inner sep=0.75pt]  [font=\footnotesize,color={rgb, 255:red, 74; green, 144; blue, 226 }  ,opacity=1 ]  {$14$};

\end{tikzpicture}

%% file: tree-after-bijection.tex
\begin{tikzpicture}[x=0.75pt,y=0.75pt,yscale=-1,xscale=1]

\draw    (220,100) -- (220,130) ;
\draw [shift={(220,130)}, rotate = 90] [color={rgb, 255:red, 0; green, 0; blue, 0 }  ][fill={rgb, 255:red, 0; green, 0; blue, 0 }  ][line width=0.75]      (0, 0) circle [x radius= 2.01, y radius= 2.01]   ;
\draw    (220,70) -- (240,100) ;
\draw [shift={(240,100)}, rotate = 56.31] [color={rgb, 255:red, 0; green, 0; blue, 0 }  ][fill={rgb, 255:red, 0; green, 0; blue, 0 }  ][line width=0.75]      (0, 0) circle [x radius= 2.01, y radius= 2.01]   ;
\draw    (270,70) -- (270,99.94) ;
\draw [shift={(270,99.94)}, rotate = 90] [color={rgb, 255:red, 0; green, 0; blue, 0 }  ][fill={rgb, 255:red, 0; green, 0; blue, 0 }  ][line width=0.75]      (0, 0) circle [x radius= 2.01, y radius= 2.01]   ;
\draw    (170,70) -- (170,99.94) ;
\draw [shift={(170,99.94)}, rotate = 90] [color={rgb, 255:red, 0; green, 0; blue, 0 }  ][fill={rgb, 255:red, 0; green, 0; blue, 0 }  ][line width=0.75]      (0, 0) circle [x radius= 2.01, y radius= 2.01]   ;
\draw    (70,40) -- (170,70) ;
\draw [shift={(170,70)}, rotate = 16.7] [color={rgb, 255:red, 0; green, 0; blue, 0 }  ][fill={rgb, 255:red, 0; green, 0; blue, 0 }  ][line width=0.75]      (0, 0) circle [x radius= 2.01, y radius= 2.01]   ;
\draw    (70,40) -- (70,70) ;
\draw [shift={(70,70)}, rotate = 90] [color={rgb, 255:red, 0; green, 0; blue, 0 }  ][fill={rgb, 255:red, 0; green, 0; blue, 0 }  ][line width=0.75]      (0, 0) circle [x radius= 2.01, y radius= 2.01]   ;
\draw [shift={(70,40)}, rotate = 90] [color={rgb, 255:red, 0; green, 0; blue, 0 }  ][fill={rgb, 255:red, 0; green, 0; blue, 0 }  ][line width=0.75]      (0, 0) circle [x radius= 2.01, y radius= 2.01]   ;
\draw    (70,70) -- (105,100) ;
\draw [shift={(105,100)}, rotate = 40.6] [color={rgb, 255:red, 0; green, 0; blue, 0 }  ][fill={rgb, 255:red, 0; green, 0; blue, 0 }  ][line width=0.75]      (0, 0) circle [x radius= 2.01, y radius= 2.01]   ;
\draw    (70,70) -- (140,100) ;
\draw [shift={(140,100)}, rotate = 23.2] [color={rgb, 255:red, 0; green, 0; blue, 0 }  ][fill={rgb, 255:red, 0; green, 0; blue, 0 }  ][line width=0.75]      (0, 0) circle [x radius= 2.01, y radius= 2.01]   ;
\draw    (70,70) -- (70,100) ;
\draw [shift={(70,100)}, rotate = 90] [color={rgb, 255:red, 0; green, 0; blue, 0 }  ][fill={rgb, 255:red, 0; green, 0; blue, 0 }  ][line width=0.75]      (0, 0) circle [x radius= 2.01, y radius= 2.01]   ;
\draw    (70,100) -- (85,130) ;
\draw [shift={(85,130)}, rotate = 63.43] [color={rgb, 255:red, 0; green, 0; blue, 0 }  ][fill={rgb, 255:red, 0; green, 0; blue, 0 }  ][line width=0.75]      (0, 0) circle [x radius= 2.01, y radius= 2.01]   ;
\draw    (70,100) -- (70,130) ;
\draw [shift={(70,130)}, rotate = 90] [color={rgb, 255:red, 0; green, 0; blue, 0 }  ][fill={rgb, 255:red, 0; green, 0; blue, 0 }  ][line width=0.75]      (0, 0) circle [x radius= 2.01, y radius= 2.01]   ;
\draw    (170,70) -- (190,100) ;
\draw [shift={(190,100)}, rotate = 56.31] [color={rgb, 255:red, 0; green, 0; blue, 0 }  ][fill={rgb, 255:red, 0; green, 0; blue, 0 }  ][line width=0.75]      (0, 0) circle [x radius= 2.01, y radius= 2.01]   ;
\draw    (105,100) -- (120,130) ;
\draw [shift={(120,130)}, rotate = 63.43] [color={rgb, 255:red, 0; green, 0; blue, 0 }  ][fill={rgb, 255:red, 0; green, 0; blue, 0 }  ][line width=0.75]      (0, 0) circle [x radius= 2.01, y radius= 2.01]   ;
\draw    (105,100) -- (105,130) ;
\draw [shift={(105,130)}, rotate = 90] [color={rgb, 255:red, 0; green, 0; blue, 0 }  ][fill={rgb, 255:red, 0; green, 0; blue, 0 }  ][line width=0.75]      (0, 0) circle [x radius= 2.01, y radius= 2.01]   ;
\draw [shift={(105,100)}, rotate = 90] [color={rgb, 255:red, 0; green, 0; blue, 0 }  ][fill={rgb, 255:red, 0; green, 0; blue, 0 }  ][line width=0.75]      (0, 0) circle [x radius= 2.01, y radius= 2.01]   ;
\draw    (140,100) -- (140,130) ;
\draw [shift={(140,130)}, rotate = 90] [color={rgb, 255:red, 0; green, 0; blue, 0 }  ][fill={rgb, 255:red, 0; green, 0; blue, 0 }  ][line width=0.75]      (0, 0) circle [x radius= 2.01, y radius= 2.01]   ;
\draw [shift={(140,100)}, rotate = 90] [color={rgb, 255:red, 0; green, 0; blue, 0 }  ][fill={rgb, 255:red, 0; green, 0; blue, 0 }  ][line width=0.75]      (0, 0) circle [x radius= 2.01, y radius= 2.01]   ;
\draw    (70,130) -- (70,160) ;
\draw [shift={(70,160)}, rotate = 90] [color={rgb, 255:red, 0; green, 0; blue, 0 }  ][fill={rgb, 255:red, 0; green, 0; blue, 0 }  ][line width=0.75]      (0, 0) circle [x radius= 2.01, y radius= 2.01]   ;
\draw    (70,40) -- (220,70) ;
\draw [shift={(220,70)}, rotate = 11.31] [color={rgb, 255:red, 0; green, 0; blue, 0 }  ][fill={rgb, 255:red, 0; green, 0; blue, 0 }  ][line width=0.75]      (0, 0) circle [x radius= 2.01, y radius= 2.01]   ;
\draw    (170,99.94) -- (170,129.94) ;
\draw [shift={(170,129.94)}, rotate = 90] [color={rgb, 255:red, 0; green, 0; blue, 0 }  ][fill={rgb, 255:red, 0; green, 0; blue, 0 }  ][line width=0.75]      (0, 0) circle [x radius= 2.01, y radius= 2.01]   ;
\draw [shift={(170,99.94)}, rotate = 90] [color={rgb, 255:red, 0; green, 0; blue, 0 }  ][fill={rgb, 255:red, 0; green, 0; blue, 0 }  ][line width=0.75]      (0, 0) circle [x radius= 2.01, y radius= 2.01]   ;
\draw    (170,99.94) -- (190,129.94) ;
\draw [shift={(190,129.94)}, rotate = 56.31] [color={rgb, 255:red, 0; green, 0; blue, 0 }  ][fill={rgb, 255:red, 0; green, 0; blue, 0 }  ][line width=0.75]      (0, 0) circle [x radius= 2.01, y radius= 2.01]   ;
\draw    (220,70) -- (220,99.94) ;
\draw [shift={(220,99.94)}, rotate = 90] [color={rgb, 255:red, 0; green, 0; blue, 0 }  ][fill={rgb, 255:red, 0; green, 0; blue, 0 }  ][line width=0.75]      (0, 0) circle [x radius= 2.01, y radius= 2.01]   ;
\draw  [color={rgb, 255:red, 139; green, 87; blue, 42 }  ,draw opacity=1 ][fill={rgb, 255:red, 139; green, 87; blue, 42 }  ,fill opacity=1 ] (68,70) .. controls (68,68.9) and (68.9,68) .. (70,68) .. controls (71.1,68) and (72,68.9) .. (72,70) .. controls (72,71.1) and (71.1,72) .. (70,72) .. controls (68.9,72) and (68,71.1) .. (68,70) -- cycle ;
\draw  [color={rgb, 255:red, 245; green, 166; blue, 35 }  ,draw opacity=1 ][fill={rgb, 255:red, 245; green, 166; blue, 35 }  ,fill opacity=1 ] (68,100) .. controls (68,98.9) and (68.9,98) .. (70,98) .. controls (71.1,98) and (72,98.9) .. (72,100) .. controls (72,101.1) and (71.1,102) .. (70,102) .. controls (68.9,102) and (68,101.1) .. (68,100) -- cycle ;
\draw  [color={rgb, 255:red, 245; green, 166; blue, 35 }  ,draw opacity=1 ][fill={rgb, 255:red, 245; green, 166; blue, 35 }  ,fill opacity=1 ] (168,70) .. controls (168,68.9) and (168.9,68) .. (170,68) .. controls (171.1,68) and (172,68.9) .. (172,70) .. controls (172,71.1) and (171.1,72) .. (170,72) .. controls (168.9,72) and (168,71.1) .. (168,70) -- cycle ;
\draw  [color={rgb, 255:red, 88; green, 152; blue, 17 }  ,draw opacity=1 ][fill={rgb, 255:red, 88; green, 152; blue, 17 }  ,fill opacity=1 ] (218,70) .. controls (218,68.9) and (218.9,68) .. (220,68) .. controls (221.1,68) and (222,68.9) .. (222,70) .. controls (222,71.1) and (221.1,72) .. (220,72) .. controls (218.9,72) and (218,71.1) .. (218,70) -- cycle ;
\draw  [color={rgb, 255:red, 88; green, 152; blue, 17 }  ,draw opacity=1 ][fill={rgb, 255:red, 88; green, 152; blue, 17 }  ,fill opacity=1 ] (103,100) .. controls (103,98.9) and (103.9,98) .. (105,98) .. controls (106.1,98) and (107,98.9) .. (107,100) .. controls (107,101.1) and (106.1,102) .. (105,102) .. controls (103.9,102) and (103,101.1) .. (103,100) -- cycle ;
\draw  [color={rgb, 255:red, 88; green, 152; blue, 17 }  ,draw opacity=1 ][fill={rgb, 255:red, 88; green, 152; blue, 17 }  ,fill opacity=1 ] (68,130) .. controls (68,128.9) and (68.9,128) .. (70,128) .. controls (71.1,128) and (72,128.9) .. (72,130) .. controls (72,131.1) and (71.1,132) .. (70,132) .. controls (68.9,132) and (68,131.1) .. (68,130) -- cycle ;
\draw  [color={rgb, 255:red, 74; green, 144; blue, 226 }  ,draw opacity=1 ][fill={rgb, 255:red, 74; green, 144; blue, 226 }  ,fill opacity=1 ] (103,130) .. controls (103,128.9) and (103.9,128) .. (105,128) .. controls (106.1,128) and (107,128.9) .. (107,130) .. controls (107,131.1) and (106.1,132) .. (105,132) .. controls (103.9,132) and (103,131.1) .. (103,130) -- cycle ;
\draw    (70,40) -- (270,70) ;
\draw [shift={(270,70)}, rotate = 8.53] [color={rgb, 255:red, 0; green, 0; blue, 0 }  ][fill={rgb, 255:red, 0; green, 0; blue, 0 }  ][line width=0.75]      (0, 0) circle [x radius= 2.01, y radius= 2.01]   ;
\draw    (415,40) -- (360,70) ;
\draw [shift={(360,70)}, rotate = 151.39] [color={rgb, 255:red, 0; green, 0; blue, 0 }  ][fill={rgb, 255:red, 0; green, 0; blue, 0 }  ][line width=0.75]      (0, 0) circle [x radius= 2.01, y radius= 2.01]   ;
\draw [shift={(415,40)}, rotate = 151.39] [color={rgb, 255:red, 0; green, 0; blue, 0 }  ][fill={rgb, 255:red, 0; green, 0; blue, 0 }  ][line width=0.75]      (0, 0) circle [x radius= 2.01, y radius= 2.01]   ;
\draw    (415,40) -- (400,70) ;
\draw [shift={(400,70)}, rotate = 116.57] [color={rgb, 255:red, 0; green, 0; blue, 0 }  ][fill={rgb, 255:red, 0; green, 0; blue, 0 }  ][line width=0.75]      (0, 0) circle [x radius= 2.01, y radius= 2.01]   ;
\draw    (360,70) -- (380,100) ;
\draw [shift={(380,100)}, rotate = 56.31] [color={rgb, 255:red, 0; green, 0; blue, 0 }  ][fill={rgb, 255:red, 0; green, 0; blue, 0 }  ][line width=0.75]      (0, 0) circle [x radius= 2.01, y radius= 2.01]   ;
\draw    (400,70) -- (400,100) ;
\draw [shift={(400,100)}, rotate = 90] [color={rgb, 255:red, 0; green, 0; blue, 0 }  ][fill={rgb, 255:red, 0; green, 0; blue, 0 }  ][line width=0.75]      (0, 0) circle [x radius= 2.01, y radius= 2.01]   ;
\draw    (415,40) -- (430,70) ;
\draw [shift={(430,70)}, rotate = 63.43] [color={rgb, 255:red, 0; green, 0; blue, 0 }  ][fill={rgb, 255:red, 0; green, 0; blue, 0 }  ][line width=0.75]      (0, 0) circle [x radius= 2.01, y radius= 2.01]   ;
\draw    (360,70) -- (340,100) ;
\draw [shift={(340,100)}, rotate = 123.69] [color={rgb, 255:red, 0; green, 0; blue, 0 }  ][fill={rgb, 255:red, 0; green, 0; blue, 0 }  ][line width=0.75]      (0, 0) circle [x radius= 2.01, y radius= 2.01]   ;
\draw    (430,70) -- (420,100) ;
\draw [shift={(420,100)}, rotate = 108.43] [color={rgb, 255:red, 0; green, 0; blue, 0 }  ][fill={rgb, 255:red, 0; green, 0; blue, 0 }  ][line width=0.75]      (0, 0) circle [x radius= 2.01, y radius= 2.01]   ;
\draw    (415,40) -- (470,70) ;
\draw [shift={(470,70)}, rotate = 28.61] [color={rgb, 255:red, 0; green, 0; blue, 0 }  ][fill={rgb, 255:red, 0; green, 0; blue, 0 }  ][line width=0.75]      (0, 0) circle [x radius= 2.01, y radius= 2.01]   ;
\draw    (430,70) -- (440,100) ;
\draw [shift={(440,100)}, rotate = 71.57] [color={rgb, 255:red, 0; green, 0; blue, 0 }  ][fill={rgb, 255:red, 0; green, 0; blue, 0 }  ][line width=0.75]      (0, 0) circle [x radius= 2.01, y radius= 2.01]   ;
\draw  [color={rgb, 255:red, 74; green, 144; blue, 226 }  ,draw opacity=1 ][fill={rgb, 255:red, 74; green, 144; blue, 226 }  ,fill opacity=1 ] (268,70) .. controls (268,68.9) and (268.9,68) .. (270,68) .. controls (271.1,68) and (272,68.9) .. (272,70) .. controls (272,71.1) and (271.1,72) .. (270,72) .. controls (268.9,72) and (268,71.1) .. (268,70) -- cycle ;
\draw    (340,100) -- (340,130) ;
\draw [shift={(340,130)}, rotate = 90] [color={rgb, 255:red, 0; green, 0; blue, 0 }  ][fill={rgb, 255:red, 0; green, 0; blue, 0 }  ][line width=0.75]      (0, 0) circle [x radius= 2.01, y radius= 2.01]   ;
\draw    (380,100) -- (370,130) ;
\draw [shift={(370,130)}, rotate = 108.43] [color={rgb, 255:red, 0; green, 0; blue, 0 }  ][fill={rgb, 255:red, 0; green, 0; blue, 0 }  ][line width=0.75]      (0, 0) circle [x radius= 2.01, y radius= 2.01]   ;
\draw    (380,100) -- (390,130) ;
\draw [shift={(390,130)}, rotate = 71.57] [color={rgb, 255:red, 0; green, 0; blue, 0 }  ][fill={rgb, 255:red, 0; green, 0; blue, 0 }  ][line width=0.75]      (0, 0) circle [x radius= 2.01, y radius= 2.01]   ;
\draw  [color={rgb, 255:red, 74; green, 144; blue, 226 }  ,draw opacity=1 ][fill={rgb, 255:red, 74; green, 144; blue, 226 }  ,fill opacity=1 ] (218,100) .. controls (218,98.9) and (218.9,98) .. (220,98) .. controls (221.1,98) and (222,98.9) .. (222,100) .. controls (222,101.1) and (221.1,102) .. (220,102) .. controls (218.9,102) and (218,101.1) .. (218,100) -- cycle ;
\draw  [color={rgb, 255:red, 88; green, 152; blue, 17 }  ,draw opacity=1 ][fill={rgb, 255:red, 88; green, 152; blue, 17 }  ,fill opacity=1 ] (168,100) .. controls (168,98.9) and (168.9,98) .. (170,98) .. controls (171.1,98) and (172,98.9) .. (172,100) .. controls (172,101.1) and (171.1,102) .. (170,102) .. controls (168.9,102) and (168,101.1) .. (168,100) -- cycle ;
\draw    (340,130) -- (340,160) ;
\draw [shift={(340,160)}, rotate = 90] [color={rgb, 255:red, 0; green, 0; blue, 0 }  ][fill={rgb, 255:red, 0; green, 0; blue, 0 }  ][line width=0.75]      (0, 0) circle [x radius= 2.01, y radius= 2.01]   ;
\draw  [color={rgb, 255:red, 74; green, 144; blue, 226 }  ,draw opacity=1 ][fill={rgb, 255:red, 74; green, 144; blue, 226 }  ,fill opacity=1 ] (68,160) .. controls (68,158.9) and (68.9,158) .. (70,158) .. controls (71.1,158) and (72,158.9) .. (72,160) .. controls (72,161.1) and (71.1,162) .. (70,162) .. controls (68.9,162) and (68,161.1) .. (68,160) -- cycle ;

\draw (113.33,80.4) node [anchor=north west][inner sep=0.75pt]  [font=\small]  {$\cdots $};
\draw (79,111.4) node [anchor=north west][inner sep=0.75pt]  [font=\small]  {$\cdots $};
\draw (114,111.4) node [anchor=north west][inner sep=0.75pt]  [font=\small]  {$\cdots $};
\draw (142,111.4) node [anchor=north west][inner sep=0.75pt]  [font=\small]  {$\cdots $};
\draw (76,144.9) node [anchor=north west][inner sep=0.75pt]  [font=\small]  {$\cdots $};
\draw (54,34.4) node [anchor=north west][inner sep=0.75pt]  [font=\small,color={rgb, 255:red, 128; green, 128; blue, 128 }  ,opacity=1 ]  {$1$};
\draw (55,64.4) node [anchor=north west][inner sep=0.75pt]  [font=\small,color={rgb, 255:red, 139; green, 87; blue, 42 }  ,opacity=1 ]  {$2$};
\draw (55,94.4) node [anchor=north west][inner sep=0.75pt]  [font=\small,color={rgb, 255:red, 245; green, 166; blue, 35 }  ,opacity=1 ]  {$5$};
\draw (174,64.4) node [anchor=north west][inner sep=0.75pt]  [font=\small,color={rgb, 255:red, 245; green, 166; blue, 35 }  ,opacity=1 ]  {$3$};
\draw (90,94.4) node [anchor=north west][inner sep=0.75pt]  [font=\small,color={rgb, 255:red, 88; green, 152; blue, 17 }  ,opacity=1 ]  {$8$};
\draw (273,64.4) node [anchor=north west][inner sep=0.75pt]  [font=\small,color={rgb, 255:red, 74; green, 144; blue, 226 }  ,opacity=1 ]  {$6$};
\draw (184,80.4) node [anchor=north west][inner sep=0.75pt]  [font=\small]  {$\cdots $};
\draw (224,64) node [anchor=north west][inner sep=0.75pt]  [font=\small,color={rgb, 255:red, 88; green, 152; blue, 17 }  ,opacity=1 ] [align=left] {$\displaystyle 4$};
\draw (97,134.4) node [anchor=north west][inner sep=0.75pt]  [font=\footnotesize,color={rgb, 255:red, 74; green, 144; blue, 226 }  ,opacity=1 ]  {$11$};
\draw (233.33,80.4) node [anchor=north west][inner sep=0.75pt]  [font=\small]  {$\cdots $};
\draw (184,111.4) node [anchor=north west][inner sep=0.75pt]  [font=\small]  {$\cdots $};
\draw (272.33,80.4) node [anchor=north west][inner sep=0.75pt]  [font=\small]  {$\cdots $};
\draw (223.33,111.4) node [anchor=north west][inner sep=0.75pt]  [font=\small]  {$\cdots $};
\draw (187.33,47.4) node [anchor=north west][inner sep=0.75pt]  [font=\small]  {$\cdots $};
\draw (402,32) node [anchor=north west][inner sep=0.75pt]  [font=\small] [align=left] {$\displaystyle 1$};
\draw (346,64) node [anchor=north west][inner sep=0.75pt]  [font=\small] [align=left] {$\displaystyle 2$};
\draw (474,64) node [anchor=north west][inner sep=0.75pt]  [font=\small] [align=left] {$\displaystyle 6$};
\draw (434,64) node [anchor=north west][inner sep=0.75pt]  [font=\small] [align=left] {$\displaystyle 4$};
\draw (326,94) node [anchor=north west][inner sep=0.75pt]  [font=\small] [align=left] {$\displaystyle 5$};
\draw (386,64) node [anchor=north west][inner sep=0.75pt]  [font=\small] [align=left] {$\displaystyle 3$};
\draw (415,104) node [anchor=north west][inner sep=0.75pt]  [font=\small] [align=left] {$\displaystyle 7$};
\draw (366,94) node [anchor=north west][inner sep=0.75pt]  [font=\small] [align=left] {$\displaystyle 8$};
\draw (396,104) node [anchor=north west][inner sep=0.75pt]  [font=\small] [align=left] {$\displaystyle 9$};
\draw (432,104) node [anchor=north west][inner sep=0.75pt]  [font=\small] [align=left] {$\displaystyle 10$};
\draw (113,134.4) node [anchor=north west][inner sep=0.75pt]  [font=\footnotesize,color={rgb, 255:red, 0; green, 0; blue, 0 }  ,opacity=1 ]  {$13$};
\draw (50,124.4) node [anchor=north west][inner sep=0.75pt]  [font=\small,color={rgb, 255:red, 88; green, 152; blue, 17 }  ,opacity=1 ]  {$12$};
\draw (155,94.4) node [anchor=north west][inner sep=0.75pt]  [font=\small,color={rgb, 255:red, 88; green, 152; blue, 17 }  ,opacity=1 ]  {$9$};
\draw (320,124) node [anchor=north west][inner sep=0.75pt]  [font=\small] [align=left] {$\displaystyle 12$};
\draw (362,134) node [anchor=north west][inner sep=0.75pt]  [font=\small] [align=left] {$\displaystyle 11$};
\draw (382,134) node [anchor=north west][inner sep=0.75pt]  [font=\small] [align=left] {$\displaystyle 13$};
\draw (205,95.4) node [anchor=north west][inner sep=0.75pt]  [font=\small,color={rgb, 255:red, 74; green, 144; blue, 226 }  ,opacity=1 ]  {$7$};
\draw (242,95.4) node [anchor=north west][inner sep=0.75pt]  [font=\small,color={rgb, 255:red, 0; green, 0; blue, 0 }  ,opacity=1 ]  {$10$};
\draw (320,154) node [anchor=north west][inner sep=0.75pt]  [font=\small] [align=left] {$\displaystyle 14$};
\draw (52,154.4) node [anchor=north west][inner sep=0.75pt]  [font=\footnotesize,color={rgb, 255:red, 74; green, 144; blue, 226 }  ,opacity=1 ]  {$14$};

\end{tikzpicture}

%% file: ulam-child-sibling.tex
\begin{tikzpicture}[x=0.75pt,y=0.75pt,yscale=-1,xscale=1]

\draw    (530,70) -- (578,70) ;
\draw [shift={(580,70)}, rotate = 180] [fill={rgb, 255:red, 0; green, 0; blue, 0 }  ][line width=0.08]  [draw opacity=0] (7.2,-1.8) -- (0,0) -- (7.2,1.8) -- cycle    ;
\draw    (220,100) -- (220,130) ;
\draw [shift={(220,130)}, rotate = 90] [color={rgb, 255:red, 0; green, 0; blue, 0 }  ][fill={rgb, 255:red, 0; green, 0; blue, 0 }  ][line width=0.75]      (0, 0) circle [x radius= 2.01, y radius= 2.01]   ;
\draw    (220,70) -- (235,100) ;
\draw [shift={(235,100)}, rotate = 63.43] [color={rgb, 255:red, 0; green, 0; blue, 0 }  ][fill={rgb, 255:red, 0; green, 0; blue, 0 }  ][line width=0.75]      (0, 0) circle [x radius= 2.01, y radius= 2.01]   ;
\draw    (270,70) -- (270,99.94) ;
\draw [shift={(270,99.94)}, rotate = 90] [color={rgb, 255:red, 0; green, 0; blue, 0 }  ][fill={rgb, 255:red, 0; green, 0; blue, 0 }  ][line width=0.75]      (0, 0) circle [x radius= 2.01, y radius= 2.01]   ;
\draw    (185,100) -- (185,130) ;
\draw [shift={(185,130)}, rotate = 90] [color={rgb, 255:red, 0; green, 0; blue, 0 }  ][fill={rgb, 255:red, 0; green, 0; blue, 0 }  ][line width=0.75]      (0, 0) circle [x radius= 2.01, y radius= 2.01]   ;
\draw    (155,70) -- (155,99.94) ;
\draw [shift={(155,99.94)}, rotate = 90] [color={rgb, 255:red, 0; green, 0; blue, 0 }  ][fill={rgb, 255:red, 0; green, 0; blue, 0 }  ][line width=0.75]      (0, 0) circle [x radius= 2.01, y radius= 2.01]   ;
\draw    (50,40) -- (155,70) ;
\draw [shift={(155,70)}, rotate = 15.95] [color={rgb, 255:red, 0; green, 0; blue, 0 }  ][fill={rgb, 255:red, 0; green, 0; blue, 0 }  ][line width=0.75]      (0, 0) circle [x radius= 2.01, y radius= 2.01]   ;
\draw    (50,40) -- (50,70) ;
\draw [shift={(50,70)}, rotate = 90] [color={rgb, 255:red, 0; green, 0; blue, 0 }  ][fill={rgb, 255:red, 0; green, 0; blue, 0 }  ][line width=0.75]      (0, 0) circle [x radius= 2.01, y radius= 2.01]   ;
\draw [shift={(50,40)}, rotate = 90] [color={rgb, 255:red, 0; green, 0; blue, 0 }  ][fill={rgb, 255:red, 0; green, 0; blue, 0 }  ][line width=0.75]      (0, 0) circle [x radius= 2.01, y radius= 2.01]   ;
\draw    (50,70) -- (90,100) ;
\draw [shift={(90,100)}, rotate = 36.87] [color={rgb, 255:red, 0; green, 0; blue, 0 }  ][fill={rgb, 255:red, 0; green, 0; blue, 0 }  ][line width=0.75]      (0, 0) circle [x radius= 2.01, y radius= 2.01]   ;
\draw    (50,70) -- (120,100) ;
\draw [shift={(120,100)}, rotate = 23.2] [color={rgb, 255:red, 0; green, 0; blue, 0 }  ][fill={rgb, 255:red, 0; green, 0; blue, 0 }  ][line width=0.75]      (0, 0) circle [x radius= 2.01, y radius= 2.01]   ;
\draw    (50,70) -- (50,100) ;
\draw [shift={(50,100)}, rotate = 90] [color={rgb, 255:red, 0; green, 0; blue, 0 }  ][fill={rgb, 255:red, 0; green, 0; blue, 0 }  ][line width=0.75]      (0, 0) circle [x radius= 2.01, y radius= 2.01]   ;
\draw    (50,100) -- (80,130) ;
\draw [shift={(80,130)}, rotate = 45] [color={rgb, 255:red, 0; green, 0; blue, 0 }  ][fill={rgb, 255:red, 0; green, 0; blue, 0 }  ][line width=0.75]      (0, 0) circle [x radius= 2.01, y radius= 2.01]   ;
\draw    (50,100) -- (65,130) ;
\draw [shift={(65,130)}, rotate = 63.43] [color={rgb, 255:red, 0; green, 0; blue, 0 }  ][fill={rgb, 255:red, 0; green, 0; blue, 0 }  ][line width=0.75]      (0, 0) circle [x radius= 2.01, y radius= 2.01]   ;
\draw    (50,100) -- (50,130) ;
\draw [shift={(50,130)}, rotate = 90] [color={rgb, 255:red, 0; green, 0; blue, 0 }  ][fill={rgb, 255:red, 0; green, 0; blue, 0 }  ][line width=0.75]      (0, 0) circle [x radius= 2.01, y radius= 2.01]   ;
\draw    (155,70) -- (185,100) ;
\draw [shift={(185,100)}, rotate = 45] [color={rgb, 255:red, 0; green, 0; blue, 0 }  ][fill={rgb, 255:red, 0; green, 0; blue, 0 }  ][line width=0.75]      (0, 0) circle [x radius= 2.01, y radius= 2.01]   ;
\draw    (90,100) -- (105,130) ;
\draw [shift={(105,130)}, rotate = 63.43] [color={rgb, 255:red, 0; green, 0; blue, 0 }  ][fill={rgb, 255:red, 0; green, 0; blue, 0 }  ][line width=0.75]      (0, 0) circle [x radius= 2.01, y radius= 2.01]   ;
\draw    (90,100) -- (90,130) ;
\draw [shift={(90,130)}, rotate = 90] [color={rgb, 255:red, 0; green, 0; blue, 0 }  ][fill={rgb, 255:red, 0; green, 0; blue, 0 }  ][line width=0.75]      (0, 0) circle [x radius= 2.01, y radius= 2.01]   ;
\draw [shift={(90,100)}, rotate = 90] [color={rgb, 255:red, 0; green, 0; blue, 0 }  ][fill={rgb, 255:red, 0; green, 0; blue, 0 }  ][line width=0.75]      (0, 0) circle [x radius= 2.01, y radius= 2.01]   ;
\draw    (120,100) -- (120,130) ;
\draw [shift={(120,130)}, rotate = 90] [color={rgb, 255:red, 0; green, 0; blue, 0 }  ][fill={rgb, 255:red, 0; green, 0; blue, 0 }  ][line width=0.75]      (0, 0) circle [x radius= 2.01, y radius= 2.01]   ;
\draw [shift={(120,100)}, rotate = 90] [color={rgb, 255:red, 0; green, 0; blue, 0 }  ][fill={rgb, 255:red, 0; green, 0; blue, 0 }  ][line width=0.75]      (0, 0) circle [x radius= 2.01, y radius= 2.01]   ;
\draw    (50,130) -- (65,160) ;
\draw [shift={(65,160)}, rotate = 63.43] [color={rgb, 255:red, 0; green, 0; blue, 0 }  ][fill={rgb, 255:red, 0; green, 0; blue, 0 }  ][line width=0.75]      (0, 0) circle [x radius= 2.01, y radius= 2.01]   ;
\draw    (50,130) -- (50,160) ;
\draw [shift={(50,160)}, rotate = 90] [color={rgb, 255:red, 0; green, 0; blue, 0 }  ][fill={rgb, 255:red, 0; green, 0; blue, 0 }  ][line width=0.75]      (0, 0) circle [x radius= 2.01, y radius= 2.01]   ;
\draw    (50,40) -- (220,70) ;
\draw [shift={(220,70)}, rotate = 10.01] [color={rgb, 255:red, 0; green, 0; blue, 0 }  ][fill={rgb, 255:red, 0; green, 0; blue, 0 }  ][line width=0.75]      (0, 0) circle [x radius= 2.01, y radius= 2.01]   ;
\draw    (155,99.94) -- (155,129.94) ;
\draw [shift={(155,129.94)}, rotate = 90] [color={rgb, 255:red, 0; green, 0; blue, 0 }  ][fill={rgb, 255:red, 0; green, 0; blue, 0 }  ][line width=0.75]      (0, 0) circle [x radius= 2.01, y radius= 2.01]   ;
\draw [shift={(155,99.94)}, rotate = 90] [color={rgb, 255:red, 0; green, 0; blue, 0 }  ][fill={rgb, 255:red, 0; green, 0; blue, 0 }  ][line width=0.75]      (0, 0) circle [x radius= 2.01, y radius= 2.01]   ;
\draw    (155,99.94) -- (170,129.94) ;
\draw [shift={(170,129.94)}, rotate = 63.43] [color={rgb, 255:red, 0; green, 0; blue, 0 }  ][fill={rgb, 255:red, 0; green, 0; blue, 0 }  ][line width=0.75]      (0, 0) circle [x radius= 2.01, y radius= 2.01]   ;
\draw    (220,70) -- (220,99.94) ;
\draw [shift={(220,99.94)}, rotate = 90] [color={rgb, 255:red, 0; green, 0; blue, 0 }  ][fill={rgb, 255:red, 0; green, 0; blue, 0 }  ][line width=0.75]      (0, 0) circle [x radius= 2.01, y radius= 2.01]   ;
\draw    (50,160) -- (50,190) ;
\draw [shift={(50,190)}, rotate = 90] [color={rgb, 255:red, 0; green, 0; blue, 0 }  ][fill={rgb, 255:red, 0; green, 0; blue, 0 }  ][line width=0.75]      (0, 0) circle [x radius= 2.01, y radius= 2.01]   ;
\draw  [color={rgb, 255:red, 139; green, 87; blue, 42 }  ,draw opacity=1 ][fill={rgb, 255:red, 139; green, 87; blue, 42 }  ,fill opacity=1 ] (48,70) .. controls (48,68.9) and (48.9,68) .. (50,68) .. controls (51.1,68) and (52,68.9) .. (52,70) .. controls (52,71.1) and (51.1,72) .. (50,72) .. controls (48.9,72) and (48,71.1) .. (48,70) -- cycle ;
\draw  [color={rgb, 255:red, 245; green, 166; blue, 35 }  ,draw opacity=1 ][fill={rgb, 255:red, 245; green, 166; blue, 35 }  ,fill opacity=1 ] (48,100) .. controls (48,98.9) and (48.9,98) .. (50,98) .. controls (51.1,98) and (52,98.9) .. (52,100) .. controls (52,101.1) and (51.1,102) .. (50,102) .. controls (48.9,102) and (48,101.1) .. (48,100) -- cycle ;
\draw  [color={rgb, 255:red, 245; green, 166; blue, 35 }  ,draw opacity=1 ][fill={rgb, 255:red, 245; green, 166; blue, 35 }  ,fill opacity=1 ] (153,70) .. controls (153,68.9) and (153.9,68) .. (155,68) .. controls (156.1,68) and (157,68.9) .. (157,70) .. controls (157,71.1) and (156.1,72) .. (155,72) .. controls (153.9,72) and (153,71.1) .. (153,70) -- cycle ;
\draw  [color={rgb, 255:red, 88; green, 152; blue, 17 }  ,draw opacity=1 ][fill={rgb, 255:red, 88; green, 152; blue, 17 }  ,fill opacity=1 ] (218,70) .. controls (218,68.9) and (218.9,68) .. (220,68) .. controls (221.1,68) and (222,68.9) .. (222,70) .. controls (222,71.1) and (221.1,72) .. (220,72) .. controls (218.9,72) and (218,71.1) .. (218,70) -- cycle ;
\draw  [color={rgb, 255:red, 88; green, 152; blue, 17 }  ,draw opacity=1 ][fill={rgb, 255:red, 88; green, 152; blue, 17 }  ,fill opacity=1 ] (153,99.94) .. controls (153,98.84) and (153.9,97.94) .. (155,97.94) .. controls (156.1,97.94) and (157,98.84) .. (157,99.94) .. controls (157,101.04) and (156.1,101.94) .. (155,101.94) .. controls (153.9,101.94) and (153,101.04) .. (153,99.94) -- cycle ;
\draw  [color={rgb, 255:red, 88; green, 152; blue, 17 }  ,draw opacity=1 ][fill={rgb, 255:red, 88; green, 152; blue, 17 }  ,fill opacity=1 ] (88,100) .. controls (88,98.9) and (88.9,98) .. (90,98) .. controls (91.1,98) and (92,98.9) .. (92,100) .. controls (92,101.1) and (91.1,102) .. (90,102) .. controls (88.9,102) and (88,101.1) .. (88,100) -- cycle ;
\draw  [color={rgb, 255:red, 88; green, 152; blue, 17 }  ,draw opacity=1 ][fill={rgb, 255:red, 88; green, 152; blue, 17 }  ,fill opacity=1 ] (48,130) .. controls (48,128.9) and (48.9,128) .. (50,128) .. controls (51.1,128) and (52,128.9) .. (52,130) .. controls (52,131.1) and (51.1,132) .. (50,132) .. controls (48.9,132) and (48,131.1) .. (48,130) -- cycle ;
\draw  [color={rgb, 255:red, 74; green, 144; blue, 226 }  ,draw opacity=1 ][fill={rgb, 255:red, 74; green, 144; blue, 226 }  ,fill opacity=1 ] (218,99.94) .. controls (218,98.84) and (218.9,97.94) .. (220,97.94) .. controls (221.1,97.94) and (222,98.84) .. (222,99.94) .. controls (222,101.04) and (221.1,101.94) .. (220,101.94) .. controls (218.9,101.94) and (218,101.04) .. (218,99.94) -- cycle ;
\draw  [color={rgb, 255:red, 74; green, 144; blue, 226 }  ,draw opacity=1 ][fill={rgb, 255:red, 74; green, 144; blue, 226 }  ,fill opacity=1 ] (183,100) .. controls (183,98.9) and (183.9,98) .. (185,98) .. controls (186.1,98) and (187,98.9) .. (187,100) .. controls (187,101.1) and (186.1,102) .. (185,102) .. controls (183.9,102) and (183,101.1) .. (183,100) -- cycle ;
\draw  [color={rgb, 255:red, 74; green, 144; blue, 226 }  ,draw opacity=1 ][fill={rgb, 255:red, 74; green, 144; blue, 226 }  ,fill opacity=1 ] (153,129.94) .. controls (153,128.84) and (153.9,127.94) .. (155,127.94) .. controls (156.1,127.94) and (157,128.84) .. (157,129.94) .. controls (157,131.04) and (156.1,131.94) .. (155,131.94) .. controls (153.9,131.94) and (153,131.04) .. (153,129.94) -- cycle ;
\draw  [color={rgb, 255:red, 74; green, 144; blue, 226 }  ,draw opacity=1 ][fill={rgb, 255:red, 74; green, 144; blue, 226 }  ,fill opacity=1 ] (88,130) .. controls (88,128.9) and (88.9,128) .. (90,128) .. controls (91.1,128) and (92,128.9) .. (92,130) .. controls (92,131.1) and (91.1,132) .. (90,132) .. controls (88.9,132) and (88,131.1) .. (88,130) -- cycle ;
\draw  [color={rgb, 255:red, 74; green, 144; blue, 226 }  ,draw opacity=1 ][fill={rgb, 255:red, 74; green, 144; blue, 226 }  ,fill opacity=1 ] (63,130) .. controls (63,128.9) and (63.9,128) .. (65,128) .. controls (66.1,128) and (67,128.9) .. (67,130) .. controls (67,131.1) and (66.1,132) .. (65,132) .. controls (63.9,132) and (63,131.1) .. (63,130) -- cycle ;
\draw  [color={rgb, 255:red, 74; green, 144; blue, 226 }  ,draw opacity=1 ][fill={rgb, 255:red, 74; green, 144; blue, 226 }  ,fill opacity=1 ] (48,160) .. controls (48,158.9) and (48.9,158) .. (50,158) .. controls (51.1,158) and (52,158.9) .. (52,160) .. controls (52,161.1) and (51.1,162) .. (50,162) .. controls (48.9,162) and (48,161.1) .. (48,160) -- cycle ;
\draw    (50,40) -- (270,70) ;
\draw [shift={(270,70)}, rotate = 7.77] [color={rgb, 255:red, 0; green, 0; blue, 0 }  ][fill={rgb, 255:red, 0; green, 0; blue, 0 }  ][line width=0.75]      (0, 0) circle [x radius= 2.01, y radius= 2.01]   ;
\draw  [color={rgb, 255:red, 74; green, 144; blue, 226 }  ,draw opacity=1 ][fill={rgb, 255:red, 74; green, 144; blue, 226 }  ,fill opacity=1 ] (268,70) .. controls (268,68.9) and (268.9,68) .. (270,68) .. controls (271.1,68) and (272,68.9) .. (272,70) .. controls (272,71.1) and (271.1,72) .. (270,72) .. controls (268.9,72) and (268,71.1) .. (268,70) -- cycle ;
\draw    (530,100) -- (530,128) ;
\draw [shift={(530,130)}, rotate = 270] [fill={rgb, 255:red, 0; green, 0; blue, 0 }  ][line width=0.08]  [draw opacity=0] (7.2,-1.8) -- (0,0) -- (7.2,1.8) -- cycle    ;
\draw    (530,100) -- (551,100) ;
\draw [shift={(553,100)}, rotate = 180] [fill={rgb, 255:red, 0; green, 0; blue, 0 }  ][line width=0.08]  [draw opacity=0] (7.2,-1.8) -- (0,0) -- (7.2,1.8) -- cycle    ;
\draw    (580,70) -- (580,97.94) ;
\draw [shift={(580,99.94)}, rotate = 270] [fill={rgb, 255:red, 0; green, 0; blue, 0 }  ][line width=0.08]  [draw opacity=0] (7.2,-1.8) -- (0,0) -- (7.2,1.8) -- cycle    ;
\draw    (495,100) -- (495,128) ;
\draw [shift={(495,130)}, rotate = 270] [fill={rgb, 255:red, 0; green, 0; blue, 0 }  ][line width=0.08]  [draw opacity=0] (7.2,-1.8) -- (0,0) -- (7.2,1.8) -- cycle    ;
\draw    (465,70) -- (465,97.94) ;
\draw [shift={(465,99.94)}, rotate = 270] [fill={rgb, 255:red, 0; green, 0; blue, 0 }  ][line width=0.08]  [draw opacity=0] (7.2,-1.8) -- (0,0) -- (7.2,1.8) -- cycle    ;
\draw    (360,70) -- (463,70) ;
\draw [shift={(465,70)}, rotate = 180] [fill={rgb, 255:red, 0; green, 0; blue, 0 }  ][line width=0.08]  [draw opacity=0] (7.2,-1.8) -- (0,0) -- (7.2,1.8) -- cycle    ;
\draw    (360,40) -- (360,68) ;
\draw [shift={(360,70)}, rotate = 270] [fill={rgb, 255:red, 0; green, 0; blue, 0 }  ][line width=0.08]  [draw opacity=0] (7.2,-1.8) -- (0,0) -- (7.2,1.8) -- cycle    ;
\draw [shift={(360,40)}, rotate = 90] [color={rgb, 255:red, 0; green, 0; blue, 0 }  ][fill={rgb, 255:red, 0; green, 0; blue, 0 }  ][line width=0.75]      (0, 0) circle [x radius= 2.01, y radius= 2.01]   ;
\draw    (360,100) -- (403,100) ;
\draw [shift={(405,100)}, rotate = 180] [fill={rgb, 255:red, 0; green, 0; blue, 0 }  ][line width=0.08]  [draw opacity=0] (7.2,-1.8) -- (0,0) -- (7.2,1.8) -- cycle    ;
\draw    (400,100) -- (428,100) ;
\draw [shift={(430,100)}, rotate = 180] [fill={rgb, 255:red, 0; green, 0; blue, 0 }  ][line width=0.08]  [draw opacity=0] (7.2,-1.8) -- (0,0) -- (7.2,1.8) -- cycle    ;
\draw    (360,70) -- (360,98) ;
\draw [shift={(360,100)}, rotate = 270] [fill={rgb, 255:red, 0; green, 0; blue, 0 }  ][line width=0.08]  [draw opacity=0] (7.2,-1.8) -- (0,0) -- (7.2,1.8) -- cycle    ;
\draw    (380,130) -- (393,130) ;
\draw [shift={(395,130)}, rotate = 180] [fill={rgb, 255:red, 0; green, 0; blue, 0 }  ][line width=0.08]  [draw opacity=0] (7.2,-1.8) -- (0,0) -- (7.2,1.8) -- cycle    ;
\draw    (360,130) -- (378,130) ;
\draw [shift={(380,130)}, rotate = 180] [fill={rgb, 255:red, 0; green, 0; blue, 0 }  ][line width=0.08]  [draw opacity=0] (7.2,-1.8) -- (0,0) -- (7.2,1.8) -- cycle    ;
\draw    (360,100) -- (360,128) ;
\draw [shift={(360,130)}, rotate = 270] [fill={rgb, 255:red, 0; green, 0; blue, 0 }  ][line width=0.08]  [draw opacity=0] (7.2,-1.8) -- (0,0) -- (7.2,1.8) -- cycle    ;
\draw    (465,99.94) -- (493,100) ;
\draw [shift={(495,100)}, rotate = 180.11] [fill={rgb, 255:red, 0; green, 0; blue, 0 }  ][line width=0.08]  [draw opacity=0] (7.2,-1.8) -- (0,0) -- (7.2,1.8) -- cycle    ;
\draw    (405,130) -- (418,130) ;
\draw [shift={(420,130)}, rotate = 180] [fill={rgb, 255:red, 0; green, 0; blue, 0 }  ][line width=0.08]  [draw opacity=0] (7.2,-1.8) -- (0,0) -- (7.2,1.8) -- cycle    ;
\draw    (405,100) -- (405,128) ;
\draw [shift={(405,130)}, rotate = 270] [fill={rgb, 255:red, 0; green, 0; blue, 0 }  ][line width=0.08]  [draw opacity=0] (7.2,-1.8) -- (0,0) -- (7.2,1.8) -- cycle    ;
\draw [shift={(405,100)}, rotate = 90] [color={rgb, 255:red, 0; green, 0; blue, 0 }  ][fill={rgb, 255:red, 0; green, 0; blue, 0 }  ][line width=0.75]      (0, 0) circle [x radius= 2.01, y radius= 2.01]   ;
\draw    (430,100) -- (430,128) ;
\draw [shift={(430,130)}, rotate = 270] [fill={rgb, 255:red, 0; green, 0; blue, 0 }  ][line width=0.08]  [draw opacity=0] (7.2,-1.8) -- (0,0) -- (7.2,1.8) -- cycle    ;
\draw [shift={(430,100)}, rotate = 90] [color={rgb, 255:red, 0; green, 0; blue, 0 }  ][fill={rgb, 255:red, 0; green, 0; blue, 0 }  ][line width=0.75]      (0, 0) circle [x radius= 2.01, y radius= 2.01]   ;
\draw    (360,160) -- (373,160) ;
\draw [shift={(375,160)}, rotate = 180] [fill={rgb, 255:red, 0; green, 0; blue, 0 }  ][line width=0.08]  [draw opacity=0] (7.2,-1.8) -- (0,0) -- (7.2,1.8) -- cycle    ;
\draw    (360,130) -- (360,158) ;
\draw [shift={(360,160)}, rotate = 270] [fill={rgb, 255:red, 0; green, 0; blue, 0 }  ][line width=0.08]  [draw opacity=0] (7.2,-1.8) -- (0,0) -- (7.2,1.8) -- cycle    ;
\draw    (465,70) -- (528,70) ;
\draw [shift={(530,70)}, rotate = 180] [fill={rgb, 255:red, 0; green, 0; blue, 0 }  ][line width=0.08]  [draw opacity=0] (7.2,-1.8) -- (0,0) -- (7.2,1.8) -- cycle    ;
\draw    (465,99.94) -- (465,129.94) ;
\draw [shift={(465,129.94)}, rotate = 90] [color={rgb, 255:red, 0; green, 0; blue, 0 }  ][fill={rgb, 255:red, 0; green, 0; blue, 0 }  ][line width=0.75]      (0, 0) circle [x radius= 2.01, y radius= 2.01]   ;
\draw [shift={(465,99.94)}, rotate = 90] [color={rgb, 255:red, 0; green, 0; blue, 0 }  ][fill={rgb, 255:red, 0; green, 0; blue, 0 }  ][line width=0.75]      (0, 0) circle [x radius= 2.01, y radius= 2.01]   ;
\draw    (465,129.94) -- (478,129.94) ;
\draw [shift={(480,129.94)}, rotate = 180] [fill={rgb, 255:red, 0; green, 0; blue, 0 }  ][line width=0.08]  [draw opacity=0] (7.2,-1.8) -- (0,0) -- (7.2,1.8) -- cycle    ;
\draw    (530,70) -- (530,97.94) ;
\draw [shift={(530,99.94)}, rotate = 270] [fill={rgb, 255:red, 0; green, 0; blue, 0 }  ][line width=0.08]  [draw opacity=0] (7.2,-1.8) -- (0,0) -- (7.2,1.8) -- cycle    ;
\draw    (360,160) -- (360,188) ;
\draw [shift={(360,190)}, rotate = 270] [fill={rgb, 255:red, 0; green, 0; blue, 0 }  ][line width=0.08]  [draw opacity=0] (7.2,-1.8) -- (0,0) -- (7.2,1.8) -- cycle    ;
\draw  [color={rgb, 255:red, 139; green, 87; blue, 42 }  ,draw opacity=1 ][fill={rgb, 255:red, 139; green, 87; blue, 42 }  ,fill opacity=1 ] (358,70) .. controls (358,68.9) and (358.9,68) .. (360,68) .. controls (361.1,68) and (362,68.9) .. (362,70) .. controls (362,71.1) and (361.1,72) .. (360,72) .. controls (358.9,72) and (358,71.1) .. (358,70) -- cycle ;
\draw  [color={rgb, 255:red, 245; green, 166; blue, 35 }  ,draw opacity=1 ][fill={rgb, 255:red, 245; green, 166; blue, 35 }  ,fill opacity=1 ] (358,100) .. controls (358,98.9) and (358.9,98) .. (360,98) .. controls (361.1,98) and (362,98.9) .. (362,100) .. controls (362,101.1) and (361.1,102) .. (360,102) .. controls (358.9,102) and (358,101.1) .. (358,100) -- cycle ;
\draw  [color={rgb, 255:red, 245; green, 166; blue, 35 }  ,draw opacity=1 ][fill={rgb, 255:red, 245; green, 166; blue, 35 }  ,fill opacity=1 ] (463,70) .. controls (463,68.9) and (463.9,68) .. (465,68) .. controls (466.1,68) and (467,68.9) .. (467,70) .. controls (467,71.1) and (466.1,72) .. (465,72) .. controls (463.9,72) and (463,71.1) .. (463,70) -- cycle ;
\draw  [color={rgb, 255:red, 88; green, 152; blue, 17 }  ,draw opacity=1 ][fill={rgb, 255:red, 88; green, 152; blue, 17 }  ,fill opacity=1 ] (528,70) .. controls (528,68.9) and (528.9,68) .. (530,68) .. controls (531.1,68) and (532,68.9) .. (532,70) .. controls (532,71.1) and (531.1,72) .. (530,72) .. controls (528.9,72) and (528,71.1) .. (528,70) -- cycle ;
\draw  [color={rgb, 255:red, 88; green, 152; blue, 17 }  ,draw opacity=1 ][fill={rgb, 255:red, 88; green, 152; blue, 17 }  ,fill opacity=1 ] (463,99.94) .. controls (463,98.84) and (463.9,97.94) .. (465,97.94) .. controls (466.1,97.94) and (467,98.84) .. (467,99.94) .. controls (467,101.04) and (466.1,101.94) .. (465,101.94) .. controls (463.9,101.94) and (463,101.04) .. (463,99.94) -- cycle ;
\draw  [color={rgb, 255:red, 88; green, 152; blue, 17 }  ,draw opacity=1 ][fill={rgb, 255:red, 88; green, 152; blue, 17 }  ,fill opacity=1 ] (403,100) .. controls (403,98.9) and (403.9,98) .. (405,98) .. controls (406.1,98) and (407,98.9) .. (407,100) .. controls (407,101.1) and (406.1,102) .. (405,102) .. controls (403.9,102) and (403,101.1) .. (403,100) -- cycle ;
\draw  [color={rgb, 255:red, 88; green, 152; blue, 17 }  ,draw opacity=1 ][fill={rgb, 255:red, 88; green, 152; blue, 17 }  ,fill opacity=1 ] (358,130) .. controls (358,128.9) and (358.9,128) .. (360,128) .. controls (361.1,128) and (362,128.9) .. (362,130) .. controls (362,131.1) and (361.1,132) .. (360,132) .. controls (358.9,132) and (358,131.1) .. (358,130) -- cycle ;
\draw  [color={rgb, 255:red, 74; green, 144; blue, 226 }  ,draw opacity=1 ][fill={rgb, 255:red, 74; green, 144; blue, 226 }  ,fill opacity=1 ] (528,99.94) .. controls (528,98.84) and (528.9,97.94) .. (530,97.94) .. controls (531.1,97.94) and (532,98.84) .. (532,99.94) .. controls (532,101.04) and (531.1,101.94) .. (530,101.94) .. controls (528.9,101.94) and (528,101.04) .. (528,99.94) -- cycle ;
\draw  [color={rgb, 255:red, 74; green, 144; blue, 226 }  ,draw opacity=1 ][fill={rgb, 255:red, 74; green, 144; blue, 226 }  ,fill opacity=1 ] (493,100) .. controls (493,98.9) and (493.9,98) .. (495,98) .. controls (496.1,98) and (497,98.9) .. (497,100) .. controls (497,101.1) and (496.1,102) .. (495,102) .. controls (493.9,102) and (493,101.1) .. (493,100) -- cycle ;
\draw  [color={rgb, 255:red, 74; green, 144; blue, 226 }  ,draw opacity=1 ][fill={rgb, 255:red, 74; green, 144; blue, 226 }  ,fill opacity=1 ] (463,129.94) .. controls (463,128.84) and (463.9,127.94) .. (465,127.94) .. controls (466.1,127.94) and (467,128.84) .. (467,129.94) .. controls (467,131.04) and (466.1,131.94) .. (465,131.94) .. controls (463.9,131.94) and (463,131.04) .. (463,129.94) -- cycle ;
\draw  [color={rgb, 255:red, 74; green, 144; blue, 226 }  ,draw opacity=1 ][fill={rgb, 255:red, 74; green, 144; blue, 226 }  ,fill opacity=1 ] (403,130) .. controls (403,128.9) and (403.9,128) .. (405,128) .. controls (406.1,128) and (407,128.9) .. (407,130) .. controls (407,131.1) and (406.1,132) .. (405,132) .. controls (403.9,132) and (403,131.1) .. (403,130) -- cycle ;
\draw  [color={rgb, 255:red, 74; green, 144; blue, 226 }  ,draw opacity=1 ][fill={rgb, 255:red, 74; green, 144; blue, 226 }  ,fill opacity=1 ] (378,130) .. controls (378,128.9) and (378.9,128) .. (380,128) .. controls (381.1,128) and (382,128.9) .. (382,130) .. controls (382,131.1) and (381.1,132) .. (380,132) .. controls (378.9,132) and (378,131.1) .. (378,130) -- cycle ;
\draw  [color={rgb, 255:red, 74; green, 144; blue, 226 }  ,draw opacity=1 ][fill={rgb, 255:red, 74; green, 144; blue, 226 }  ,fill opacity=1 ] (358,160) .. controls (358,158.9) and (358.9,158) .. (360,158) .. controls (361.1,158) and (362,158.9) .. (362,160) .. controls (362,161.1) and (361.1,162) .. (360,162) .. controls (358.9,162) and (358,161.1) .. (358,160) -- cycle ;
\draw  [color={rgb, 255:red, 74; green, 144; blue, 226 }  ,draw opacity=1 ][fill={rgb, 255:red, 74; green, 144; blue, 226 }  ,fill opacity=1 ] (578,70) .. controls (578,68.9) and (578.9,68) .. (580,68) .. controls (581.1,68) and (582,68.9) .. (582,70) .. controls (582,71.1) and (581.1,72) .. (580,72) .. controls (578.9,72) and (578,71.1) .. (578,70) -- cycle ;
\draw  [color={rgb, 255:red, 74; green, 144; blue, 226 }  ,draw opacity=1 ][fill={rgb, 255:red, 74; green, 144; blue, 226 }  ,fill opacity=1 ] (428,100) .. controls (428,98.9) and (428.9,98) .. (430,98) .. controls (431.1,98) and (432,98.9) .. (432,100) .. controls (432,101.1) and (431.1,102) .. (430,102) .. controls (428.9,102) and (428,101.1) .. (428,100) -- cycle ;

\draw (89.33,80.4) node [anchor=north west][inner sep=0.75pt]  [font=\small]  {$\cdots $};
\draw (65,111.4) node [anchor=north west][inner sep=0.75pt]  [font=\small]  {$\cdots $};
\draw (98,111.4) node [anchor=north west][inner sep=0.75pt]  [font=\small]  {$\cdots $};
\draw (122,111.4) node [anchor=north west][inner sep=0.75pt]  [font=\small]  {$\cdots $};
\draw (74,145.9) node [anchor=north west][inner sep=0.75pt]  [font=\small]  {$\cdots $};
\draw (34,34.4) node [anchor=north west][inner sep=0.75pt]  [font=\small,color={rgb, 255:red, 128; green, 128; blue, 128 }  ,opacity=1 ]  {$\emptyset $};
\draw (35,64.4) node [anchor=north west][inner sep=0.75pt]  [font=\small,color={rgb, 255:red, 139; green, 87; blue, 42 }  ,opacity=1 ]  {$1$};
\draw (30,94.4) node [anchor=north west][inner sep=0.75pt]  [font=\small,color={rgb, 255:red, 245; green, 166; blue, 35 }  ,opacity=1 ]  {$11$};
\draw (22,124.4) node [anchor=north west][inner sep=0.75pt]  [font=\small,color={rgb, 255:red, 88; green, 152; blue, 17 }  ,opacity=1 ]  {$111$};
\draw (159,64.4) node [anchor=north west][inner sep=0.75pt]  [font=\small,color={rgb, 255:red, 245; green, 166; blue, 35 }  ,opacity=1 ]  {$2$};
\draw (70,94.4) node [anchor=north west][inner sep=0.75pt]  [font=\small,color={rgb, 255:red, 88; green, 152; blue, 17 }  ,opacity=1 ]  {$12$};
\draw (100,94.4) node [anchor=north west][inner sep=0.75pt]  [font=\small,color={rgb, 255:red, 74; green, 144; blue, 226 }  ,opacity=1 ]  {$13$};
\draw (135,94.4) node [anchor=north west][inner sep=0.75pt]  [font=\small,color={rgb, 255:red, 88; green, 152; blue, 17 }  ,opacity=1 ]  {$21$};
\draw (15,154.4) node [anchor=north west][inner sep=0.75pt]  [font=\small,color={rgb, 255:red, 74; green, 144; blue, 226 }  ,opacity=1 ]  {$1111$};
\draw (165,94.4) node [anchor=north west][inner sep=0.75pt]  [font=\small,color={rgb, 255:red, 74; green, 144; blue, 226 }  ,opacity=1 ]  {$22$};
\draw (172.33,80.4) node [anchor=north west][inner sep=0.75pt]  [font=\small]  {$\cdots $};
\draw (224,64) node [anchor=north west][inner sep=0.75pt]  [font=\small,color={rgb, 255:red, 88; green, 152; blue, 17 }  ,opacity=1 ] [align=left] {$\displaystyle 3$};
\draw (54,134.4) node [anchor=north west][inner sep=0.75pt]  [font=\footnotesize,color={rgb, 255:red, 74; green, 144; blue, 226 }  ,opacity=1 ]  {$112$};
\draw (80,134.4) node [anchor=north west][inner sep=0.75pt]  [font=\footnotesize,color={rgb, 255:red, 74; green, 144; blue, 226 }  ,opacity=1 ]  {$121$};
\draw (145,134.4) node [anchor=north west][inner sep=0.75pt]  [font=\footnotesize,color={rgb, 255:red, 74; green, 144; blue, 226 }  ,opacity=1 ]  {$211$};
\draw (200,94.4) node [anchor=north west][inner sep=0.75pt]  [font=\small,color={rgb, 255:red, 74; green, 144; blue, 226 }  ,opacity=1 ]  {$31$};
\draw (228.33,80.4) node [anchor=north west][inner sep=0.75pt]  [font=\small]  {$\cdots $};
\draw (163,111.4) node [anchor=north west][inner sep=0.75pt]  [font=\small]  {$\cdots $};
\draw (52,170.4) node [anchor=north west][inner sep=0.75pt]  [font=\small]  {$\cdots $};
\draw (272.33,80.4) node [anchor=north west][inner sep=0.75pt]  [font=\small]  {$\cdots $};
\draw (274,64) node [anchor=north west][inner sep=0.75pt]  [font=\small,color={rgb, 255:red, 74; green, 144; blue, 226 }  ,opacity=1 ] [align=left] {$\displaystyle 4$};
\draw (187.33,111.4) node [anchor=north west][inner sep=0.75pt]  [font=\small]  {$\cdots $};
\draw (223.33,111.4) node [anchor=north west][inner sep=0.75pt]  [font=\small]  {$\cdots $};
\draw (195.33,48.4) node [anchor=north west][inner sep=0.75pt]  [font=\small]  {$\cdots $};
\draw (419.43,133.4) node [anchor=north west][inner sep=0.75pt]  [font=\small]  {$\cdots $};
\draw (434,113.4) node [anchor=north west][inner sep=0.75pt]  [font=\small]  {$\cdots $};
\draw (381,154.9) node [anchor=north west][inner sep=0.75pt]  [font=\small]  {$\cdots $};
\draw (344,34.4) node [anchor=north west][inner sep=0.75pt]  [font=\small,color={rgb, 255:red, 128; green, 128; blue, 128 }  ,opacity=1 ]  {$\emptyset $};
\draw (345,64.4) node [anchor=north west][inner sep=0.75pt]  [font=\small,color={rgb, 255:red, 139; green, 87; blue, 42 }  ,opacity=1 ]  {$1$};
\draw (340,94.4) node [anchor=north west][inner sep=0.75pt]  [font=\small,color={rgb, 255:red, 245; green, 166; blue, 35 }  ,opacity=1 ]  {$11$};
\draw (332,124.4) node [anchor=north west][inner sep=0.75pt]  [font=\small,color={rgb, 255:red, 88; green, 152; blue, 17 }  ,opacity=1 ]  {$111$};
\draw (457,54.4) node [anchor=north west][inner sep=0.75pt]  [font=\small,color={rgb, 255:red, 245; green, 166; blue, 35 }  ,opacity=1 ]  {$10$};
\draw (389,85.4) node [anchor=north west][inner sep=0.75pt]  [font=\small,color={rgb, 255:red, 88; green, 152; blue, 17 }  ,opacity=1 ]  {$110$};
\draw (420,86.4) node [anchor=north west][inner sep=0.75pt]  [font=\footnotesize,color={rgb, 255:red, 74; green, 144; blue, 226 }  ,opacity=1 ]  {$1100$};
\draw (441,102.4) node [anchor=north west][inner sep=0.75pt]  [font=\small,color={rgb, 255:red, 88; green, 152; blue, 17 }  ,opacity=1 ]  {$101$};
\draw (325,154.4) node [anchor=north west][inner sep=0.75pt]  [font=\small,color={rgb, 255:red, 74; green, 144; blue, 226 }  ,opacity=1 ]  {$1111$};
\draw (484,87.4) node [anchor=north west][inner sep=0.75pt]  [font=\footnotesize,color={rgb, 255:red, 74; green, 144; blue, 226 }  ,opacity=1 ]  {$1010$};
\draw (498.33,113.4) node [anchor=north west][inner sep=0.75pt]  [font=\small]  {$\cdots $};
\draw (518,55) node [anchor=north west][inner sep=0.75pt]  [font=\small,color={rgb, 255:red, 88; green, 152; blue, 17 }  ,opacity=1 ] [align=left] {$\displaystyle 100$};
\draw (367,116.4) node [anchor=north west][inner sep=0.75pt]  [font=\footnotesize,color={rgb, 255:red, 74; green, 144; blue, 226 }  ,opacity=1 ]  {$1110$};
\draw (393,135.4) node [anchor=north west][inner sep=0.75pt]  [font=\footnotesize,color={rgb, 255:red, 74; green, 144; blue, 226 }  ,opacity=1 ]  {$1101$};
\draw (455,135.4) node [anchor=north west][inner sep=0.75pt]  [font=\footnotesize,color={rgb, 255:red, 74; green, 144; blue, 226 }  ,opacity=1 ]  {$1011$};
\draw (533,87.4) node [anchor=north west][inner sep=0.75pt]  [font=\footnotesize,color={rgb, 255:red, 74; green, 144; blue, 226 }  ,opacity=1 ]  {$1001$};
\draw (485,133.4) node [anchor=north west][inner sep=0.75pt]  [font=\small]  {$\cdots $};
\draw (362,173.4) node [anchor=north west][inner sep=0.75pt]  [font=\small]  {$\cdots $};
\draw (565,55) node [anchor=north west][inner sep=0.75pt]  [font=\small,color={rgb, 255:red, 74; green, 144; blue, 226 }  ,opacity=1 ] [align=left] {$\displaystyle 1000$};
\draw (545.33,113.4) node [anchor=north west][inner sep=0.75pt]  [font=\small]  {$\cdots $};
\draw (585.48,75.4) node [anchor=north west][inner sep=0.75pt]  [font=\small]  {$\cdots $};
\draw (362,51.4) node [anchor=north west][inner sep=0.75pt]  [font=\tiny,color={rgb, 255:red, 155; green, 155; blue, 155 }  ,opacity=1 ]  {$1$};
\draw (409,61.4) node [anchor=north west][inner sep=0.75pt]  [font=\tiny,color={rgb, 255:red, 155; green, 155; blue, 155 }  ,opacity=1 ]  {$0$};
\draw (493,61.4) node [anchor=north west][inner sep=0.75pt]  [font=\tiny,color={rgb, 255:red, 155; green, 155; blue, 155 }  ,opacity=1 ]  {$0$};
\draw (362,81.4) node [anchor=north west][inner sep=0.75pt]  [font=\tiny,color={rgb, 255:red, 155; green, 155; blue, 155 }  ,opacity=1 ]  {$1$};

\end{tikzpicture}

%% file: ulam-bijection.tex
\begin{tikzpicture}[x=0.75pt,y=0.75pt,yscale=-1,xscale=1]

\draw  [color={rgb, 255:red, 192; green, 192; blue, 192 }  ,draw opacity=1 ][fill={rgb, 255:red, 230; green, 237; blue, 247 }  ,fill opacity=1 ] (530,70) .. controls (537.6,64.6) and (578.25,61.05) .. (588.25,71.55) .. controls (598.25,82.05) and (590.25,106.55) .. (575.25,123.05) .. controls (560.25,139.55) and (547.75,144.55) .. (533.25,137.55) .. controls (518.75,130.55) and (522.4,75.4) .. (530,70) -- cycle ;
\draw  [color={rgb, 255:red, 192; green, 192; blue, 192 }  ,draw opacity=1 ][fill={rgb, 255:red, 247; green, 242; blue, 214 }  ,fill opacity=1 ] (220,70) .. controls (227.6,64.6) and (268.25,61.05) .. (278.25,71.55) .. controls (288.25,82.05) and (280.25,106.55) .. (265.25,123.05) .. controls (250.25,139.55) and (237.75,144.55) .. (223.25,137.55) .. controls (208.75,130.55) and (212.4,75.4) .. (220,70) -- cycle ;
\draw  [color={rgb, 255:red, 192; green, 192; blue, 192 }  ,draw opacity=1 ][fill={rgb, 255:red, 242; green, 242; blue, 242 }  ,fill opacity=1 ] (465,99.94) .. controls (472.6,94.54) and (500.6,95.85) .. (513,105.05) .. controls (525.4,114.25) and (521,134.05) .. (510,144.05) .. controls (499,154.05) and (478.5,162.55) .. (470,154.55) .. controls (461.5,146.55) and (457.4,105.34) .. (465,99.94) -- cycle ;
\draw  [color={rgb, 255:red, 192; green, 192; blue, 192 }  ,draw opacity=1 ][fill={rgb, 255:red, 247; green, 242; blue, 214 }  ,fill opacity=1 ] (360,130) .. controls (367.6,124.6) and (387.5,126.05) .. (393,137.55) .. controls (398.5,149.05) and (400.5,163.55) .. (394.5,177.05) .. controls (388.5,190.55) and (371.5,205.5) .. (360.5,193.05) .. controls (349.5,180.6) and (352.4,135.4) .. (360,130) -- cycle ;
\draw  [color={rgb, 255:red, 192; green, 192; blue, 192 }  ,draw opacity=1 ][fill={rgb, 255:red, 235; green, 242; blue, 227 }  ,fill opacity=1 ] (405,100) .. controls (412.6,94.6) and (437.5,93.05) .. (441,105.05) .. controls (444.5,117.05) and (445,134.05) .. (439,147.55) .. controls (433,161.05) and (418.5,167.05) .. (408.5,157.55) .. controls (398.5,148.05) and (397.4,105.4) .. (405,100) -- cycle ;
\draw    (380,130) -- (380,148) ;
\draw [shift={(380,150)}, rotate = 270] [fill={rgb, 255:red, 0; green, 0; blue, 0 }  ][line width=0.08]  [draw opacity=0] (7.2,-1.8) -- (0,0) -- (7.2,1.8) -- cycle    ;
\draw  [color={rgb, 255:red, 192; green, 192; blue, 192 }  ,draw opacity=1 ][fill={rgb, 255:red, 230; green, 237; blue, 247 }  ,fill opacity=1 ] (50,130) .. controls (57.6,124.6) and (77.5,126.05) .. (83,137.55) .. controls (88.5,149.05) and (90.5,163.55) .. (84.5,177.05) .. controls (78.5,190.55) and (61.5,205.5) .. (50.5,193.05) .. controls (39.5,180.6) and (42.4,135.4) .. (50,130) -- cycle ;
\draw    (70,130) -- (70,148) ;
\draw [shift={(70,150)}, rotate = 270] [fill={rgb, 255:red, 0; green, 0; blue, 0 }  ][line width=0.08]  [draw opacity=0] (7.2,-1.8) -- (0,0) -- (7.2,1.8) -- cycle    ;
\draw  [color={rgb, 255:red, 192; green, 192; blue, 192 }  ,draw opacity=1 ][fill={rgb, 255:red, 242; green, 242; blue, 242 }  ,fill opacity=1 ] (95,100) .. controls (102.6,94.6) and (127.5,93.05) .. (131,105.05) .. controls (134.5,117.05) and (135,134.05) .. (129,147.55) .. controls (123,161.05) and (108.5,167.05) .. (98.5,157.55) .. controls (88.5,148.05) and (87.4,105.4) .. (95,100) -- cycle ;
\draw  [color={rgb, 255:red, 192; green, 192; blue, 192 }  ,draw opacity=1 ][fill={rgb, 255:red, 235; green, 242; blue, 227 }  ,fill opacity=1 ] (155,99.94) .. controls (162.6,94.54) and (190.6,95.85) .. (203,105.05) .. controls (215.4,114.25) and (211,134.05) .. (200,144.05) .. controls (189,154.05) and (168.5,162.55) .. (160,154.55) .. controls (151.5,146.55) and (147.4,105.34) .. (155,99.94) -- cycle ;
\draw    (530,70) -- (578,70) ;
\draw [shift={(580,70)}, rotate = 180] [fill={rgb, 255:red, 0; green, 0; blue, 0 }  ][line width=0.08]  [draw opacity=0] (7.2,-1.8) -- (0,0) -- (7.2,1.8) -- cycle    ;
\draw    (530,100) -- (530,128) ;
\draw [shift={(530,130)}, rotate = 270] [fill={rgb, 255:red, 0; green, 0; blue, 0 }  ][line width=0.08]  [draw opacity=0] (7.2,-1.8) -- (0,0) -- (7.2,1.8) -- cycle    ;
\draw    (530,100) -- (551,100) ;
\draw [shift={(553,100)}, rotate = 180] [fill={rgb, 255:red, 0; green, 0; blue, 0 }  ][line width=0.08]  [draw opacity=0] (7.2,-1.8) -- (0,0) -- (7.2,1.8) -- cycle    ;
\draw    (580,70) -- (580,97.94) ;
\draw [shift={(580,99.94)}, rotate = 270] [fill={rgb, 255:red, 0; green, 0; blue, 0 }  ][line width=0.08]  [draw opacity=0] (7.2,-1.8) -- (0,0) -- (7.2,1.8) -- cycle    ;
\draw    (495,100) -- (495,128) ;
\draw [shift={(495,130)}, rotate = 270] [fill={rgb, 255:red, 0; green, 0; blue, 0 }  ][line width=0.08]  [draw opacity=0] (7.2,-1.8) -- (0,0) -- (7.2,1.8) -- cycle    ;
\draw    (465,70) -- (465,97.94) ;
\draw [shift={(465,99.94)}, rotate = 270] [fill={rgb, 255:red, 0; green, 0; blue, 0 }  ][line width=0.08]  [draw opacity=0] (7.2,-1.8) -- (0,0) -- (7.2,1.8) -- cycle    ;
\draw    (360,70) -- (463,70) ;
\draw [shift={(465,70)}, rotate = 180] [fill={rgb, 255:red, 0; green, 0; blue, 0 }  ][line width=0.08]  [draw opacity=0] (7.2,-1.8) -- (0,0) -- (7.2,1.8) -- cycle    ;
\draw    (360,40) -- (360,68) ;
\draw [shift={(360,70)}, rotate = 270] [fill={rgb, 255:red, 0; green, 0; blue, 0 }  ][line width=0.08]  [draw opacity=0] (7.2,-1.8) -- (0,0) -- (7.2,1.8) -- cycle    ;
\draw [shift={(360,40)}, rotate = 90] [color={rgb, 255:red, 0; green, 0; blue, 0 }  ][fill={rgb, 255:red, 0; green, 0; blue, 0 }  ][line width=0.75]      (0, 0) circle [x radius= 2.01, y radius= 2.01]   ;
\draw    (360,100) -- (403,100) ;
\draw [shift={(405,100)}, rotate = 180] [fill={rgb, 255:red, 0; green, 0; blue, 0 }  ][line width=0.08]  [draw opacity=0] (7.2,-1.8) -- (0,0) -- (7.2,1.8) -- cycle    ;
\draw    (400,100) -- (428,100) ;
\draw [shift={(430,100)}, rotate = 180] [fill={rgb, 255:red, 0; green, 0; blue, 0 }  ][line width=0.08]  [draw opacity=0] (7.2,-1.8) -- (0,0) -- (7.2,1.8) -- cycle    ;
\draw    (360,70) -- (360,98) ;
\draw [shift={(360,100)}, rotate = 270] [fill={rgb, 255:red, 0; green, 0; blue, 0 }  ][line width=0.08]  [draw opacity=0] (7.2,-1.8) -- (0,0) -- (7.2,1.8) -- cycle    ;
\draw    (360,130) -- (378,130) ;
\draw [shift={(380,130)}, rotate = 180] [fill={rgb, 255:red, 0; green, 0; blue, 0 }  ][line width=0.08]  [draw opacity=0] (7.2,-1.8) -- (0,0) -- (7.2,1.8) -- cycle    ;
\draw    (360,100) -- (360,128) ;
\draw [shift={(360,130)}, rotate = 270] [fill={rgb, 255:red, 0; green, 0; blue, 0 }  ][line width=0.08]  [draw opacity=0] (7.2,-1.8) -- (0,0) -- (7.2,1.8) -- cycle    ;
\draw    (465,99.94) -- (493,100) ;
\draw [shift={(495,100)}, rotate = 180.11] [fill={rgb, 255:red, 0; green, 0; blue, 0 }  ][line width=0.08]  [draw opacity=0] (7.2,-1.8) -- (0,0) -- (7.2,1.8) -- cycle    ;
\draw    (405,130) -- (418,130) ;
\draw [shift={(420,130)}, rotate = 180] [fill={rgb, 255:red, 0; green, 0; blue, 0 }  ][line width=0.08]  [draw opacity=0] (7.2,-1.8) -- (0,0) -- (7.2,1.8) -- cycle    ;
\draw    (405,100) -- (405,128) ;
\draw [shift={(405,130)}, rotate = 270] [fill={rgb, 255:red, 0; green, 0; blue, 0 }  ][line width=0.08]  [draw opacity=0] (7.2,-1.8) -- (0,0) -- (7.2,1.8) -- cycle    ;
\draw [shift={(405,100)}, rotate = 90] [color={rgb, 255:red, 0; green, 0; blue, 0 }  ][fill={rgb, 255:red, 0; green, 0; blue, 0 }  ][line width=0.75]      (0, 0) circle [x radius= 2.01, y radius= 2.01]   ;
\draw    (430,100) -- (430,128) ;
\draw [shift={(430,130)}, rotate = 270] [fill={rgb, 255:red, 0; green, 0; blue, 0 }  ][line width=0.08]  [draw opacity=0] (7.2,-1.8) -- (0,0) -- (7.2,1.8) -- cycle    ;
\draw [shift={(430,100)}, rotate = 90] [color={rgb, 255:red, 0; green, 0; blue, 0 }  ][fill={rgb, 255:red, 0; green, 0; blue, 0 }  ][line width=0.75]      (0, 0) circle [x radius= 2.01, y radius= 2.01]   ;
\draw    (360,160) -- (373,160) ;
\draw [shift={(375,160)}, rotate = 180] [fill={rgb, 255:red, 0; green, 0; blue, 0 }  ][line width=0.08]  [draw opacity=0] (7.2,-1.8) -- (0,0) -- (7.2,1.8) -- cycle    ;
\draw    (360,130) -- (360,158) ;
\draw [shift={(360,160)}, rotate = 270] [fill={rgb, 255:red, 0; green, 0; blue, 0 }  ][line width=0.08]  [draw opacity=0] (7.2,-1.8) -- (0,0) -- (7.2,1.8) -- cycle    ;
\draw    (465,70) -- (528,70) ;
\draw [shift={(530,70)}, rotate = 180] [fill={rgb, 255:red, 0; green, 0; blue, 0 }  ][line width=0.08]  [draw opacity=0] (7.2,-1.8) -- (0,0) -- (7.2,1.8) -- cycle    ;
\draw    (465,99.94) -- (465,129.94) ;
\draw [shift={(465,129.94)}, rotate = 90] [color={rgb, 255:red, 0; green, 0; blue, 0 }  ][fill={rgb, 255:red, 0; green, 0; blue, 0 }  ][line width=0.75]      (0, 0) circle [x radius= 2.01, y radius= 2.01]   ;
\draw [shift={(465,99.94)}, rotate = 90] [color={rgb, 255:red, 0; green, 0; blue, 0 }  ][fill={rgb, 255:red, 0; green, 0; blue, 0 }  ][line width=0.75]      (0, 0) circle [x radius= 2.01, y radius= 2.01]   ;
\draw    (465,129.94) -- (478,129.94) ;
\draw [shift={(480,129.94)}, rotate = 180] [fill={rgb, 255:red, 0; green, 0; blue, 0 }  ][line width=0.08]  [draw opacity=0] (7.2,-1.8) -- (0,0) -- (7.2,1.8) -- cycle    ;
\draw    (530,70) -- (530,97.94) ;
\draw [shift={(530,99.94)}, rotate = 270] [fill={rgb, 255:red, 0; green, 0; blue, 0 }  ][line width=0.08]  [draw opacity=0] (7.2,-1.8) -- (0,0) -- (7.2,1.8) -- cycle    ;
\draw    (360,160) -- (360,188) ;
\draw [shift={(360,190)}, rotate = 270] [fill={rgb, 255:red, 0; green, 0; blue, 0 }  ][line width=0.08]  [draw opacity=0] (7.2,-1.8) -- (0,0) -- (7.2,1.8) -- cycle    ;
\draw  [color={rgb, 255:red, 139; green, 87; blue, 42 }  ,draw opacity=1 ][fill={rgb, 255:red, 139; green, 87; blue, 42 }  ,fill opacity=1 ] (358,70) .. controls (358,68.9) and (358.9,68) .. (360,68) .. controls (361.1,68) and (362,68.9) .. (362,70) .. controls (362,71.1) and (361.1,72) .. (360,72) .. controls (358.9,72) and (358,71.1) .. (358,70) -- cycle ;
\draw  [color={rgb, 255:red, 245; green, 166; blue, 35 }  ,draw opacity=1 ][fill={rgb, 255:red, 245; green, 166; blue, 35 }  ,fill opacity=1 ] (358,100) .. controls (358,98.9) and (358.9,98) .. (360,98) .. controls (361.1,98) and (362,98.9) .. (362,100) .. controls (362,101.1) and (361.1,102) .. (360,102) .. controls (358.9,102) and (358,101.1) .. (358,100) -- cycle ;
\draw  [color={rgb, 255:red, 245; green, 166; blue, 35 }  ,draw opacity=1 ][fill={rgb, 255:red, 245; green, 166; blue, 35 }  ,fill opacity=1 ] (463,70) .. controls (463,68.9) and (463.9,68) .. (465,68) .. controls (466.1,68) and (467,68.9) .. (467,70) .. controls (467,71.1) and (466.1,72) .. (465,72) .. controls (463.9,72) and (463,71.1) .. (463,70) -- cycle ;
\draw  [color={rgb, 255:red, 88; green, 152; blue, 17 }  ,draw opacity=1 ][fill={rgb, 255:red, 88; green, 152; blue, 17 }  ,fill opacity=1 ] (528,70) .. controls (528,68.9) and (528.9,68) .. (530,68) .. controls (531.1,68) and (532,68.9) .. (532,70) .. controls (532,71.1) and (531.1,72) .. (530,72) .. controls (528.9,72) and (528,71.1) .. (528,70) -- cycle ;
\draw  [color={rgb, 255:red, 88; green, 152; blue, 17 }  ,draw opacity=1 ][fill={rgb, 255:red, 88; green, 152; blue, 17 }  ,fill opacity=1 ] (463,99.94) .. controls (463,98.84) and (463.9,97.94) .. (465,97.94) .. controls (466.1,97.94) and (467,98.84) .. (467,99.94) .. controls (467,101.04) and (466.1,101.94) .. (465,101.94) .. controls (463.9,101.94) and (463,101.04) .. (463,99.94) -- cycle ;
\draw  [color={rgb, 255:red, 88; green, 152; blue, 17 }  ,draw opacity=1 ][fill={rgb, 255:red, 88; green, 152; blue, 17 }  ,fill opacity=1 ] (403,100) .. controls (403,98.9) and (403.9,98) .. (405,98) .. controls (406.1,98) and (407,98.9) .. (407,100) .. controls (407,101.1) and (406.1,102) .. (405,102) .. controls (403.9,102) and (403,101.1) .. (403,100) -- cycle ;
\draw  [color={rgb, 255:red, 88; green, 152; blue, 17 }  ,draw opacity=1 ][fill={rgb, 255:red, 88; green, 152; blue, 17 }  ,fill opacity=1 ] (358,130) .. controls (358,128.9) and (358.9,128) .. (360,128) .. controls (361.1,128) and (362,128.9) .. (362,130) .. controls (362,131.1) and (361.1,132) .. (360,132) .. controls (358.9,132) and (358,131.1) .. (358,130) -- cycle ;
\draw  [color={rgb, 255:red, 74; green, 144; blue, 226 }  ,draw opacity=1 ][fill={rgb, 255:red, 74; green, 144; blue, 226 }  ,fill opacity=1 ] (528,99.94) .. controls (528,98.84) and (528.9,97.94) .. (530,97.94) .. controls (531.1,97.94) and (532,98.84) .. (532,99.94) .. controls (532,101.04) and (531.1,101.94) .. (530,101.94) .. controls (528.9,101.94) and (528,101.04) .. (528,99.94) -- cycle ;
\draw  [color={rgb, 255:red, 74; green, 144; blue, 226 }  ,draw opacity=1 ][fill={rgb, 255:red, 74; green, 144; blue, 226 }  ,fill opacity=1 ] (493,100) .. controls (493,98.9) and (493.9,98) .. (495,98) .. controls (496.1,98) and (497,98.9) .. (497,100) .. controls (497,101.1) and (496.1,102) .. (495,102) .. controls (493.9,102) and (493,101.1) .. (493,100) -- cycle ;
\draw  [color={rgb, 255:red, 74; green, 144; blue, 226 }  ,draw opacity=1 ][fill={rgb, 255:red, 74; green, 144; blue, 226 }  ,fill opacity=1 ] (463,129.94) .. controls (463,128.84) and (463.9,127.94) .. (465,127.94) .. controls (466.1,127.94) and (467,128.84) .. (467,129.94) .. controls (467,131.04) and (466.1,131.94) .. (465,131.94) .. controls (463.9,131.94) and (463,131.04) .. (463,129.94) -- cycle ;
\draw  [color={rgb, 255:red, 74; green, 144; blue, 226 }  ,draw opacity=1 ][fill={rgb, 255:red, 74; green, 144; blue, 226 }  ,fill opacity=1 ] (403,130) .. controls (403,128.9) and (403.9,128) .. (405,128) .. controls (406.1,128) and (407,128.9) .. (407,130) .. controls (407,131.1) and (406.1,132) .. (405,132) .. controls (403.9,132) and (403,131.1) .. (403,130) -- cycle ;
\draw  [color={rgb, 255:red, 74; green, 144; blue, 226 }  ,draw opacity=1 ][fill={rgb, 255:red, 74; green, 144; blue, 226 }  ,fill opacity=1 ] (378,130) .. controls (378,128.9) and (378.9,128) .. (380,128) .. controls (381.1,128) and (382,128.9) .. (382,130) .. controls (382,131.1) and (381.1,132) .. (380,132) .. controls (378.9,132) and (378,131.1) .. (378,130) -- cycle ;
\draw  [color={rgb, 255:red, 74; green, 144; blue, 226 }  ,draw opacity=1 ][fill={rgb, 255:red, 74; green, 144; blue, 226 }  ,fill opacity=1 ] (358,160) .. controls (358,158.9) and (358.9,158) .. (360,158) .. controls (361.1,158) and (362,158.9) .. (362,160) .. controls (362,161.1) and (361.1,162) .. (360,162) .. controls (358.9,162) and (358,161.1) .. (358,160) -- cycle ;
\draw  [color={rgb, 255:red, 74; green, 144; blue, 226 }  ,draw opacity=1 ][fill={rgb, 255:red, 74; green, 144; blue, 226 }  ,fill opacity=1 ] (578,70) .. controls (578,68.9) and (578.9,68) .. (580,68) .. controls (581.1,68) and (582,68.9) .. (582,70) .. controls (582,71.1) and (581.1,72) .. (580,72) .. controls (578.9,72) and (578,71.1) .. (578,70) -- cycle ;
\draw  [color={rgb, 255:red, 74; green, 144; blue, 226 }  ,draw opacity=1 ][fill={rgb, 255:red, 74; green, 144; blue, 226 }  ,fill opacity=1 ] (428,100) .. controls (428,98.9) and (428.9,98) .. (430,98) .. controls (431.1,98) and (432,98.9) .. (432,100) .. controls (432,101.1) and (431.1,102) .. (430,102) .. controls (428.9,102) and (428,101.1) .. (428,100) -- cycle ;
\draw    (220,70) -- (268,70) ;
\draw [shift={(270,70)}, rotate = 180] [fill={rgb, 255:red, 0; green, 0; blue, 0 }  ][line width=0.08]  [draw opacity=0] (7.2,-1.8) -- (0,0) -- (7.2,1.8) -- cycle    ;
\draw    (220,100) -- (220,128) ;
\draw [shift={(220,130)}, rotate = 270] [fill={rgb, 255:red, 0; green, 0; blue, 0 }  ][line width=0.08]  [draw opacity=0] (7.2,-1.8) -- (0,0) -- (7.2,1.8) -- cycle    ;
\draw    (220,100) -- (241,100) ;
\draw [shift={(243,100)}, rotate = 180] [fill={rgb, 255:red, 0; green, 0; blue, 0 }  ][line width=0.08]  [draw opacity=0] (7.2,-1.8) -- (0,0) -- (7.2,1.8) -- cycle    ;
\draw    (270,70) -- (270,97.94) ;
\draw [shift={(270,99.94)}, rotate = 270] [fill={rgb, 255:red, 0; green, 0; blue, 0 }  ][line width=0.08]  [draw opacity=0] (7.2,-1.8) -- (0,0) -- (7.2,1.8) -- cycle    ;
\draw    (185,100) -- (185,128) ;
\draw [shift={(185,130)}, rotate = 270] [fill={rgb, 255:red, 0; green, 0; blue, 0 }  ][line width=0.08]  [draw opacity=0] (7.2,-1.8) -- (0,0) -- (7.2,1.8) -- cycle    ;
\draw    (155,70) -- (155,97.94) ;
\draw [shift={(155,99.94)}, rotate = 270] [fill={rgb, 255:red, 0; green, 0; blue, 0 }  ][line width=0.08]  [draw opacity=0] (7.2,-1.8) -- (0,0) -- (7.2,1.8) -- cycle    ;
\draw    (50,70) -- (153,70) ;
\draw [shift={(155,70)}, rotate = 180] [fill={rgb, 255:red, 0; green, 0; blue, 0 }  ][line width=0.08]  [draw opacity=0] (7.2,-1.8) -- (0,0) -- (7.2,1.8) -- cycle    ;
\draw    (50,40) -- (50,68) ;
\draw [shift={(50,70)}, rotate = 270] [fill={rgb, 255:red, 0; green, 0; blue, 0 }  ][line width=0.08]  [draw opacity=0] (7.2,-1.8) -- (0,0) -- (7.2,1.8) -- cycle    ;
\draw [shift={(50,40)}, rotate = 90] [color={rgb, 255:red, 0; green, 0; blue, 0 }  ][fill={rgb, 255:red, 0; green, 0; blue, 0 }  ][line width=0.75]      (0, 0) circle [x radius= 2.01, y radius= 2.01]   ;
\draw    (50,100) -- (93,100) ;
\draw [shift={(95,100)}, rotate = 180] [fill={rgb, 255:red, 0; green, 0; blue, 0 }  ][line width=0.08]  [draw opacity=0] (7.2,-1.8) -- (0,0) -- (7.2,1.8) -- cycle    ;
\draw    (90,100) -- (118,100) ;
\draw [shift={(120,100)}, rotate = 180] [fill={rgb, 255:red, 0; green, 0; blue, 0 }  ][line width=0.08]  [draw opacity=0] (7.2,-1.8) -- (0,0) -- (7.2,1.8) -- cycle    ;
\draw    (50,70) -- (50,98) ;
\draw [shift={(50,100)}, rotate = 270] [fill={rgb, 255:red, 0; green, 0; blue, 0 }  ][line width=0.08]  [draw opacity=0] (7.2,-1.8) -- (0,0) -- (7.2,1.8) -- cycle    ;
\draw    (50,130) -- (68,130) ;
\draw [shift={(70,130)}, rotate = 180] [fill={rgb, 255:red, 0; green, 0; blue, 0 }  ][line width=0.08]  [draw opacity=0] (7.2,-1.8) -- (0,0) -- (7.2,1.8) -- cycle    ;
\draw    (50,100) -- (50,128) ;
\draw [shift={(50,130)}, rotate = 270] [fill={rgb, 255:red, 0; green, 0; blue, 0 }  ][line width=0.08]  [draw opacity=0] (7.2,-1.8) -- (0,0) -- (7.2,1.8) -- cycle    ;
\draw    (155,99.94) -- (183,100) ;
\draw [shift={(185,100)}, rotate = 180.11] [fill={rgb, 255:red, 0; green, 0; blue, 0 }  ][line width=0.08]  [draw opacity=0] (7.2,-1.8) -- (0,0) -- (7.2,1.8) -- cycle    ;
\draw    (95,130) -- (108,130) ;
\draw [shift={(110,130)}, rotate = 180] [fill={rgb, 255:red, 0; green, 0; blue, 0 }  ][line width=0.08]  [draw opacity=0] (7.2,-1.8) -- (0,0) -- (7.2,1.8) -- cycle    ;
\draw    (95,100) -- (95,128) ;
\draw [shift={(95,130)}, rotate = 270] [fill={rgb, 255:red, 0; green, 0; blue, 0 }  ][line width=0.08]  [draw opacity=0] (7.2,-1.8) -- (0,0) -- (7.2,1.8) -- cycle    ;
\draw [shift={(95,100)}, rotate = 90] [color={rgb, 255:red, 0; green, 0; blue, 0 }  ][fill={rgb, 255:red, 0; green, 0; blue, 0 }  ][line width=0.75]      (0, 0) circle [x radius= 2.01, y radius= 2.01]   ;
\draw    (120,100) -- (120,118) ;
\draw [shift={(120,120)}, rotate = 270] [fill={rgb, 255:red, 0; green, 0; blue, 0 }  ][line width=0.08]  [draw opacity=0] (7.2,-1.8) -- (0,0) -- (7.2,1.8) -- cycle    ;
\draw [shift={(120,100)}, rotate = 90] [color={rgb, 255:red, 0; green, 0; blue, 0 }  ][fill={rgb, 255:red, 0; green, 0; blue, 0 }  ][line width=0.75]      (0, 0) circle [x radius= 2.01, y radius= 2.01]   ;
\draw    (50,160) -- (63,160) ;
\draw [shift={(65,160)}, rotate = 180] [fill={rgb, 255:red, 0; green, 0; blue, 0 }  ][line width=0.08]  [draw opacity=0] (7.2,-1.8) -- (0,0) -- (7.2,1.8) -- cycle    ;
\draw    (50,130) -- (50,158) ;
\draw [shift={(50,160)}, rotate = 270] [fill={rgb, 255:red, 0; green, 0; blue, 0 }  ][line width=0.08]  [draw opacity=0] (7.2,-1.8) -- (0,0) -- (7.2,1.8) -- cycle    ;
\draw    (155,70) -- (218,70) ;
\draw [shift={(220,70)}, rotate = 180] [fill={rgb, 255:red, 0; green, 0; blue, 0 }  ][line width=0.08]  [draw opacity=0] (7.2,-1.8) -- (0,0) -- (7.2,1.8) -- cycle    ;
\draw    (155,99.94) -- (155,129.94) ;
\draw [shift={(155,129.94)}, rotate = 90] [color={rgb, 255:red, 0; green, 0; blue, 0 }  ][fill={rgb, 255:red, 0; green, 0; blue, 0 }  ][line width=0.75]      (0, 0) circle [x radius= 2.01, y radius= 2.01]   ;
\draw [shift={(155,99.94)}, rotate = 90] [color={rgb, 255:red, 0; green, 0; blue, 0 }  ][fill={rgb, 255:red, 0; green, 0; blue, 0 }  ][line width=0.75]      (0, 0) circle [x radius= 2.01, y radius= 2.01]   ;
\draw    (155,129.94) -- (168,129.94) ;
\draw [shift={(170,129.94)}, rotate = 180] [fill={rgb, 255:red, 0; green, 0; blue, 0 }  ][line width=0.08]  [draw opacity=0] (7.2,-1.8) -- (0,0) -- (7.2,1.8) -- cycle    ;
\draw    (220,70) -- (220,97.94) ;
\draw [shift={(220,99.94)}, rotate = 270] [fill={rgb, 255:red, 0; green, 0; blue, 0 }  ][line width=0.08]  [draw opacity=0] (7.2,-1.8) -- (0,0) -- (7.2,1.8) -- cycle    ;
\draw    (50,160) -- (50,188) ;
\draw [shift={(50,190)}, rotate = 270] [fill={rgb, 255:red, 0; green, 0; blue, 0 }  ][line width=0.08]  [draw opacity=0] (7.2,-1.8) -- (0,0) -- (7.2,1.8) -- cycle    ;
\draw  [color={rgb, 255:red, 139; green, 87; blue, 42 }  ,draw opacity=1 ][fill={rgb, 255:red, 139; green, 87; blue, 42 }  ,fill opacity=1 ] (48,70) .. controls (48,68.9) and (48.9,68) .. (50,68) .. controls (51.1,68) and (52,68.9) .. (52,70) .. controls (52,71.1) and (51.1,72) .. (50,72) .. controls (48.9,72) and (48,71.1) .. (48,70) -- cycle ;
\draw  [color={rgb, 255:red, 245; green, 166; blue, 35 }  ,draw opacity=1 ][fill={rgb, 255:red, 245; green, 166; blue, 35 }  ,fill opacity=1 ] (48,100) .. controls (48,98.9) and (48.9,98) .. (50,98) .. controls (51.1,98) and (52,98.9) .. (52,100) .. controls (52,101.1) and (51.1,102) .. (50,102) .. controls (48.9,102) and (48,101.1) .. (48,100) -- cycle ;
\draw  [color={rgb, 255:red, 245; green, 166; blue, 35 }  ,draw opacity=1 ][fill={rgb, 255:red, 245; green, 166; blue, 35 }  ,fill opacity=1 ] (153,70) .. controls (153,68.9) and (153.9,68) .. (155,68) .. controls (156.1,68) and (157,68.9) .. (157,70) .. controls (157,71.1) and (156.1,72) .. (155,72) .. controls (153.9,72) and (153,71.1) .. (153,70) -- cycle ;
\draw  [color={rgb, 255:red, 88; green, 152; blue, 17 }  ,draw opacity=1 ][fill={rgb, 255:red, 88; green, 152; blue, 17 }  ,fill opacity=1 ] (218,70) .. controls (218,68.9) and (218.9,68) .. (220,68) .. controls (221.1,68) and (222,68.9) .. (222,70) .. controls (222,71.1) and (221.1,72) .. (220,72) .. controls (218.9,72) and (218,71.1) .. (218,70) -- cycle ;
\draw  [color={rgb, 255:red, 88; green, 152; blue, 17 }  ,draw opacity=1 ][fill={rgb, 255:red, 88; green, 152; blue, 17 }  ,fill opacity=1 ] (153,99.94) .. controls (153,98.84) and (153.9,97.94) .. (155,97.94) .. controls (156.1,97.94) and (157,98.84) .. (157,99.94) .. controls (157,101.04) and (156.1,101.94) .. (155,101.94) .. controls (153.9,101.94) and (153,101.04) .. (153,99.94) -- cycle ;
\draw  [color={rgb, 255:red, 88; green, 152; blue, 17 }  ,draw opacity=1 ][fill={rgb, 255:red, 88; green, 152; blue, 17 }  ,fill opacity=1 ] (93,100) .. controls (93,98.9) and (93.9,98) .. (95,98) .. controls (96.1,98) and (97,98.9) .. (97,100) .. controls (97,101.1) and (96.1,102) .. (95,102) .. controls (93.9,102) and (93,101.1) .. (93,100) -- cycle ;
\draw  [color={rgb, 255:red, 88; green, 152; blue, 17 }  ,draw opacity=1 ][fill={rgb, 255:red, 88; green, 152; blue, 17 }  ,fill opacity=1 ] (48,130) .. controls (48,128.9) and (48.9,128) .. (50,128) .. controls (51.1,128) and (52,128.9) .. (52,130) .. controls (52,131.1) and (51.1,132) .. (50,132) .. controls (48.9,132) and (48,131.1) .. (48,130) -- cycle ;
\draw  [color={rgb, 255:red, 74; green, 144; blue, 226 }  ,draw opacity=1 ][fill={rgb, 255:red, 74; green, 144; blue, 226 }  ,fill opacity=1 ] (218,99.94) .. controls (218,98.84) and (218.9,97.94) .. (220,97.94) .. controls (221.1,97.94) and (222,98.84) .. (222,99.94) .. controls (222,101.04) and (221.1,101.94) .. (220,101.94) .. controls (218.9,101.94) and (218,101.04) .. (218,99.94) -- cycle ;
\draw  [color={rgb, 255:red, 74; green, 144; blue, 226 }  ,draw opacity=1 ][fill={rgb, 255:red, 74; green, 144; blue, 226 }  ,fill opacity=1 ] (183,100) .. controls (183,98.9) and (183.9,98) .. (185,98) .. controls (186.1,98) and (187,98.9) .. (187,100) .. controls (187,101.1) and (186.1,102) .. (185,102) .. controls (183.9,102) and (183,101.1) .. (183,100) -- cycle ;
\draw  [color={rgb, 255:red, 74; green, 144; blue, 226 }  ,draw opacity=1 ][fill={rgb, 255:red, 74; green, 144; blue, 226 }  ,fill opacity=1 ] (153,129.94) .. controls (153,128.84) and (153.9,127.94) .. (155,127.94) .. controls (156.1,127.94) and (157,128.84) .. (157,129.94) .. controls (157,131.04) and (156.1,131.94) .. (155,131.94) .. controls (153.9,131.94) and (153,131.04) .. (153,129.94) -- cycle ;
\draw  [color={rgb, 255:red, 74; green, 144; blue, 226 }  ,draw opacity=1 ][fill={rgb, 255:red, 74; green, 144; blue, 226 }  ,fill opacity=1 ] (93,130) .. controls (93,128.9) and (93.9,128) .. (95,128) .. controls (96.1,128) and (97,128.9) .. (97,130) .. controls (97,131.1) and (96.1,132) .. (95,132) .. controls (93.9,132) and (93,131.1) .. (93,130) -- cycle ;
\draw  [color={rgb, 255:red, 74; green, 144; blue, 226 }  ,draw opacity=1 ][fill={rgb, 255:red, 74; green, 144; blue, 226 }  ,fill opacity=1 ] (68,130) .. controls (68,128.9) and (68.9,128) .. (70,128) .. controls (71.1,128) and (72,128.9) .. (72,130) .. controls (72,131.1) and (71.1,132) .. (70,132) .. controls (68.9,132) and (68,131.1) .. (68,130) -- cycle ;
\draw  [color={rgb, 255:red, 74; green, 144; blue, 226 }  ,draw opacity=1 ][fill={rgb, 255:red, 74; green, 144; blue, 226 }  ,fill opacity=1 ] (48,160) .. controls (48,158.9) and (48.9,158) .. (50,158) .. controls (51.1,158) and (52,158.9) .. (52,160) .. controls (52,161.1) and (51.1,162) .. (50,162) .. controls (48.9,162) and (48,161.1) .. (48,160) -- cycle ;
\draw  [color={rgb, 255:red, 74; green, 144; blue, 226 }  ,draw opacity=1 ][fill={rgb, 255:red, 74; green, 144; blue, 226 }  ,fill opacity=1 ] (268,70) .. controls (268,68.9) and (268.9,68) .. (270,68) .. controls (271.1,68) and (272,68.9) .. (272,70) .. controls (272,71.1) and (271.1,72) .. (270,72) .. controls (268.9,72) and (268,71.1) .. (268,70) -- cycle ;
\draw  [color={rgb, 255:red, 74; green, 144; blue, 226 }  ,draw opacity=1 ][fill={rgb, 255:red, 74; green, 144; blue, 226 }  ,fill opacity=1 ] (118,100) .. controls (118,98.9) and (118.9,98) .. (120,98) .. controls (121.1,98) and (122,98.9) .. (122,100) .. controls (122,101.1) and (121.1,102) .. (120,102) .. controls (118.9,102) and (118,101.1) .. (118,100) -- cycle ;

\draw (419.43,133.4) node [anchor=north west][inner sep=0.75pt]  [font=\small]  {$\cdots $};
\draw (434,113.4) node [anchor=north west][inner sep=0.75pt]  [font=\small]  {$\cdots $};
\draw (381,154.9) node [anchor=north west][inner sep=0.75pt]  [font=\small]  {$\cdots $};
\draw (344,34.4) node [anchor=north west][inner sep=0.75pt]  [font=\small,color={rgb, 255:red, 128; green, 128; blue, 128 }  ,opacity=1 ]  {$\emptyset $};
\draw (345,64.4) node [anchor=north west][inner sep=0.75pt]  [font=\small,color={rgb, 255:red, 139; green, 87; blue, 42 }  ,opacity=1 ]  {$1$};
\draw (340,94.4) node [anchor=north west][inner sep=0.75pt]  [font=\small,color={rgb, 255:red, 245; green, 166; blue, 35 }  ,opacity=1 ]  {$10$};
\draw (332,124.4) node [anchor=north west][inner sep=0.75pt]  [font=\small,color={rgb, 255:red, 88; green, 152; blue, 17 }  ,opacity=1 ]  {$100$};
\draw (457,54.4) node [anchor=north west][inner sep=0.75pt]  [font=\small,color={rgb, 255:red, 245; green, 166; blue, 35 }  ,opacity=1 ]  {$11$};
\draw (389,85.4) node [anchor=north west][inner sep=0.75pt]  [font=\small,color={rgb, 255:red, 88; green, 152; blue, 17 }  ,opacity=1 ]  {$101$};
\draw (420,86.4) node [anchor=north west][inner sep=0.75pt]  [font=\footnotesize,color={rgb, 255:red, 74; green, 144; blue, 226 }  ,opacity=1 ]  {$1010$};
\draw (441,102.4) node [anchor=north west][inner sep=0.75pt]  [font=\small,color={rgb, 255:red, 88; green, 152; blue, 17 }  ,opacity=1 ]  {$110$};
\draw (325,154.4) node [anchor=north west][inner sep=0.75pt]  [font=\small,color={rgb, 255:red, 74; green, 144; blue, 226 }  ,opacity=1 ]  {$1001$};
\draw (484,87.4) node [anchor=north west][inner sep=0.75pt]  [font=\footnotesize,color={rgb, 255:red, 74; green, 144; blue, 226 }  ,opacity=1 ]  {$1100$};
\draw (498.33,113.4) node [anchor=north west][inner sep=0.75pt]  [font=\small]  {$\cdots $};
\draw (518,55) node [anchor=north west][inner sep=0.75pt]  [font=\small,color={rgb, 255:red, 88; green, 152; blue, 17 }  ,opacity=1 ] [align=left] {$\displaystyle 111$};
\draw (367,116.4) node [anchor=north west][inner sep=0.75pt]  [font=\footnotesize,color={rgb, 255:red, 74; green, 144; blue, 226 }  ,opacity=1 ]  {$1000$};
\draw (393,135.4) node [anchor=north west][inner sep=0.75pt]  [font=\footnotesize,color={rgb, 255:red, 74; green, 144; blue, 226 }  ,opacity=1 ]  {$1011$};
\draw (455,135.4) node [anchor=north west][inner sep=0.75pt]  [font=\footnotesize,color={rgb, 255:red, 74; green, 144; blue, 226 }  ,opacity=1 ]  {$1101$};
\draw (533,87.4) node [anchor=north west][inner sep=0.75pt]  [font=\footnotesize,color={rgb, 255:red, 74; green, 144; blue, 226 }  ,opacity=1 ]  {$1111$};
\draw (485,133.4) node [anchor=north west][inner sep=0.75pt]  [font=\small]  {$\cdots $};
\draw (362,173.4) node [anchor=north west][inner sep=0.75pt]  [font=\small]  {$\cdots $};
\draw (565,55) node [anchor=north west][inner sep=0.75pt]  [font=\small,color={rgb, 255:red, 74; green, 144; blue, 226 }  ,opacity=1 ] [align=left] {$\displaystyle 1110$};
\draw (545.33,113.4) node [anchor=north west][inner sep=0.75pt]  [font=\small]  {$\cdots $};
\draw (585.48,75.4) node [anchor=north west][inner sep=0.75pt]  [font=\small]  {$\cdots $};
\draw (362,51.4) node [anchor=north west][inner sep=0.75pt]  [font=\tiny,color={rgb, 255:red, 155; green, 155; blue, 155 }  ,opacity=1 ]  {$1$};
\draw (409,61.4) node [anchor=north west][inner sep=0.75pt]  [font=\tiny,color={rgb, 255:red, 155; green, 155; blue, 155 }  ,opacity=1 ]  {$1$};
\draw (493,61.4) node [anchor=north west][inner sep=0.75pt]  [font=\tiny,color={rgb, 255:red, 155; green, 155; blue, 155 }  ,opacity=1 ]  {$1$};
\draw (362,81.4) node [anchor=north west][inner sep=0.75pt]  [font=\tiny,color={rgb, 255:red, 155; green, 155; blue, 155 }  ,opacity=1 ]  {$0$};
\draw (109.43,133.4) node [anchor=north west][inner sep=0.75pt]  [font=\small]  {$\cdots $};
\draw (124,113.4) node [anchor=north west][inner sep=0.75pt]  [font=\small]  {$\cdots $};
\draw (71,154.9) node [anchor=north west][inner sep=0.75pt]  [font=\small]  {$\cdots $};
\draw (34,34.4) node [anchor=north west][inner sep=0.75pt]  [font=\small,color={rgb, 255:red, 128; green, 128; blue, 128 }  ,opacity=1 ]  {$\emptyset $};
\draw (35,64.4) node [anchor=north west][inner sep=0.75pt]  [font=\small,color={rgb, 255:red, 139; green, 87; blue, 42 }  ,opacity=1 ]  {$1$};
\draw (30,94.4) node [anchor=north west][inner sep=0.75pt]  [font=\small,color={rgb, 255:red, 245; green, 166; blue, 35 }  ,opacity=1 ]  {$11$};
\draw (22,124.4) node [anchor=north west][inner sep=0.75pt]  [font=\small,color={rgb, 255:red, 88; green, 152; blue, 17 }  ,opacity=1 ]  {$111$};
\draw (147,54.4) node [anchor=north west][inner sep=0.75pt]  [font=\small,color={rgb, 255:red, 245; green, 166; blue, 35 }  ,opacity=1 ]  {$10$};
\draw (79,85.4) node [anchor=north west][inner sep=0.75pt]  [font=\small,color={rgb, 255:red, 88; green, 152; blue, 17 }  ,opacity=1 ]  {$110$};
\draw (110,86.4) node [anchor=north west][inner sep=0.75pt]  [font=\footnotesize,color={rgb, 255:red, 74; green, 144; blue, 226 }  ,opacity=1 ]  {$1100$};
\draw (131,102.4) node [anchor=north west][inner sep=0.75pt]  [font=\small,color={rgb, 255:red, 88; green, 152; blue, 17 }  ,opacity=1 ]  {$101$};
\draw (15,154.4) node [anchor=north west][inner sep=0.75pt]  [font=\small,color={rgb, 255:red, 74; green, 144; blue, 226 }  ,opacity=1 ]  {$1111$};
\draw (174,87.4) node [anchor=north west][inner sep=0.75pt]  [font=\footnotesize,color={rgb, 255:red, 74; green, 144; blue, 226 }  ,opacity=1 ]  {$1010$};
\draw (188.33,113.4) node [anchor=north west][inner sep=0.75pt]  [font=\small]  {$\cdots $};
\draw (208,55) node [anchor=north west][inner sep=0.75pt]  [font=\small,color={rgb, 255:red, 88; green, 152; blue, 17 }  ,opacity=1 ] [align=left] {$\displaystyle 100$};
\draw (57,116.4) node [anchor=north west][inner sep=0.75pt]  [font=\footnotesize,color={rgb, 255:red, 74; green, 144; blue, 226 }  ,opacity=1 ]  {$1110$};
\draw (83,135.4) node [anchor=north west][inner sep=0.75pt]  [font=\footnotesize,color={rgb, 255:red, 74; green, 144; blue, 226 }  ,opacity=1 ]  {$1101$};
\draw (145,135.4) node [anchor=north west][inner sep=0.75pt]  [font=\footnotesize,color={rgb, 255:red, 74; green, 144; blue, 226 }  ,opacity=1 ]  {$1011$};
\draw (223,87.4) node [anchor=north west][inner sep=0.75pt]  [font=\footnotesize,color={rgb, 255:red, 74; green, 144; blue, 226 }  ,opacity=1 ]  {$1001$};
\draw (175,133.4) node [anchor=north west][inner sep=0.75pt]  [font=\small]  {$\cdots $};
\draw (52,173.4) node [anchor=north west][inner sep=0.75pt]  [font=\small]  {$\cdots $};
\draw (255,55) node [anchor=north west][inner sep=0.75pt]  [font=\small,color={rgb, 255:red, 74; green, 144; blue, 226 }  ,opacity=1 ] [align=left] {$\displaystyle 1000$};
\draw (235.33,113.4) node [anchor=north west][inner sep=0.75pt]  [font=\small]  {$\cdots $};
\draw (275.48,75.4) node [anchor=north west][inner sep=0.75pt]  [font=\small]  {$\cdots $};
\draw (52,51.4) node [anchor=north west][inner sep=0.75pt]  [font=\tiny,color={rgb, 255:red, 155; green, 155; blue, 155 }  ,opacity=1 ]  {$1$};
\draw (99,61.4) node [anchor=north west][inner sep=0.75pt]  [font=\tiny,color={rgb, 255:red, 155; green, 155; blue, 155 }  ,opacity=1 ]  {$0$};
\draw (183,61.4) node [anchor=north west][inner sep=0.75pt]  [font=\tiny,color={rgb, 255:red, 155; green, 155; blue, 155 }  ,opacity=1 ]  {$0$};
\draw (52,81.4) node [anchor=north west][inner sep=0.75pt]  [font=\tiny,color={rgb, 255:red, 155; green, 155; blue, 155 }  ,opacity=1 ]  {$1$};
\draw (52,111.4) node [anchor=north west][inner sep=0.75pt]  [font=\tiny,color={rgb, 255:red, 155; green, 155; blue, 155 }  ,opacity=1 ]  {$1$};
\draw (362,115.4) node [anchor=north west][inner sep=0.75pt]  [font=\tiny,color={rgb, 255:red, 155; green, 155; blue, 155 }  ,opacity=1 ]  {$0$};
\draw (362,145.4) node [anchor=north west][inner sep=0.75pt]  [font=\tiny,color={rgb, 255:red, 155; green, 155; blue, 155 }  ,opacity=1 ]  {$1$};
\draw (241,61.4) node [anchor=north west][inner sep=0.75pt]  [font=\tiny,color={rgb, 255:red, 155; green, 155; blue, 155 }  ,opacity=1 ]  {$0$};
\draw (67,91.4) node [anchor=north west][inner sep=0.75pt]  [font=\tiny,color={rgb, 255:red, 155; green, 155; blue, 155 }  ,opacity=1 ]  {$0$};
\draw (551,61.4) node [anchor=north west][inner sep=0.75pt]  [font=\tiny,color={rgb, 255:red, 155; green, 155; blue, 155 }  ,opacity=1 ]  {$0$};
\draw (377,91.4) node [anchor=north west][inner sep=0.75pt]  [font=\tiny,color={rgb, 255:red, 155; green, 155; blue, 155 }  ,opacity=1 ]  {$1$};
\draw (52,141.4) node [anchor=north west][inner sep=0.75pt]  [font=\tiny,color={rgb, 255:red, 155; green, 155; blue, 155 }  ,opacity=1 ]  {$1$};
\draw (104,91.4) node [anchor=north west][inner sep=0.75pt]  [font=\tiny,color={rgb, 255:red, 155; green, 155; blue, 155 }  ,opacity=1 ]  {$0$};
\draw (165,91.4) node [anchor=north west][inner sep=0.75pt]  [font=\tiny,color={rgb, 255:red, 155; green, 155; blue, 155 }  ,opacity=1 ]  {$0$};
\draw (157,81.4) node [anchor=north west][inner sep=0.75pt]  [font=\tiny,color={rgb, 255:red, 155; green, 155; blue, 155 }  ,opacity=1 ]  {$1$};
\draw (222,81.4) node [anchor=north west][inner sep=0.75pt]  [font=\tiny,color={rgb, 255:red, 155; green, 155; blue, 155 }  ,opacity=1 ]  {$1$};
\draw (467,81.4) node [anchor=north west][inner sep=0.75pt]  [font=\tiny,color={rgb, 255:red, 155; green, 155; blue, 155 }  ,opacity=1 ]  {$0$};
\draw (532,81.4) node [anchor=north west][inner sep=0.75pt]  [font=\tiny,color={rgb, 255:red, 155; green, 155; blue, 155 }  ,opacity=1 ]  {$1$};
\draw (475,91.4) node [anchor=north west][inner sep=0.75pt]  [font=\tiny,color={rgb, 255:red, 155; green, 155; blue, 155 }  ,opacity=1 ]  {$0$};
\draw (97,111.4) node [anchor=north west][inner sep=0.75pt]  [font=\tiny,color={rgb, 255:red, 155; green, 155; blue, 155 }  ,opacity=1 ]  {$1$};
\draw (157,111.4) node [anchor=north west][inner sep=0.75pt]  [font=\tiny,color={rgb, 255:red, 155; green, 155; blue, 155 }  ,opacity=1 ]  {$1$};
\draw (467,111.4) node [anchor=north west][inner sep=0.75pt]  [font=\tiny,color={rgb, 255:red, 155; green, 155; blue, 155 }  ,opacity=1 ]  {$1$};
\draw (407,111.4) node [anchor=north west][inner sep=0.75pt]  [font=\tiny,color={rgb, 255:red, 155; green, 155; blue, 155 }  ,opacity=1 ]  {$1$};
\draw (414,91.4) node [anchor=north west][inner sep=0.75pt]  [font=\tiny,color={rgb, 255:red, 155; green, 155; blue, 155 }  ,opacity=1 ]  {$0$};

\end{tikzpicture}